\documentclass{article}

\usepackage{fancybox}
\usepackage{algorithmic}
\usepackage[colorinlistoftodos,prependcaption,textsize=footnotesize]{todonotes}
\usepackage[utf8]{inputenc} 
\usepackage{amsfonts}
\usepackage{amsmath}
\usepackage{amsthm} 
\usepackage{amssymb}
\usepackage{color}
\usepackage{graphicx}
\usepackage{setspace}
\usepackage{enumerate}
\usepackage{booktabs}
\usepackage{longtable}
\usepackage{pstricks}
\usepackage{url}
\usepackage{mathtools}
\mathtoolsset{showonlyrefs=true}
\usepackage{cancel}
\usepackage[ruled, boxed, noend, noline, linesnumbered, norelsize]{algorithm2e}
\usepackage[normalem]{ulem}
\usepackage{multicol}
\usepackage{bm}
\usepackage{here}
\usepackage{subcaption}
\usepackage[position=above]{caption}
\usepackage{hyperref}
\usepackage{siunitx}
\usepackage{xfp}
\usepackage{etoolbox}
\usepackage{ifthen}
\usepackage{numprint} 
\usepackage{csvsimple}
\hypersetup{
	colorlinks,
	linkcolor={red!50!black},
	citecolor={blue!50!black},
	urlcolor={blue!80!black}
}
\graphicspath{{./figures/}}
\usepackage{verbatim}

\newcommand{\norm}[1]{\left\|#1\right\|}
\newcommand{\interior}{\mathrm{int}\,}
\newcommand{\reInt}{\mathrm{ri}\,}

\newcommand{\spanVec}{\mathrm{span}\,}
\newcommand{\inProd}[2]{\langle #1 , #2 \rangle }

\newcommand{\PSDcone}[1]{{\mathcal{S}^{#1}_+}}

\newcommand{\stdCone}{ {\mathcal{K}}}

\newcommand{\stdFace}{ \mathcal{F}}

\newcommand{\diag}{{\mathrm{ diag } \,}}

\newcommand{\T}{\top\hspace{-1pt}}               

\renewcommand{\Re}{\mathbb{R}}

\def\dom{{\rm dom}}
\newcommand{\Sd}[1]{{#1}^{\downarrow}}

\newcommand{\orR}[1]{\Re^{#1}_{\downarrow}}
\DeclarePairedDelimiter\abs{\lvert}{\rvert}%
\newcommand{\dist}{ {\mathrm{dist}\,}}
\DeclareMathOperator*{\argmin}{arg\,min}
\newcommand{\mult}{\operatorname{mult}}
\newcommand{\POC}[2]{{\mathcal{K}_{#1}^{#2}}}
\newcommand{\pK}{\POC{p}{n+1}}

\newcommand{\cj}[1]{\stdFace_{#1}^{\Delta}}

\renewcommand{\S}{\mathcal{S}}                    
    
\newcommand{\OPT}[1]{{#1}_{\mathrm{opt}}}
\newcommand{\RNum}[1]{\uppercase\expandafter{\romannumeral #1\relax}}
\newtheorem{definition}{Definition}[section]
\newtheorem{lemma}[definition]{Lemma}
\newtheorem{proposition}[definition]{Proposition}
\newtheorem{example}[definition]{Example}

\newtheorem{corollary}[definition]{Corollary}

\newtheorem{theorem}[definition]{Theorem}
\newtheorem*{remark}{Remark}
\newtheorem{assumption}{Assumption}
\newtheorem{assumptionAlpha}{Assumption}

\newcommand{\COMM}[2]{{
\ifthenelse{\equal{#1}{AT}}{\color{red}}{
\ifthenelse{\equal{#1}{BFL}}{\color{blue}}}
[#1: #2]
}}

\npdecimalsign{.}
\nprounddigits{2}

\newcounter{greyFW}
\newcounter{greyFWEl}
\newcommand{\greyIfOne}[2]{
	\ifnumequal{#1}{1}{{\color{gray}{#2}}}{{#2}}
}

\newenvironment{DIFnomarkup}{}{}

\csvstyle{hyperExp}{
	head to column names,
	separator = comma,
	centered tabular = |c|c|c|c|c|c|c|c|c|c|,
	table head =  \hline\multicolumn{1}{|c|}{}& \multicolumn{2}{c|}{FW} & \multicolumn{2}{c|}{FW EleSym}& \multicolumn{2}{c}{AGM}& \multicolumn{2}{|c|}{AGM EleSym} \\
	\hline  Error& Rel.~time  & S(\%) & Rel.~time & S(\%) & Rel.~time & S(\%)& Rel.~time & S(\%)\\\hline\hline,
	late after line = \\\hline,
	filter expr={ test{\ifdimgreater{\FWsuccess pt}{0.0pt}}
		or test{\ifdimgreater{\FWElsuccess pt}{0.0pt}} }, 
}

\csvstyle{hyperExpDet}{
	head to column names,
	separator = comma,
	centered tabular = |c|c|c|c|c|c|c|,
	table head =  \hline\multicolumn{1}{|c|}{}& \multicolumn{3}{c|}{FW} & \multicolumn{3}{c|}{FW EleSym} \\
	\hline  Error & Iterations& $\norm{x_k-x_{\text{DDS}}}_{\infty}$ & Abs.~time (sec.) & Iterations& $\norm{x_k-x_{\text{DDS}}}_{\infty}$ & Abs.~time (sec.)\\\hline\hline,
	late after line = \\\hline,
	filter expr={ test{\ifdimgreater{\FWsuccess pt}{15.0pt}}
		or test{\ifdimgreater{\FWElsuccess pt}{15.0pt}} }, 
}

\newcommand{\printExpDataLineDet}{\error\%&%
	\setcounter{greyFW}{0}\setcounter{greyFWEl}{0}%
	\ifdimless{\FWsuccess cm}{90.0 cm}%
	{%
		\stepcounter{greyFW}%
	}{}%
	\ifdimless{\FWElsuccess cm}{90.0 cm}{%
			\stepcounter{greyFWEl}%
		}{}%
	\ifcsvstrcmp{\FWsuccess}{0}{- & - & - &}{
		\greyIfOne{\value{greyFW}}{\npfourdigitnosep\nprounddigits{1}{\numprint{\FWit}$\pm$\numprint{\FWitstd}}}& 
		\greyIfOne{\value{greyFW}}{{\FWx}$\pm${\FWxstd}}&
		\greyIfOne{\value{greyFW}}{\FWabs$\pm$\FWabsstd}&%
	}%
	 \ifcsvstrcmp{\FWElsuccess}{0}{- & - & -}{%
	 			  \greyIfOne{\value{greyFWEl}}{\npfourdigitnosep\nprounddigits{1}{\numprint{\FWElit}$\pm$\numprint{\FWElitstd}}}&%
	 			  \greyIfOne{\value{greyFWEl}}{{\FWElx}$\pm$\FWElxstd}&%
	 			  \greyIfOne{\value{greyFWEl}}{{\FWElabs}$\pm${\FWElabsstd}}%
 			     }%
}

\csvstyle{hyperExpcd}{
	head to column names,
	separator = comma,
	centered tabular = |c|c|c|c|c|c|c|c|c|c|,
	table head =  \hline\multicolumn{1}{|c|}{}& \multicolumn{3}{c|}{$\hat c_D$} & \multicolumn{3}{c|}{$2\hat c_D$} & \multicolumn{3}{c|}{$4\hat c_D$} \\
	\hline  Error & Rel.~ time & Iterations & S(\%) & Rel.~ time & Iterations & S(\%)&Rel.~ time & Iterations & S(\%)\\\hline\hline,
	late after line = \\\hline,
	filter expr={ test{\ifdimgreater{\cdiFWElsuccess pt}{1.0pt}}},
}

\newcommand{\printExpDatacd}{\error\%&%
	  \ifcsvstrcmp{\cdiFWElsuccess}{0}{- & - & - &}{%
	 	\nprounddigits{2}\numprint{\cdiFWEl} $\pm$ \numprint{\cdiFWElstd}&
	 	\npfourdigitnosep\nprounddigits{1}{\numprint{\cdiFWElit}$\pm$\numprint{\cdiFWElitstd}}&%
	 	\nprounddigits{1}\numprint{\cdiFWElsuccess}&%
	 }%
	 \ifcsvstrcmp{\cdiiFWElsuccess}{0}{- & - & - &}{%
	 	\nprounddigits{2}\numprint{\cdiiFWEl} $\pm$ \numprint{\cdiiFWElstd}&
	 	\npfourdigitnosep\nprounddigits{1}{\numprint{\cdiiFWElit}$\pm$\numprint{\cdiiFWElitstd}}&%
	 	\nprounddigits{1}\numprint{\cdiiFWElsuccess}&%
	 }%
 	 \ifcsvstrcmp{\cdiiiFWElsuccess}{0}{- & - & - }{%
 		\nprounddigits{2}\numprint{\cdiiiFWEl} $\pm$ \numprint{\cdiiiFWElstd}&
 		\npfourdigitnosep\nprounddigits{1}{\numprint{\cdiiiFWElit}$\pm$\numprint{\cdiiiFWElitstd}}&%
 		\nprounddigits{1}\numprint{\cdiiiFWElsuccess}%
 	}%
}

\csvstyle{hyperExpcdAlt}{
	head to column names,
	separator = comma,
	centered tabular = |c|c|c|c|c|c|c|c|c|c|,
	table head =  \hline\multicolumn{1}{|c|}{}& \multicolumn{3}{c|}{$8\hat c_D$} & \multicolumn{3}{c|}{$16\hat c_D$} & \multicolumn{3}{c|}{$100\hat c_D$} \\
	\hline  Error & Rel.~ time & Iterations & S(\%) & Rel.~ time & Iterations & S(\%)&Rel.~ time & Iterations & S(\%)\\\hline\hline,
	late after line = \\\hline,
	filter expr={ test{\ifdimgreater{\cdiFWElsuccess pt}{1.0pt}}},	
}

\newcommand{\printExpDatacdAlt}{\error\%&%
	\ifcsvstrcmp{\cdivFWElsuccess}{0}{- & - & - &}{%
		\nprounddigits{2}\numprint{\cdivFWEl} $\pm$ \numprint{\cdivFWElstd}&
		\npfourdigitnosep\nprounddigits{1}{\numprint{\cdivFWElit}$\pm$\numprint{\cdivFWElitstd}}&%
		\nprounddigits{1}\numprint{\cdivFWElsuccess}&%
	}%
	\ifcsvstrcmp{\cdvFWElsuccess}{0}{- & - & - &}{%
		\nprounddigits{2}\numprint{\cdvFWEl} $\pm$ \numprint{\cdvFWElstd}&
		\npfourdigitnosep\nprounddigits{1}{\numprint{\cdvFWElit}$\pm$\numprint{\cdvFWElitstd}}&%
		\nprounddigits{1}\numprint{\cdvFWElsuccess}&%
	}%
	\ifcsvstrcmp{\cdviFWElsuccess}{0}{- & - & - }{%
		\nprounddigits{2}\numprint{\cdviFWEl} $\pm$ \numprint{\cdviFWElstd}&
		\npfourdigitnosep\nprounddigits{1}{\numprint{\cdviFWElit}$\pm$\numprint{\cdviFWElitstd}}&%
		\nprounddigits{1}\numprint{\cdviFWElsuccess}%
	}%
}

\DeclareCaptionLabelFormat{continued}{Continued #1\~#2}
\newcounter{boldFW}
\newcounter{boldFWEl}
\newcommand{\boldIfOne}[2]{
	\ifnumequal{#1}{1}{\textbf{#2}}{{#2}}
}
\newcommand{\boldIfLessThan}[2]{
	\ifdimless{#1 cm}{15 cm}{\textbf{#2}}{{#2}}
}
\newcommand{\printExpDataLine}{\error\%& %
	\setcounter{boldFW}{0}\setcounter{boldFWEl}{0}
	\ifdimgreater{\FWsuccess cm}{99.9 cm}%
	{%
		\ifdimless{\FWElsuccess cm}{99.9 cm}{%
			\stepcounter{boldFW}%
		}%
		{%
			\ifdimless{\FW cm}{\FWEl cm}{\stepcounter{boldFW}}{\stepcounter{boldFWEl}}%
		}%
	}%
	{%
		\ifdimless{\FWElsuccess cm}{99.9 cm}{%
		}%
		{%
			\stepcounter{boldFWEl}%
		}	%
	}%
	\ifcsvstrcmp{\FWsuccess}{0}{- & 0 &}{
		\boldIfOne{\value{boldFW}}{\nprounddigits{2}\numprint{\FW} $\pm$ \numprint{\FWstd}} & \nprounddigits{1}\numprint{\FWsuccess}&}
	\ifcsvstrcmp{\FWElsuccess}{0}{- & 0 &}{
		\boldIfOne{\value{boldFWEl}}{\nprounddigits{2}\numprint{\FWEl} $\pm$ \numprint{\FWElstd}} & \nprounddigits{1}\numprint{\FWElsuccess}&}
	\ifcsvstrcmp{\AGMsuccess}{0}{- & 0 &}{
		\nprounddigits{2}\numprint{\AGM} $\pm$ \numprint{\AGMstd} & \nprounddigits{1}\numprint{\AGMsuccess}&}		
	\ifcsvstrcmp{\AGMElsuccess}{0}{- & 0 }{
		\nprounddigits{2}\numprint{\AGMEl} $\pm$ \numprint{\AGMElstd} & \nprounddigits{1}\numprint{\AGMElsuccess}}%
}

\csvstyle{pconeExp}{
	centered tabular=|c|c|c|c|c|c|c|c|c|,
	table head =  \hline\multicolumn{1}{|c|}{}& \multicolumn{2}{c|}{$p=1.1$} & \multicolumn{2}{c|}{$p=1.3$}& \multicolumn{2}{c}{$p=3$}& \multicolumn{2}{|c|}{$p=5$} \\
	\hline  Error& Rel.~time & S(\%) & Rel.~time & S(\%) & Rel.~time & S(\%)& Rel.~time & S(\%)\\\hline\hline,
	late after line=\\\hline,
	column names = {1=\error,
		2=\FWix,3=\FWixstd,4=\FWiit,5=\FWiitstd,6=\FWiabs,7=\FWiabsstd,
		8=\FWi,9=\FWistd,10=\FWisuccess,
		12=\FWiix,13=\FWiixstd,14=\FWiiit,15=\FWiiitstd,16=\FWiiabs,17=\FWiiabsstd,
		18=\FWii,19=\FWiistd,20=\FWiisuccess,
		22=\FWiiix,23=\FWiiixstd,24=\FWiiiit,25=\FWiiiitstd,26=\FWiiiabs,27=\FWiiiabsstd,
		28=\FWiii,29=\FWiiistd,30=\FWiiisuccess,
		32=\FWivx,33=\FWivxstd,34=\FWivit,35=\FWivitstd,36=\FWivabs,37=\FWivabsstd,
		38=\FWiv,39=\FWivstd,40=\FWivsuccess},
}
\newcommand{\printpConeExpDataLine}{\error\% & 
	\ifcsvstrcmp{\FWisuccess}{0}{- & 0 &}{
		\boldIfLessThan{\FWi}{\nprounddigits{2}\numprint{\FWi} $\pm$ \numprint{\FWistd}} & \nprounddigits{1}\numprint{\FWisuccess}&}
	\ifcsvstrcmp{\FWiisuccess}{0}{- & 0 &}{
		\boldIfLessThan{\FWii}{\nprounddigits{2}\numprint{\FWii} $\pm$ \numprint{\FWiistd}} & \nprounddigits{1}\numprint{\FWiisuccess}&} 
	\ifcsvstrcmp{\FWiiisuccess}{0}{- & 0 &}{
		\boldIfLessThan{\FWiii}{\nprounddigits{2}\numprint{\FWiii} $\pm$ \numprint{\FWiiistd}} & \nprounddigits{1}\numprint{\FWiiisuccess}&}
	\ifcsvstrcmp{\FWivsuccess}{0}{- & 0 }{
		\boldIfLessThan{\FWiv}{\nprounddigits{2}\numprint{\FWiv} $\pm$ \numprint{\FWivstd}} & \nprounddigits{1}\numprint{\FWivsuccess}}
}

\csvstyle{pconeExpDet}{
	centered tabular=|c|c|c|c|c|c|c|c|c|,
	table head =  \hline\multicolumn{1}{|c|}{}& \multicolumn{3}{c|}{$p=1.1$} & \multicolumn{3}{c|}{$p=1.3$}\\
	\hline  Error& Iterations & $\norm{x_k - x_{\text{Mos}}}_\infty$ & Abs.~time (sec.) &
	Iterations & $\norm{x_k - x_{\text{Mos}}}_\infty$ & Abs.~time (sec.) \\\hline\hline,
	late after line=\\\hline,
	column names = {1=\error,
		2=\FWix,3=\FWixstd,4=\FWiit,5=\FWiitstd,6=\FWiabs,7=\FWiabsstd,
		8=\FWi,9=\FWistd,10=\FWisuccess,
		12=\FWiix,13=\FWiixstd,14=\FWiiit,15=\FWiiitstd,16=\FWiiabs,17=\FWiiabsstd,
		18=\FWii,19=\FWiistd,20=\FWiisuccess,
		22=\FWiiix,23=\FWiiixstd,24=\FWiiiit,25=\FWiiiitstd,26=\FWiiiabs,27=\FWiiiabsstd,
		28=\FWiii,29=\FWiiistd,30=\FWiiisuccess,
		32=\FWivx,33=\FWivxstd,34=\FWivit,35=\FWivitstd,36=\FWivabs,37=\FWivabsstd,
		38=\FWiv,39=\FWivstd,40=\FWivsuccess},
}
\newcommand{\printpConeExpDataLineDet}{\error\% & %
	\setcounter{greyFW}{0}\setcounter{greyFWEl}{0}%
	\ifdimless{\FWisuccess cm}{90.0 cm}%
	{
		\stepcounter{greyFW}%
	}{}%
	\ifdimless{\FWiisuccess cm}{90.0 cm}{%
		\stepcounter{greyFWEl}%
	}{}%
	\ifcsvstrcmp{\FWisuccess}{0}{- & - & - &}{
		\greyIfOne{\value{greyFW}}{\npfourdigitnosep\nprounddigits{1}{\numprint{\FWiit}$\pm$\numprint{\FWiitstd}}}& 
		\greyIfOne{\value{greyFW}}{{\FWix}$\pm${\FWixstd}}&
		\greyIfOne{\value{greyFW}}{\FWiabs$\pm$\FWiabsstd}& 
	}
	\ifcsvstrcmp{\FWiisuccess}{0}{- & - & -}{
		\greyIfOne{\value{greyFWEl}}{\npfourdigitnosep\nprounddigits{1}{\numprint{\FWiiit}$\pm$\numprint{\FWiiitstd}}} & 
		\greyIfOne{\value{greyFWEl}}{{\FWiix}$\pm${\FWiixstd}}&
		\greyIfOne{\value{greyFWEl}}{\FWiiabs$\pm$\FWiiabsstd}%
	}	 
}

\csvstyle{pconeExpDetAlt}{
	centered tabular=|c|c|c|c|c|c|c|c|c|,
	table head =  \hline\multicolumn{1}{|c|}{}& \multicolumn{3}{c|}{$p=3$} & \multicolumn{3}{c|}{$p=5$}\\
	\hline  Error& Iterations & $\norm{x_k - x_{\text{Mos}}}_\infty$ & Abs.~time (sec.) &
	Iterations & $\norm{x_k - x_{\text{Mos}}}_\infty$ & Abs.~time (sec.) \\\hline\hline,
	late after line=\\\hline,
	column names = {1=\error,
		2=\FWix,3=\FWixstd,4=\FWiit,5=\FWiitstd,6=\FWiabs,7=\FWiabsstd,
		8=\FWi,9=\FWistd,10=\FWisuccess,
		12=\FWiix,13=\FWiixstd,14=\FWiiit,15=\FWiiitstd,16=\FWiiabs,17=\FWiiabsstd,
		18=\FWii,19=\FWiistd,20=\FWiisuccess,
		22=\FWiiix,23=\FWiiixstd,24=\FWiiiit,25=\FWiiiitstd,26=\FWiiiabs,27=\FWiiiabsstd,
		28=\FWiii,29=\FWiiistd,30=\FWiiisuccess,
		32=\FWivx,33=\FWivxstd,34=\FWivit,35=\FWivitstd,36=\FWivabs,37=\FWivabsstd,
		38=\FWiv,39=\FWivstd,40=\FWivsuccess},
}
\newcommand{\printpConeExpDataLineDetAlt}{\error\% &%
	\setcounter{greyFW}{0}\setcounter{greyFWEl}{0}%
	\ifdimless{\FWiiisuccess cm}{90.0 cm}%
	{
		\stepcounter{greyFW}%
	}{}%
	\ifdimless{\FWivsuccess cm}{90.0 cm}{%
		\stepcounter{greyFWEl}%
	}{}%
	\ifcsvstrcmp{\FWiiisuccess}{0}{- & - & - &}{%
		\greyIfOne{\value{greyFW}}{\npfourdigitnosep\nprounddigits{1}{\numprint{\FWiiiit}$\pm$\numprint{\FWiiiitstd}}}& 
		\greyIfOne{\value{greyFW}}{{\FWiiix}$\pm${\FWiiixstd}}&
		\greyIfOne{\value{greyFW}}{\FWiiiabs$\pm$\FWiiiabsstd}&%
	}
	\ifcsvstrcmp{\FWivsuccess}{0}{- & - & - }{%
		\greyIfOne{\value{greyFWEl}}{\npfourdigitnosep\nprounddigits{1}{\numprint{\FWivit}$\pm$\numprint{\FWivitstd}}}& 
		\greyIfOne{\value{greyFWEl}}{{\FWix}$\pm${\FWivxstd}}&
		\greyIfOne{\value{greyFWEl}}{\FWiabs$\pm$\FWivabsstd}%
	} 
}

\numberwithin{equation}{section}

\begin{document}
\title{Projection onto hyperbolicity cones and beyond: a dual Frank-Wolfe approach}
\author{
	Takayuki Nagano\thanks{
	Department of Mathematical Informatics, Graduate School of Information Science and Technology, University of Tokyo,  7-3-1 Hongo, Bunkyo-ku, Tokyo 113-8656, Japan.
	}
	\and
	Bruno F. Louren\c{c}o%
	\thanks{Department of Fundamental Statistical Mathematics, Institute of Statistical Mathematics, 10-3 Midori-cho, Tachikawa, Tokyo 190-8562, Japan.
		This author was supported partly by the JSPS Grant-in-Aid for Early-Career Scientists  23K16844 and by the Japan Science and Technology Agency (JST) as part of Adopting Sustainable Partnerships for Innovative Research Ecosystem (ASPIRE), Grant Number JPMJAP2520.
		(\texttt{bruno@ism.ac.jp})}
	\and 
	Akiko Takeda\thanks{
		Department of Mathematical Informatics, Graduate School of Information Science and Technology,
	  University of Tokyo,  7-3-1 Hongo, Bunkyo-ku, Tokyo 113-8656, Japan and
          Center  for  Advanced  Intelligence  Project, RIKEN,  1-4-1, Nihonbashi, Chuo-ku, Tokyo  103-0027, Japan.
          This was author was supported partly by the JSPS Grant-in-Aid  	
          for Scientific Research (B)  23H03351 and JST ERATO Grant Number JPMJER1903. (\texttt{takeda@mist.i.u-tokyo.ac.jp}) 
	}
}
\maketitle
\begin{abstract}
We discuss the problem of projecting a point onto an arbitrary hyperbolicity cone from both theoretical and numerical perspectives. 
While hyperbolicity cones are furnished with a generalization of the notion of eigenvalues, obtaining closed form expressions for the projection operator as in the case of semidefinite matrices is an elusive endeavour.
To address that we propose a Frank-Wolfe method to handle this task and, more generally, strongly convex optimization over closed convex cones. One of our innovations is that the Frank-Wolfe method is actually applied to the dual problem and, by doing so, subproblems can be solved in closed-form using minimum eigenvalue functions and conjugate vectors. 
To test the validity of our proposed approach, we present numerical experiments where we check  the performance of alternative approaches including interior point methods and an earlier accelerated gradient method proposed by Renegar. 
We also show numerical examples where the hyperbolic polynomial has millions of monomials. 
Finally, we also discuss the problem of projecting onto $p$-cones which, although not hyperbolicity cones in general, are still amenable to our techniques.  
\end{abstract}

\section{Introduction}\label{sec:int}
Hyperbolicity cones \cite{gaarding1959inequality,guler1997hyperbolic,BGLS01,renegar2004hyperbolic} are a far-reaching family of closed convex cones containing all symmetric cones and all polyhedral cones. 
In particular, the cone  of $n\times n$ real symmetric positive semidefinite matrices $\PSDcone{n}$ is a hyperbolicity cone. 
One distinctive feature of $\PSDcone{n}$ is that the orthogonal projection onto $\PSDcone{n}$ has a well-known expression that can be described in terms of the spectral decomposition of a matrix. 
Analogously, for a hyperbolicity cone there is a natural notion of \emph{eigenvalues} (see Section~\ref{sec:hyp_p}) that is strong enough to allow the extension of certain linear algebraic results about symmetric matrices, e.g., \cite{BGLS01,renegar2004hyperbolic}. 
With this in mind, the starting point of this project was the following questions:
\begin{quote}
	Given an arbitrary hyperbolicity cone $\Lambda \subseteq \Re^n$ and $x \in \Re^n$, how to compute the projection of $x$ onto $\Lambda$ efficiently, i.e., how to compute
	$P_\Lambda(x) \coloneqq \argmin_{y \in \Lambda } \norm{y-x}$?
	Are there closed form expressions in terms of the eigenvalues of $x$ ?
\end{quote}
In a nutshell our answers are as follows. We only found a closed form expression for the projection operator for a rather narrow family of hyperbolicity cones (see Section~\ref{sec:projection} and Proposition~\ref{prop:proj}). For more general cones, numerical methods seem necessary, which we discuss in detail in this paper in Sections~\ref{sec:algorithm} and \ref{sec:numerical_experiment}. 

Before we move on, we take a step back to motivate these questions. 
Besides being an interesting geometric question on its own, the usefulness of having a readily available projection operator for some convex set is well-documented in optimization (e.g., see \cite{HM11}). 
It is, after all, a basic requirement for the applicability of several algorithms. 
For example, augmented Lagrangian methods for solving conic optimization problems typically require the projection operator onto the underlying cone or its dual, e.g., see \cite[Section~2]{SS04}. 
The same is true for operator splitting approaches, e.g., see \cite[Section~3]{OCPB16} and \cite[Section~3.2]{BNZ24} which discuss specifically the case of conic constraints.

Not only that, methods such as \emph{cyclic projections} and others (e.g., see \cite{BB96}) can be used to refine the feasibility properties of a solution obtained by a numerical solver. 
All of this is, of course, contingent on either having a ``reasonable'' closed form solution for the projection operator or a fast numerical method.

In the case of an arbitrary hyperbolicity cone, the fact that we have a relatively powerful notion of eigenvalues  gives some hope of an analogue of \eqref{eq:psd_proj}. 
Unfortunately, even though we have \emph{eigenvalues}, we do not have a suitable generalization of the notion of \emph{spectral decomposition} that is always available for an arbitrary hyperbolicity cone. Nevertheless, in this paper we will present several partial results on this front regarding the computation of \emph{distance functions to hyperbolicity cones}.


As for numerical approaches, we propose a Frank-Wolfe based method for computing the projection operator onto a hyperbolicity cone $\Lambda$. However, developing a Frank-Wolfe based approach successfully  has its own challenges. For example, the subproblems appearing during the Frank-Wolfe iteration should either have closed form solutions or be efficiently solvable. 
In addition, it is typically required that the feasible region be compact, 
which is not true for the problem of projecting a point onto a convex cone.

In this work, we show that it is possible to overcome \emph{all} these difficulties in the case of hyperbolicity cones and we will discuss a \emph{dual} Frank-Wolfe method for solving the projection problem over a hyperbolicity cone and beyond. 
Our approach is dual in the sense that the Frank-Wolfe algorithm is actually applied to the Fenchel dual of our problem of interest. 
This is because, surprisingly, solving the problem from the dual side leads to subproblems that have closed form solutions in terms of the underlying hyperbolic polynomial. 

Although our focus will be on the hyperbolicity cone case, 
the method we discuss in this paper is actually capable of solving a larger class of problems as follows:
\begin{equation}\label{prob:main}
\begin{array}{ll}
\displaystyle
\min_{x \in \mathbb{R}^n} &f(x) \\
\mathrm{s.t. } & Tx+b \in \stdCone
\end{array}
\end{equation}
where $f:\mathbb{R}^n \rightarrow \mathbb{R}$ is a closed proper $\mu$-strongly convex function,
$T:\mathbb{R}^n \rightarrow \mathbb{R}^m$ is a linear map and $\stdCone$ is
a full-dimensional pointed closed convex cone. 
In this case, we are able to show that the subproblems that appear in the Frank-Wolfe algorithm can be expressed in terms of \emph{generalized eigenvalue functions}.
The problem \eqref{prob:main} contains as a special case  the problem of computing the projection of an arbitrary point onto $\stdCone$, since for fixed $x_0 \in \Re^n$, we can take $f(x) \coloneqq \norm{x-x_0}^2$, let $T$ be the identity map and $b \coloneqq 0$.

We now summarize the main contributions of this paper.
\begin{itemize}
	\item We provide a few theoretical results on the projection operator and the distance function to hyperbolicity cones. In particular, for the so-called \emph{isometric} hyperbolic polynomials, there are formulae  analogous to the ones that hold for the positive semidefinite cone, see Propositions~\ref{prop:dist_hyperbolicity} and \ref{prop:projection}.
	\item We propose a dual Frank-Wolfe method for solving \eqref{prob:main}, which includes the particular case of projecting onto a hyperbolicity cones. One of our main results is that the solution to the subproblems appearing in our method can be expressed in terms of generalized minimal eigenvalue computations and conjugate vectors, see Theorem~\ref{theo:FW_sub_opt}. 
	In the particular case of hyperbolicity cones, we show how conjugate vectors can easily be obtained from the underlying hyperbolic polynomial, see Proposition~\ref{prop:normal_vector}.
	We then provide several convergence results in Section~\ref{sec:convergence}. 
	We emphasize that since the Frank-Wolfe method is applied to a dual problem of \eqref{prob:main}, it is still necessary to bridge the gap between the dual and primal problems. With this issue in mind, we provide some convergence results from the primal side, see Theorems~\ref{theo:xk} and \ref{thm:convergence_rate}. 
	We also provide a discussion on practical issues one may find when implementing our approach, see Section~\ref{sec:practical}.
	\item We provide an implementation of our algorithm and numerical experiments in Section~\ref{sec:numerical_experiment}. Taking interior point methods as a baseline, we compare against an earlier algorithm proposed by Renegar for hyperbolicity cones \cite{renegar2019accelerated}.  
	We also show that our implementation is capable of handling polynomials with millions of monomials, provided that the underlying computational algebra is carefully implemented, see Section~\ref{sec:num_mil}. At the end, we also have numerical experiments for non-hyperbolicity cones, see Section~\ref{sec:experiment_proj_p}.

\end{itemize}

\subsection{Related works}
In this brief subsection, we review some key works on optimization aspects of hyperbolicity cones.
G\"{u}ler wrote a pioneering work on hyperbolic polynomials and interior point methods (IPMs) \cite{guler1997hyperbolic}.
Nowadays, there are a few IPM-based generic conic solvers that are capable of handling hyperbolicity cone constraints, such as  DDS \cite{karimi2019domain}, Hypatia \cite{coey2022solving} and alfonso \cite{papp2022alfonso}. 
For Hypatia and alfonso, it is possible to use their ``generic conic interface'' to implement optimization over hyperbolicity cones.
DDS, on the other hand, has specific functionalities tailored for hyperbolicity cones. 

In any case, the problem of finding the orthogonal projection onto a hyperbolicity cone $\stdCone$ can be naturally formulated as a conic linear program over the direct product between $\stdCone$ and an additional second-order cone constraint, e.g., see \eqref{eq:hypcone2cone}.
Therefore, finding the projection can be, in theory, done with one of those solvers.

Regarding first-order methods, Renegar proposed a first-order algorithm 
for conic linear programs over hyperbolicity cones which uses smoothing and accelerating techniques \cite{renegar2019accelerated}. 

In Section~\ref{sec:numerical_experiment}, we present numerical experiments in order to compare the performance of different approaches for particular cases of the problem in \eqref{prob:main}.
Although we defer a detailed discussion to Section~\ref{sec:numerical_experiment}, we will see that our proposed algorithm (which is a first-order method) is quite competitive in comparison with the aforementioned approaches. 

\paragraph{On the generalized Lax conjecture}
A cone is said to be \emph{spectrahedral} if it is linearly isomorphic to 
an intersection of the form $\PSDcone{n}\cap L$, where $L \subseteq \S^n$ is subspace and $\S^n$ denotes the space of $n\times n$ real symmetric matrices. 
Every spectrahedral cone is a hyperbolicity cone and the \emph{generalized Lax conjecture} is precisely the question of whether the converse also holds.
Under the generalized Lax conjecture, a conic linear program over a hyperbolicity cone (i.e., a hyperbolic program) can be reformulated as a single semidefinite program, see also the 
discussion in \cite[Corollary~2]{MM23}.

Given $X \in \S^n$, it is not clear how the spectral decomposition of $X$ is related to the orthogonal projection of $X$ onto $\PSDcone{n}\cap L$, when $L$ is an arbitrary subspace of $\S^n$.
Because of that, even if, say, the generalized Lax conjecture turns out to be true, it is still makes sense to investigate numerical methods for computing projections onto hyperbolicity cones. 
Another issue is that even when a spectrahedral representation is available, it may lift the problem to a larger ambient space compared to its original formulation, which may be less efficient computationally. This happens, for example, when expressing a second-order cone constraint as a semidefinite constraint.

\subsection{Outline of this work}
In Section~\ref{sec:preliminaries}, we recall some necessary definitions from convex analysis, hyperbolic polynomials and the Frank-Wolfe method.
In Section~\ref{sec:projection} we prove a theoretical discussion on the distance function and the projection operator onto an isometric hyperbolicity cone.
In Section~\ref{sec:algorithm}, we propose and analyze a  first-order dual algorithm to optimize strongly convex functions over regular cones based on a Frank-Wolfe method. 
In Section~\ref{sec:numerical_experiment}, we show the results of numerical experiments. We compare our algorithm with Renegar's method described in \cite{renegar2019accelerated} and
the DDS \cite{karimi2019domain} solver. As our proposed method is also applicable to more general cones, we also include numerical experiments for the problem of projecting onto $p$-cones and a comparison with Mosek \cite{MC2020}. 
Section~\ref{sec:conc} concludes this paper.

\section{Preliminaries}\label{sec:preliminaries}
We start with notations and basic definitions.
Given an element $u \in \Re^n$, we will denote its $i$-th component 
by $u_i$. We use $\bm{1}_n$ to denote the $n$-dimensional vector whose components are all equal to $1$.
We write $\orR{n}$ for the cone of elements $u \in \Re^n$ satisfying 
$u_1 \geq \cdots \geq u_n$. Let $u \in \Re^n$, we denote by $\Sd{u}$ the element in $\Re^n_{\downarrow}$  
corresponding to a reordering of the coordinates of $u$ in such a way 
that
\[
\Sd{u}_1 \geq \cdots \geq \Sd{u}_n.
\]
 We write $\Re^n_+$ for the nonnegative orthant, i.e., the elements $u \in \Re^n$ such that $u_i \geq 0$ for every $i$.

For a convex subset $S \subseteq \Re^n$, we denote its indicator function, recession cone, interior and relative interior by $\delta_S$, $0^+S$, $\interior S$ and $\reInt S$, respectively. %
Additionally, we suppose $\Re^n$ is endowed with an inner product $\inProd{\cdot}{\cdot}$ and a corresponding induced norm $\norm{\cdot}$.
With that, we denote by $S^\perp$ the set of elements orthogonal to $S$.
For a convex cone $\stdCone \subseteq \Re^n$, we define its dual cone as
\[ \stdCone^* \coloneqq  \{ x \in \Re^n \mid \forall y \in \stdCone, \:\:\langle x ,y \rangle \geq 0 \}. \] 
A cone $\stdCone$ is said to be \emph{pointed} if $\stdCone \cap - \stdCone = \{0\}$ and \emph{full-dimensional} if its interior is non-empty.  A full-dimensional pointed cone is said to be \emph{regular}.

Two elements satisfying $x \in \stdCone, y \in \stdCone^*$  and $\inProd{x}{y} = 0$ are said to be \emph{conjugate}. 
For $x \in \stdCone$, we denote the set of elements conjugate to $x$ by $\cj{x}$ so that
\begin{equation}\label{eq:cj_def}
\cj{x} \coloneqq  \left\lbrace y \in \stdCone^* \mid \inProd{y}{x} = 0 \right\rbrace = \stdCone^*\cap\{x\}^\perp.
\end{equation}
The reason for this notation is that if $\stdFace_x$ denotes the unique face of $\stdCone$ satisfying $x \in \reInt \stdFace_x$, then $\cj{x}$ as defined in \eqref{eq:cj_def} is exactly the conjugate face to $\stdFace_x$, i.e., $\cj{x} = \stdCone^* \cap \stdFace_x^\perp$ holds. For more details on faces of cones, see \cite{Ba81,Pa00}.

For a closed convex function $f:\Re^n \to \Re\cup \{\infty\}$, we denote its conjugate function by $f^*$, which is defined as
\begin{equation*}
  f^*(s) \coloneqq \sup_{x\in \mathrm{dom}f} \{\langle s, x\rangle -f(x)\},
\end{equation*}
where $\mathrm{dom}f = \{x \in \Re^n \mid f(x) < \infty\}$.
The subdifferential of $f$ at $x \in \Re^n$ is denoted by $\partial f(x)$.

For a linear map $T:\Re^n\to \Re^m$, we denote by $T^*$ the adjoint map of $T$. We denote the operator norm of $T$ by $\|T\|_{\mathrm{op}}$, which is defined as
\begin{equation*}
\|T\|_{\mathrm{op}} \coloneqq \sup_{x \neq 0} \frac{\|Tx\|}{\|x\|},
\end{equation*}
where by a slight abuse of notation we use the same symbol $\| \cdot \|$ to indicate the underlying norm in $\Re^n$ and $\Re^m$.

A differentiable function $f:\Re^n\to \Re$ is called \emph{$L$-smooth} if $\nabla f$ is Lipschitz continuous with constant $L>0$, i.e.,
\begin{equation*}
\norm{\nabla f(x)-\nabla f(y)} \leq L\norm{x-y}, \qquad \forall x,y \in \Re^n.
\end{equation*}
For a $\mu>0$, $f$ is called \emph{$\mu$-strongly convex} if for every $\theta \in [0,1]$ we have
\begin{equation*} f(\theta x +(1-\theta)y)
  \leq \theta f(x) + (1-\theta)f(y) -\frac{\mu}{2}\theta(1-\theta)\norm{x-y}^2, \qquad \forall x,y \in \Re^n.
\end{equation*}
Finally, we recall that the conjugate of a $\mu$-strongly convex function is $1/\mu$-smooth, see \cite[Theorem~4.2.1]{HUL96_II}.

\subsection{Hyperbolic polynomials}\label{sec:hyp_p}

Let $p: \Re^n \to \Re$ be a homogeneous polynomial, we say that 
\emph{$p$ is hyperbolic along the direction $e$} if and only if 
$p(e) \neq 0$ and for every $x \in \Re^n$, the univariate polynomial 
\[
t \mapsto p(x -te)
\]
has only real roots. Here we summarize some basic properties of hyperbolic polynomials that will be necessary in the subsequent sections. For more details, see \cite{gaarding1959inequality,guler1997hyperbolic,renegar2004hyperbolic}.

Suppose that $p$ has degree $d$. 
We define the map $\lambda: \Re^n \to \orR{d}$ that maps $x \in \Re^n$ to the $d$ roots of the polynomial $p(x -te)$,
ordered from largest to smallest. That is, we have 
\begin{align*}
p(x-te) & = p(e)\prod _{i=1}^d(\lambda_i(x)-t)\\
& \lambda_1(x) \geq \cdots \geq \lambda_d(x).
\end{align*}
In analogy to classical linear algebra, we will say that $\lambda_1(x),\ldots, \lambda_d(x)$ are the \emph{eigenvalues of $x$}. We also write $\lambda_d(x)$ as $\lambda_{\min}(x)$ to emphasize that $\lambda_d(x)$ is the smallest eigenvalue.
Then, the \emph{hyperbolicity cone of $p$ along the direction $e$} is
the closed convex cone $\Lambda(p,e)$  given by 
\[
\Lambda(p,e) = \{x \in \Re^n \mid \lambda_{\min}(x) \geq 0 \},
\]
see Section~2 in \cite{renegar2004hyperbolic}. If $p$ and $e$ are clear from the context, we write $\Lambda(p,e)$ as $\Lambda$. 
If $\hat e \in \interior \Lambda(p,e)$, then 
$\Lambda(p,e) = \Lambda(p,\hat e)$ holds, see \cite[Theorem~3]{renegar2004hyperbolic}.

For $u \in \Re^{d}$ we denote by $\lambda^{-1}(u)$ the inverse image of $\{u\}$ by $\lambda$, so that 
\begin{equation}\label{eq:lamb_inv}
\lambda^{-1}(u) \coloneqq \{x \in \Re^n \mid \lambda(x) = u\}.
\end{equation}
Also, for a subset $S \subseteq \Re^n$ we have 
\begin{equation}\label{eq:lamb_S}
\lambda(S) \coloneqq \{\lambda(x) \mid x \in S\}. 
\end{equation}
Since we are assuming that the eigenvalues are ordered, we always have $\lambda(S) \subseteq \Re^d_{\downarrow}$.


For $x \in \Re^n$, we denote by $\mult(x)$ the number of zero eigenvalues of $x$. That is, $\mult(x)$ is the multiplicity of zero as a root of $t \mapsto p(x-te)$.

A hyperbolic polynomial $p$ is said to be \emph{complete} if and only if \[\{x \in \Re^n \mid \lambda(x) = 0 \} = \{x \in \Re^n \mid \mult(x) = d\} =\{0\}.\]
This happens if and only if $\Lambda$ is \emph{pointed}, see \cite[Proposition~11]{renegar2004hyperbolic}.

Before we move on, we briefly relate the discussion so far to classical linear algebra. 
The eigenvalues given by some hyperbolic polynomial $p$ are  roots of univariate polynomials that are analogous to characteristic polynomials of real symmetric matrices.
In particular, as will be discussed later in Example~\ref{ex:isom}, the determinant polynomial  $\det: \S^n \to \Re$ over the real $n\times n$ symmetric matrices is a hyperbolic polynomial along the identity matrix. 
In this case, the eigenvalues of $x \in \S^n$ as a symmetric matrix coincide with the eigenvalues of $x$ in the sense described in this section. 

That said, the analogy to symmetric matrices is not perfect because while we have a suitable generalization of the notion of characteristic polynomial, there are no obvious analogues to eigenvectors and spectral decompositions when $p$ is an arbitrary hyperbolic polynomial.
In \cite[Section~6]{BGLS98_pre}, there is a discussion on an analogue of the spectral decomposition for hyperbolic polynomials, but it does not apply to arbitrary $p$ in view of \cite[Theorem~6.5]{BGLS98_pre}. 
We will revisit this point in Section~\ref{sec:proj_op}.

\paragraph{Derivative relaxations}
Let $D_{e}p$ denote the directional derivative of $p$ along $e$, so that\[
D_{e}p(x) = \lim _{t \to 0} \frac{p(x+te) -p(x)}{t}, \quad \forall x \in \Re^n.
\]
Then, the function $D_ep : \Re^n \to \Re$ is also a hyperbolic polynomial along $e$. The hyperbolicity cone associated with $(D_e p,e)$,
is called \emph{the derivative cone} (of $p$ along $e$) and is denoted by $\Lambda'$.


We write $D_e^i p$ for the higher order derivatives, so that $D_e^i p (x) = \frac{d^ip(x+te)}{dt^i}|_{t=0}$. Then, 
we define $p^{(0)}\coloneqq p$, $p^{(1)} \coloneqq D_e^1 p, \ldots, p^{(d)} \coloneqq D_e^d p$. Taking the derivative repeatedly gives a sequence of hyperbolic polynomials
and associated hyperbolicity cones
\begin{equation*}
  \Lambda \subseteq \Lambda^{(1)} \subseteq \dots \subseteq \Lambda^{(d-1)},
\end{equation*}
where $\Lambda^{(i)} \coloneqq \Lambda(p^{(i)},e)$.

Finally, we need the following property of hyperbolicity cones.
\begin{theorem}{\cite[Theorem 12]{renegar2004hyperbolic}}
  \label{thm:mult_derivative}
  Let $\Lambda$ be a hyperbolicity cone. Define $\partial^r \Lambda = \{ x \in \Lambda \mid \mathrm{mult}(x) = r \}$.
  Then, for $r \geq 2$,
  \begin{equation*}
      \partial^r \Lambda^{(1)} = \partial^{r+1} \Lambda.
  \end{equation*}
  Also,
  \begin{equation*}
  \partial^1\Lambda^{(1)} \cap \Lambda  = \partial^2 \Lambda.
  \end{equation*}
  \end{theorem}

\subsection{Generalized minimum eigenvalue functions}\label{sec:gmin}
In this subsection we discuss a generalization of the minimum eigenvalue function that is applicable to arbitrary regular cones.
As we will see shortly, when specialized to hyperbolicity cones, this generalization coincides with the definition given in Section~\ref{sec:hyp_p}.
Let $\stdCone \subseteq \Re^n$ be a regular (i.e., full-dimensional and pointed) closed convex cone. 
Let $e \in \interior \stdCone$ be fixed, then the \emph{minimum eigenvalue function $\lambda_{\min}:\Re^n \to \Re$ with respect to $\stdCone$ and $e$} is defined as:
\begin{equation}\label{eq:min_eig_def}
\lambda_{\min}(x) \coloneqq \sup\,\{t \mid x - te \in \stdCone\}.
\end{equation}
First we observe that $\lambda_{\min}(x) = \inf \, \{t \mid x - te \not\in \stdCone\}$ holds and this was the original definition considered in \cite[Section~2]{Renegar16}, see also \cite[Section~2.5]{L17} and \cite[Section~2]{IL17}.
Given that there is a dependency on $e$ and $\stdCone$, it might be appropriate to use some  notation similar to ``${\lambda^{\stdCone,e}_{\min}}$'', but since there will be no ambiguity regarding the chosen $\stdCone$ and $e$, we will use the simpler $\lambda_{\min}$.
We recall the following basic properties of $\lambda_{\min}$.
\begin{proposition}\label{prop:eig_p}
Let $\stdCone \subseteq \Re^n$ be a regular closed convex cone and $e \in \interior \stdCone$ be fixed. Let $\lambda_{\min}$ be as in \eqref{eq:min_eig_def}. For $x \in \Re^n$, the following items hold.
\begin{enumerate}[$(i)$]
	\item $\lambda_{\min}(x)$ is finite.
	\item $x \in \stdCone \Longleftrightarrow \lambda_{\min}(x) \geq 0$.
	\item $x \in \interior \stdCone \Longleftrightarrow \lambda_{\min}(x) > 0$.
\end{enumerate}
\end{proposition}
\begin{proof}
Since $e$ is an interior point of $\stdCone$, there exists $u > 0$ such 
that $e + ux \in \stdCone$ and since $\stdCone$ is a cone, 
$x + e/u \in \stdCone$ holds as well. This shows that $\lambda_{\min}(x) > -\infty$. 
On the other hand if $\lambda_{\min}(x)  = \infty$, then $x-te$ belongs to $\stdCone$ for all $t$. This would imply that $-e \in \stdCone$ (e.g., see \cite[Theorem~8.3]{rockafellar}), which would contradict the pointedness of $\stdCone$. This shows that item~$(i)$ holds.

For item~$(ii)$, if $x \in \stdCone$, then $t = 0$ is a feasible solution to \eqref{eq:min_eig_def}, so indeed $\lambda_{\min}(x) \geq 0$. Conversely, 
if  $\lambda_{\min}(x) \geq 0$, then for every $k > 0$, there exists $t_{k}$ satisfying
$\lambda_{\min}(x)  \geq t_{k} > \lambda_{\min}(x)-\frac{1}{k} \geq - \frac{1}{k}$ and $x-t_{k}e \in \stdCone$. Passing to a convergent subsequence if necessary, we may assume that $t_k$ converges to some $\bar{t} \geq 0$. 
Since $\stdCone$ is closed, we have $x-\bar{t}e \in \stdCone$. 
This implies that $x = (x-\bar{t}e) + (\bar{t}e) \in \stdCone$.

The proof of item~$(iii)$ is similar so we omit it.  
\end{proof}

Now, suppose that $\stdCone = \Lambda(p,e)$ is a regular hyperbolicity cone. In the following proposition we observe that $\lambda_{\min}$ as defined in Section~\ref{sec:hyp_p} and in \eqref{eq:min_eig_def} coincide.
\begin{proposition}\label{prop:hyp_min_eig}
Let $\Lambda(p,e) \subseteq \Re^n$ be a regular hyperbolicity cone and suppose that the degree of $p$ is $d$. Then, for every $x \in \Re^n$, $\lambda_{d}(x) = \sup\{t \mid x - te \in \Lambda(p,e) \}$, where $\lambda_{d}(x)$ is the smallest root of 
$t \mapsto p(x - te)$.
\end{proposition}
\begin{proof}
	Let $x \in \Re^n$ and suppose that $p$ has degree $d$. 	
For any $t \in \Re$, we observe that  $\gamma$ is a root of $s\mapsto p(x-es)$ if 
and only if $\gamma-t$ is a root of $s \mapsto p(x-t e -es)$.
Therefore, the eigenvalue map (as in Section~\ref{sec:hyp_p}) satisfies 
 \[\lambda(x - t e) = \lambda(x) -t \textbf{1}, \qquad \forall t \in \Re,
 \] 
 where $\textbf{1} \in \Re^d$ is the vector where each component is $1$.
 Recalling that $x - t e \in \Lambda(p,e)$ if and only if 
 $\lambda(x-te) \geq 0$, we see 
 that the condition ``$x - te \in \Lambda(p,e)$'' implies the component-wise inequality 
 \[
 \lambda(x) \geq t\textbf{1}.
 \]
 In particular, the maximum value that $t$ can assume under the constraint ``$x - te \in \Lambda(p,e)$'' is $\lambda_d(x)$, where $\lambda_d(x)$ is the smallest root of $t \mapsto p(x - te)$.
	Conversely, if $t \coloneqq \lambda_{d}(x)$, 
	we have  $\lambda_{d}(x - \lambda_{d}(x)e)= 0$, so 
	$x - \lambda_{d}(x) e \in \Lambda(p,e)$.
	That is,  $\lambda_{d}(x)$ is the optimal solution to the maximization problem in \eqref{eq:min_eig_def}.
\end{proof}

%

\subsection{The Frank-Wolfe Method}
In this subsection, we review some of the basic aspects of 
the Frank-Wolfe method (FW method) proposed by Frank and Wolfe \cite{frank1956algorithm}, which is also known as conditional gradient method \cite[Chapter~3]{demyanov1970approximate}, see also \cite{FG14,bomze2021frank,Po23}.
Originally Frank and Wolfe proposed the algorithm to optimize a quadratic function over a polyhedral set, but the FW method is applicable to
the following more general problem:
\begin{equation}\label{eq:fw_prob}
   \min_{x \in C} f(x), 
\end{equation}
where $C$ is a convex and compact set in $\mathbb{R}^n$ and $f$ is a differentiable and $L$-smooth function in $C$.
There are several variants on FW methods \cite{wolfe1970convergence,holloway1974extension,lacoste2015global,PRS16}, but here we 
only make use of the simplest version, see Algorithm~\ref{alg:FW_method}, 
which follows the description in the survey \cite{bomze2021frank}.

\begin{figure}[htbp]
  \begin{algorithm}[H]
    \caption{The Frank-Wolfe method \cite{frank1956algorithm}}
    \label{alg:FW_method}
    \begin{algorithmic}[1]
      \STATE Choose a point $x_0 \in C$
      \FOR{$k = 0,1,\dots$}
        \STATE If $x_k$ satisfies some stopping criterion, STOP
        \STATE Compute $s_k \in \argmin_{x \in C} \langle \nabla f(x_k), x \rangle $
        \STATE Set $d_k \coloneqq s_k - x_k$
        \STATE Choose step size $\alpha_k \in (0,1] $
        \STATE Set $x_{k+1} \coloneqq x_k + \alpha_k d_k$
      \ENDFOR
      \end{algorithmic}
  \end{algorithm}
\end{figure}
One important aspect of the FW method is that it does not require the projection operator onto $C$. Instead, we assume the availability of a linear optimization oracle over $C$ which is capable of solving the subproblem $\min_{x \in C} \langle \nabla f(x_k), x \rangle $ appearing in line 4 of Algorithm~\ref{alg:FW_method}.
A successful application of the FW  method thus depends on having a 
fast way to solve the underlying subproblem. 
Fortunately there are many such problems, which have contributed for the recent renewed interest in FW methods in optimization, machine learning and even extensions to nonconvex problems, e.g., \cite{jaggi2013revisiting,LT13,BPT20,ZZLP21}.


Another useful feature of FW methods is that there is an easily computable measure of convergence called the Frank-Wolfe gap (the FW gap).
The FW gap at $x$ is denoted by $G(x)$ and is defined as
\begin{equation}
  \label{eq:FW_gap}
  G(x) \coloneqq \max_{s \in C} \langle -\nabla f(x), s-x\rangle.
\end{equation}
Let $\OPT{x}$ denote an optimal solution to \eqref{eq:fw_prob}. Since 
$f$ is convex, for $x \in C$ we have $\inProd{-\nabla f(x)} {\OPT{x}-x}\geq f(x) -f(\OPT{x}) \geq 0$. We also recall that $x \in C$ is optimal if and only if 
$-\nabla f(x)$ belongs to the normal cone of $C$ at $x$.
Consequently, the FW gap has the following properties for $x \in C$, 
see also \cite[Section~5.1]{bomze2021frank}:
\begin{itemize}
  \item $G(x)$ is always nonnegative and equal to $0$ if and only if $x$ is optimal.
  \item $G(x) \geq  f(x) -f(\OPT{x})$ holds.
  \item $G(x_k) = \langle -\nabla f(x_k), d_k\rangle$ for the $k$-th iterate in Algorithm~\ref{alg:FW_method}.
\end{itemize}
In particular, under convexity, $G(x)$ is an upper bound to the optimality gap $f(x) -f(\OPT{x})$.
Also, $G(x_k)$ can be calculated easily at each iteration. 
Because of these properties, $G(x_k)$ is often used as a stopping criterion.

We end this section with some known convergence result regarding Algorithm~\ref{alg:FW_method}.
\begin{theorem}\label{thm:convergence_FW}
  If $f$ is convex and step size $\alpha_k$ is given by one of the following rules:
  \begin{itemize}
    \item diminishing step size rule \cite[Theorem~1]{jaggi2013revisiting}: $ \alpha_k \coloneqq \frac{2}{k+2}$
    \item exact line search \cite[Theorem~3.1]{dunn1980convergence}: $\alpha_k \coloneqq \argmin_{\alpha\in[0,1]} f(x_k+\alpha d_k)$
    \item Lipschitz constant dependent step size \cite[Theorem~6.1]{levitin1966constrained}: $\alpha_k \coloneqq \min\left\lbrace -\frac{\langle \nabla f(x_k), d_k \rangle}{L\norm{d_k}^2} , 1 \right\rbrace$,  where $L$ is a Lipschitz constant of $\nabla f$.
  \end{itemize}
  Then  the sequence $\{x_k\}$ generated by Algorithm~\ref{alg:FW_method} satisfies
  \begin{equation*}
  f(x_k) - f_{\mathrm{opt}} = O(1/k).
  \end{equation*}
\end{theorem}
Finally, we mention in passing that there are also convergence results about 
the FW gap, e.g., see \cite{jaggi2013revisiting}.

\section{Distance function and projection under the norm induced by
	a hyperbolic polynomial}
\label{sec:projection}
We assume throughout this section that $p:\Re^n \to \Re$ is a complete hyperbolic polynomial along $e \in \Re^n$ of degree $d$. 
We will denote the corresponding hyperbolicity cone $\Lambda(p,e) \subseteq \Re^n$ simply by $\Lambda$. 
The results in this section are applicable to the so-called 
\emph{isometric} hyperbolic polynomials which were initially considered in \cite{BGLS01} and are defined as follows.
\begin{definition}[Isometric hyperbolic polynomial, {\cite[Definition~5.1]{BGLS01}}]\label{def:isometric}
	A hyperbolic polynomial $p$ is isometric if and only if for all $y, z \in \Re^n$, there exists $x \in \Re^n$ satisfying
	\[
	\lambda(x) = \lambda(z) \textrm{ and } \lambda(x+y) = \lambda(x) + \lambda(y).
	\]
\end{definition}
In general, given $y,z \in \Re^n$, $\lambda(y+z)$ does not necessarily coincide with $\lambda(y) + \lambda(z)$. 
When $p$ is isometric, it is always possible to find some $x \in \Re^n$ that has the same eigenvalues of $z$ but it is ``compatible'' with $y$ in the sense that $\lambda(x+y) = \lambda(z) + \lambda(y) = \lambda(x) + \lambda(y)$.
This implies, for example, that the image of $\Re^n$ by $\lambda(\cdot)$ is convex, which is a property that may fail for arbitrary hyperbolic polynomials, see \cite[Example~5.2]{BGLS01}.

Another important caveat about the results discussed in this section is
that they are considered with respect to a certain norm $\norm{\cdot}_p$ and inner product $\inProd{\cdot}{\cdot}_p$ derived from $p$ as follows
\begin{align}
\norm{x}_p &\coloneqq \sqrt{\lambda_1(x)^2 + \cdots + \lambda _d(x)^2} = \norm{\lambda(x)}_2 \label{eq:norm_poly}\\
\inProd{x}{y}_p & \coloneqq \frac{1}{4}\norm{x+y}_p^2 - \frac{1}{4}\norm{x-y}_p^2, \notag
\end{align}
where $\norm{\lambda(x)}_2$ indicates the usual Euclidean $2$-norm of $\lambda(x)$ in $\Re^d$.
We also assume that $\Re^d$ is equipped with the usual Euclidean inner product denoted by $\inProd{\cdot}{\cdot}_2$, so that $\norm{\lambda(x)}_2 = \sqrt{\inProd{\lambda(x)}{\lambda(x)}_2}$ holds.

Under the assumption that $p$ is complete, $\inProd{\cdot}{\cdot}_p$ is indeed an inner product, see \cite[Theorem~4.2]{BGLS01}.
With that, if $S \subseteq \Re^n$ is a subset, we have
\[
\dist(x, S) \coloneqq \inf_{y \in S } \norm{x-y}_p.
\]
Sometimes we also write $\dist(S,x)$, which is equal to  $\dist(x,S)$.
For convenience, for a singleton set, we define $\dist(x,{y}) \coloneqq \dist(x,\{y\}) = \norm{x-y}_p$.
For a closed convex $S \subseteq \Re^n$ we have
\[P_{S}(x) \coloneqq \argmin_{y \in S } \norm{x-y}_p.\]
For subsets of $\Re^d$ we define distance functions and projections analogously except that we use the Euclidean norm $\norm{\cdot}_2$ instead. 

Here are some examples of isometric hyperbolic polynomials and the corresponding inner products and norms.
\begin{example}[The hyperbolic norm $\norm{\cdot}_p$, the nonnegative orthant, the semidefinite cone and symmetric cones]\label{ex:isom}
The distance function depends on the choice of norm. On the other hand, membership on the hyperbolicity cone depends on the eigenvalues of a given element. 
Therefore, if we hope to get a closed formula for the distance function to a hyperbolicity cone it seems reasonable to demand that the norm and the eigenvalues are connected somehow.

The norm in \eqref{eq:norm_poly} that was introduced in \cite{BGLS01} is one way to enforce this connection.
Although seemingly arbitrary, it reduces to some well-known norms for certain important cones as discussed next.
Furthermore, $\norm{\cdot}_p$ is also connected to the theory of interior point methods. Recalling that $F\coloneqq-\log(p)$ is a logarithmically homogeneous self-concordant barrier for $\Lambda(p,e)$ \cite[Theorem~4.1]{guler1997hyperbolic}, its Hessian at a point in the interior of  $\Lambda(p,e)$ is positive definite and induce an inner product that is important in the analysis of interior point methods, see \cite[Theorem~7]{guler1997hyperbolic}. It turns out that 
\[
\norm{x}_p^2 = D^2F(e)[x,x],
\]
where $D^2F(e)[x,x] = \frac{d^2}{dt^2}F(e+tx)$, see \cite[Remark~4.3]{BGLS01}. 
That is, $\norm{\cdot}_p^2$  corresponds to the quadratic form used to compute the Dikin ellipsoid at $e$.

The nonnegative orthant $\Re^n_+$ can be realized as a hyperbolicity cone by taking $p : \Re^n \to \Re$ to be the polynomial $p(x) \coloneqq \prod_{i=1}^n x_i$ and $e \coloneqq (1,\ldots,1)$. 
With that, the induced inner product and norm are the usual Euclidean ones.
A proof that $p$ is isometric is given in \cite[Section~7]{BGLS98_pre}.

Let $\mathcal{S}^n$ be the space of $n\times n$ real symmetric matrices, and $\mathcal{S}^n_+$ be the cone of $n\times n$ real symmetric positive semidefinite matrices. For a matter of notational consistency, in what follows we will use lower case letters to denote matrices.
Then, letting $e$ denote the $n\times n$ identity matrix, the polynomial $\det$ on $\mathcal{S}^n$ is hyperbolic  with respect to $e$ and $\Lambda(\det, e) = \mathcal{S}^n_+$ holds.
The eigenvalue map $\lambda(\cdot):\S^n \to \Re^n$ corresponds to the usual eigenvalues of a symmetric matrix in decreasing order.
The norm induced by $\det$ is the Frobenius norm, and 
the induced inner product is  the usual trace inner product.

We present a short argument that $\det$ is isometric, see also   \cite[Section~7]{BGLS98_pre}.
Let $y,z \in \S^n$ and consider the spectral decomposition  of $y$ so that $y = u\diag(\lambda(y))u^\T$ holds, where $u$ is an $n\times n$ orthogonal matrix, $\diag(\lambda(y))$ is the diagonal matrix associated to $\lambda(y)$ and $u^\T$ is the transpose of $u$. 
Then, letting $x \coloneqq u\diag(\lambda(z))u^\T$, we have $\lambda(x) = \lambda(z)$ and $x+y = u\diag(\lambda(x)+\lambda(y))u^\T$ so that $\lambda(x+y) = \lambda(x) + \lambda(y)$ holds.

More generally, every symmetric (i.e., homogeneous self-dual) cone can be realized as a hyperbolicity cone induced by some isometric hyperbolic polynomial,  see \cite[Section~2.2]{SS08} for more details.


\end{example}

Before we move on, we recall that if $X$ is an $n\times n$ symmetric matrix,  $v_{1},\ldots v_{n}$ are $n$ orthonormal eigenvectors of $X$ and $\lambda_1,\ldots, \lambda_n$ are the corresponding eigenvalues, then the projection onto $\PSDcone{n}$ with respect to the Frobenius norm is the result of ``zeroing the negative eigenvalues in the spectral decomposition'':
\begin{equation}\label{eq:psd_proj}
P_{\PSDcone{n}}(X) = \argmin _{Y \in \PSDcone{n}} \norm{Y-X}_F = \sum _{i=1}^n \max(0,\lambda_i)v_{i}v_{i}^\T,
\end{equation}
where $\norm{\cdot}_F$ is the Frobenius norm.
This implies that the distance function with respect the Frobenius norm is given 
by
\begin{equation}\label{eq:psd_dist}
\dist(x,\PSDcone{n})^2 = \sum_{i=1}^n \min(\lambda_i,0)^2.
\end{equation}
In the remainder of this section, by imposing increasingly more restrictive assumptions, we will discuss potential analogues of \eqref{eq:psd_proj} and \eqref{eq:psd_dist} for the case of the hyperbolicity cones.

\subsection{The distance function to a hyperbolicity cone}

Let $x \in \Re^n$. Because $x$ belongs to $\Lambda$ if and only if 
$\lambda(x)  \in \Re^d_+$, in analogy to \eqref{eq:psd_dist}, a reasonable guess is that $\dist(x,\Lambda)$ satisfies
\begin{equation}\label{eq:hyper_guess}
\dist(x,\Lambda)^2 = \sum _{i=1}^d \min(\lambda_i(x),0)^2,
\end{equation}
which does indeed hold in quite a few cases, for example, when $\Lambda$ is the cone of $d\times d$ positive semidefinite real matrices. 
Unfortunately, we will see that proving this expression seems to require extra assumptions on $p$.


Recalling the definitions in \eqref{eq:lamb_inv} and \eqref{eq:lamb_S}, a general formula for $\dist(x,\Lambda)$ can be obtained by making use of the following lemma, see also {\cite[Proposition~5.3]{BBEGG10}}. 
\begin{lemma}\label{lem:iso}
	If a complete hyperbolic polynomial is isometric then
$\dist(u,\lambda(x)) = \dist(\lambda^{-1}(u),x)$ holds for all 
$u \in \lambda(\Re^n)$ and all $x \in \Re^n$.
\end{lemma}
\begin{proof}
See Appendix~\ref{app:lem_iso}.
\end{proof}

\begin{proposition}[The distance to an isometric hyperbolicity cone]\label{prop:dist_hyperbolicity}
	Let $\Lambda = \Lambda(p,e)$ be a hyperbolicity cone, $p$ a complete isometric hyperbolic polynomial and let $x \in \Re^n$. 
	We have
	\[
	\dist(x, \Lambda) = \inf _{u \in  \lambda(\Lambda)  }\norm{\lambda(x)-u}_2.
	\]	
	In particular, if $\lambda(\Lambda) = \Re^d_+ \cap \orR{d}$ holds, then 
	$
	\dist(x, \Lambda)^2 = \sum _{i=1}^d \min(\lambda_i(x),0)^2.
	$
\end{proposition}
\begin{proof}
First, we note that that $\Lambda$ can be written as an union of sets that correspond to elements that have the same eigenvalues, i.e., 
$\Lambda = \bigcup _{u \in \lambda(\Lambda)} \lambda^{-1}(u)$. With that, we have
\begin{align}
\dist(x, \Lambda) = \min _{y \in \Lambda} \dist(x,y) & = \min_{u \in \lambda(\Lambda)}\left[ \min _{y \in \lambda^{-1}(u)} \dist(y,x) \right]\notag\\
& = \min_{u \in \lambda(\Lambda)}\left[ \dist(\lambda^{-1}(u),x) \right] \notag\\
& = \min _{u \in \lambda(\Lambda)} \dist(u,\lambda(x)), \label{eq:iso}
\end{align}	
where the last equality follows from Lemma~\ref{lem:iso}. This shows the first half of the proposition. If additionally $\lambda(\Lambda) = \Re^d_+ \cap \orR{d}$ holds, we have
\begin{equation}\label{eq:iso1}
\min _{u \in \lambda(\Lambda)} \dist(u,\lambda(x)) = \min _{u \in \Re^d_+ \cap \orR{d}}\norm{\lambda(x)-u}_2 \geq \min _{u \in \Re^d_+}\norm{\lambda(x)-u}_2 = \norm{\lambda(x) - \lambda(x)^+}_2,
\end{equation}
where $\lambda(x)^+$ is the projection of $\lambda(x)$ onto $\Re^d_+$, which  is obtained by zeroing the negative components of $\lambda(x)$. Since  $\lambda(x)\in \orR{d}$ holds, we also have
$\lambda(x)^+ \in \orR{d}$. We conclude that
\begin{equation}\label{eq:iso2}
\norm{\lambda(x) - \lambda(x)^+}_2 \geq \min _{u \in \Re^d_+ \cap \orR{d}}\norm{\lambda(x)-u}_2.
\end{equation}
From \eqref{eq:iso1} and \eqref{eq:iso2} we have
$\norm{\lambda(x) - \lambda(x)^+}_2 = \min _{u \in \lambda(\Lambda)} \dist(u,\lambda(x))$. This, together with \eqref{eq:iso}, leads to $\dist(x, \Lambda) = \norm{\lambda(x) - \lambda(x)^+}_2 $ which implies $\dist(x, \Lambda)^2 = \sum _{i=1}^d \min(\lambda_i(x),0)^2.
$	
\end{proof}

We now take a look at the requirement that $\lambda(\Lambda)$  coincides with $\Re^d_+ \cap \orR{d}$. By construction, $\lambda(\Lambda)$ is always
contained in $\Re^d_+ \cap \orR{d}$, so the nontrivial part is the opposite containment. 
As we will see in the next lemma, the condition $\lambda(\Lambda) = \Re^d_+ \cap \orR{d}$ is equivalent to $\lambda(\Re^n) = \Re^d_{\downarrow}$ when $p$ is isometric.
The significance of the equality $\lambda(\Re^n) = \Re^d_{\downarrow}$ is that it makes it possible to prove certain results on the Fenchel conjugate of spectral functions, see 
\cite[Theorem~5.5]{BGLS01}.
The equality $\lambda(\Re^n) = \Re^d_{\downarrow}$ also indicates that every element in $\Re^d_{\downarrow}$ is an eigenvalue vector of some $x \in \Re^n$.
\begin{lemma}\label{lem:dist_con}
Suppose that $p:\Re^n \to \Re$ is a complete isometric hyperbolic polynomial.
Then, $\lambda(\Lambda) = \Re^d_+ \cap \orR{d}$ holds if and only if 
$\lambda(\Re^n) = \orR{d}$.
\end{lemma}
\begin{proof}
Suppose $\lambda(\Re^n) = \orR{d}$. As we have observed,
 $\lambda(\Lambda) \subseteq \Re^d_+ \cap \orR{d}$ always holds, so 
 let us prove the opposite containment. If $u \in \Re^d_+ \cap \orR{d}$,
 since $\lambda(\Re^n) = \orR{d}$, there exists $x \in \Re^n$ such 
 that $\lambda(x) = u$. Since $u \in \Re^d_+$, such an $x$ must belong 
 to $\Lambda$.
 
Next, suppose that $\lambda(\Lambda) = \Re^d_+ \cap \orR{d}$. Since $\lambda(\Re^n) \subseteq \orR{d}$ always holds, let us prove the 
opposite containment. Let $u \in \orR{d}$. Then, $u$ can be written as the sum of two elements in $\orR{d}$ corresponding to 
the positive and negative components of $u$, i.e.,
\[
u = u^+  + u^-.
\]
Since $u^+ \in \Re^d_+$ as well, by hypothesis, there exists 
$y \in \Lambda$ such that $\lambda(y) = u^+$.
Similarly, there is  $z \in \Lambda$ such that 
$\lambda(z) = (-u^-)^\downarrow$.
Since the  eigenvalues of $-z$ are the 
negatives of the eigenvalues of $z$, we have
$\lambda(-z) = u^-$.

The polynomial $p$ is assumed to be isometric,  so there exists 
$x$ such that $\lambda(x) = \lambda(-z)$ and 
$\lambda(x+y) = \lambda(x) + \lambda(y)$.
That is, 
we have
\[
\lambda(x+y) = \lambda(x) + \lambda(y) = u^- + u^+ = u.
\]
\end{proof}

\subsection{Partial results on the projection operator}\label{sec:proj_op}
Having discussed distance functions, we now take a look 
at how to actually project an arbitrary point onto a hyperbolicity cone. 
If $\Lambda$ is the cone of positive semidefinite matrices, computing 
the projection is easy: we just compute a spectral decomposition of $x \in \S^n$ and zero its negative eigenvalues.
However, when $\Lambda$ is an arbitrary hyperbolicity cone, the analogy to positive semidefinite matrices does not seem to take us very far. A difficulty is that, although we have a generalized notion of eigenvalues, there is no obvious notion of ``eigenvectors''. Similarly, no obvious notion of spectral decomposition exists when $p$ is an arbitrary hyperbolic polynomial.

That said, if $p$ is isometric and each eigenvalue of $\lambda_i(x)$ is simple, i.e., a root of multiplicity one for the polynomial univariate polynomial $t \mapsto p(x-te)$, then the projection of $x$ onto $\Lambda$ can be represented in closed form. 
We use results from \cite{BGLS01}  in our proof.

\begin{proposition}
  \label{prop:projection}
  Let $\Lambda = \Lambda(p,e)$ be a  hyperbolicity cone and suppose that $p$ is complete, isometric and $\lambda(\Re^n) = \orR{d}$ holds.
  Define $f:\mathbb{R}^d \rightarrow \mathbb{R}$ as $f(x) \coloneqq \frac{1}{2}\sum_{i=1}^d \max \{x_i,0\}^2$.
  Then, $f\circ \lambda$ is convex differentiable, and for all $x\in\Re^n$,
  \begin{equation*}
    P_\Lambda(x) = \nabla (f\circ \lambda)(x).
  \end{equation*}
\end{proposition}
\begin{proof}
	Under the stated assumptions, \cite[Theorem~3.9]{BGLS01} implies convexity of $f \circ \lambda$ and \cite[Theorem 5.5]{BGLS01} leads to the following equivalence.
  \begin{equation}\label{eq:Bauschke_symmetric}
    y \in \partial (f\circ \lambda)(x) \iff \lambda(y) \in \partial f(\lambda(x)) \:\:\text{and}\:\:
    \langle x,y\rangle_p = \langle \lambda(x),\lambda(y)\rangle_2.
  \end{equation}
  Because  $\norm{\lambda(x)}_2 = \norm{x}_p$ holds, expanding $\norm{x-y}_p^2$ we also obtain the following equivalence
  \begin{equation}\label{eq:property}
    \langle x,y\rangle_p = \langle \lambda(x),\lambda(y)\rangle_2 \iff \| \lambda(x)-\lambda(y)\|_2 = \norm{x-y}_p.
  \end{equation}  
  By definition, $f$ is differentiable and $\nabla f$ is the projection operator onto $\Re^n_+$, i.e.
  \begin{equation}\label{eq:proj_onto_nonnegative}
    \nabla f = P_{\Re_+^d}.
  \end{equation}
  By \eqref{eq:property} and \eqref{eq:proj_onto_nonnegative}, \eqref{eq:Bauschke_symmetric} can be rewritten as
  \begin{equation}\label{eq:proj_sub}
    y \in \partial (f\circ \lambda)(x) \iff \lambda(y) = P_{\Re_+^d}(\lambda(x)) \:\:\text{and}\:\:
    \| \lambda(x)-\lambda(y)\|_2 = \|x-y\|_p.
  \end{equation}
Let $y \in \partial (f\circ \lambda)(x)$, then in view of \eqref{eq:proj_sub} we have $\norm{x-y}_p^2 = \sum _{i=1}^d \min(0,\lambda_i(x))^2$. Since $\lambda(y) \in \Re^d_+$, we have $y \in \Lambda$, so by Lemma~\ref{lem:dist_con} and Proposition~\ref{prop:dist_hyperbolicity}, $y$ is precisely the projection of $x$ onto $\Lambda$.


%
Since the projection onto a closed convex set is unique, the subdifferential $\partial (f\circ \lambda)(x)$ is a singleton. 
This means that $f\circ \lambda$ is differentiable at $x$ and $\nabla(f\circ \lambda)$ is the projection onto $\Lambda$.
\end{proof}

\begin{corollary}
  In addition to the assumption in Proposition~\ref{prop:projection}, suppose each eigenvalue $\lambda_i(x)$ is simple. Then,
  \begin{equation*}
    P_\Lambda(x) = \sum_{i=1}^d \max(\lambda_i(x),0) \frac{1}{p^{(1)}(x-\lambda_i(x)e)}\nabla p(x-\lambda_i(x)e)
  \end{equation*}
\end{corollary}
\begin{proof}
 If each eigenvalue $\lambda_i(x)$ is simple, then $x \mapsto \lambda_i(x)$ is analytic in a neighbourhood of $x$, and
\begin{equation}\label{eq:lambda_deriv}
\nabla \lambda_i(x) = \frac{1}{p^{(1)}(x-\lambda_i(x)e)}\nabla p(x-\lambda_i(x)e)
\end{equation}
holds, see Section~3.1 and Equation~(3.5) in \cite{renegar2019accelerated}.	
From Proposition~\ref{prop:projection} we have
  \begin{equation*}
    P_\Lambda(x) = \nabla(f\circ \lambda)(x).
  \end{equation*}
  From the definition of $f$ and \eqref{eq:lambda_deriv} we have
  \begin{align*}
    P_\Lambda(x) = \nabla(f\circ \lambda)(x) = & \sum_{i=1}^d\max(\lambda_i(x),0) \nabla \lambda_i(x)\\
    & \sum_{i=1}^d\max(\lambda_i(x),0) \frac{1}{p^{(1)}(x-\lambda_i(x)e)}\nabla p(x-\lambda_i(x)e).
  \end{align*}
\end{proof}

In \cite[Definition~6.3]{BGLS98_pre}, the authors introduced a notion of diagonalizable 
hyperbolic polynomial. In what follows, let $\mathcal{G}_p$ denote the set 
of  linear transformations $A:\Re^n \to \Re^n$ satisfying
$A^* = A^{-1}$ and $\lambda(x) = \lambda(A(x))$ for all $x \in \Re^n$. With that, $\mathcal{G}_p$ is a closed subgroup of the  group of orthogonal linear transformations of $\Re^n$, where the operation is the usual function composition.

\begin{definition}\label{def:diag}
A hyperbolic polynomial $p$ is said to be \emph{diagonalizable} if there exists a linear map $\Delta : \spanVec \lambda(\Re^n) \to \Re^n$ such that the following two properties are satisfied.
\begin{enumerate}[$(i)$]
	\item $\Delta$ is an isometry, i.e., $\norm{u}_2 = \norm{\Delta(u)}_p$ for all $u \in \spanVec \lambda(\Re^n) $,
	\item for every $x \in \Re^n$, there exists $A \in \mathcal{G}_p$ such that $x = A(\Delta(\lambda(x)))$.
\end{enumerate}	
\end{definition}

The determinant polynomial over the $n\times n$ symmetric matrices $\det: \S^n \to \mathbb{R}$ is diagonalizable in the sense of Definition~\ref{def:diag} and a suitable $\Delta$ is given by the map that takes a vector in $\Re^n$ and 
maps to the corresponding diagonal matrix in $\S^n$.
In this case, $\mathcal{G}_{\det}$ corresponds to the linear transformations on $\S^n$ that preserve the eigenvalue map. It is a non-trivial but well-known fact that every such transformation must be of the form 
$x \mapsto vxv^\T$ for some  $n\times n$ orthogonal matrix $v$, where $v^\T$ indicates the usual matrix transpose. 
This follows, for example, from the results contained in \cite[Section~5]{Li79}. 
Overall, for the case of the determinant polynomial $\det: \S^n \to \mathbb{R}$, Definition~\ref{def:diag} expresses the well-known fact that every symmetric matrix has a spectral decomposition and the linear function that takes a vector in $\Re^n$ to its corresponding diagonal matrix in $\mathcal{S}^n$ is an isometry.
For other examples of diagonalizable hyperbolic polynomials see \cite[Section~7]{BGLS98_pre}.
Diagonalizable hyperbolic polynomials must be isometric by \cite[Theorem~6.5]{BGLS98_pre}.

We conclude this subsection with a result that provides the closest analogue to \eqref{eq:psd_proj} we were able to find for hyperbolic polynomials.

\begin{proposition}\label{prop:proj}
 Let $\Lambda = \Lambda(p,e)$ be a  hyperbolicity cone and suppose that $p$ is complete, diagonalizable and $\lambda(\Re^n) = \orR{d}$ holds.
 For $x \in \Re^n$, let $A \in \mathcal{G}_p$ be such that 
 $x = A( \Delta(\lambda(x)))$ holds. Then, we have
 \[
  P_\Lambda(x) = A (\Delta(\lambda(x)_+)),
 \] 
where $\lambda(x)_+$ is the projection of $\lambda(x)$ onto $\Re^n_+$.
\end{proposition}
\begin{proof}
Since $p$ is diagonalizable, it is also isometric by \cite[Theorem~6.5]{BGLS98_pre}.
Therefore, by Lemma~\ref{lem:dist_con} and Proposition~\ref{prop:dist_hyperbolicity} we have
\[
\dist(x, \Lambda)^2 = \sum _{i=1}^d \min(\lambda_i(x),0)^2.
\]
Let $y \coloneqq A( \Delta(\lambda(x)_+))$.
First, we verify that $y \in \Lambda$.
Since $\lambda(x)_+ \in  \Re^d_{\downarrow} = \lambda(\Re^n)$, there exists 
some $z \in \Re^n$ such that $\lambda(z) = \lambda(x)_+$ holds.
Let $B \in \mathcal{G}_p$ be such that $z = B(\Delta \lambda(x)_+)$.
Since $A,B \in \mathcal{G}_p$, we have
\[
\lambda(y) = \lambda(\Delta(\lambda(x)_+)) = \lambda(B(\Delta \lambda(x)_+)) = \lambda(z) = \lambda(x)_+
\]
%
%
%
%
We conclude that $\lambda(y) = \lambda(x)_+$ and $y \in \Lambda$.

Finally, since $A \in \mathcal{G}_p$ and $\Delta$ is an isometry we have
\[
\norm{x-y}_p^2 = \norm{\Delta(\lambda(x) - \lambda(x)_+)}_p^2 = \norm{\lambda(x) - \lambda(x)_+}^2_2 = \sum _{i=1}^d \min(\lambda_i(x),0)^2 = \dist(x, \Lambda)^2.
\]
Therefore, $P_\Lambda(x) = y$.
\end{proof}

\subsection{Limitations and further discussion}
The results we proved so far on the distance function and the projection operator are restricted to isometric hyperbolic polynomials and the sharpest result requires the additional provisions that $\lambda(\Re^n) = \orR{d}$ and that $p$ be diagonalizable. 
Unfortunately there are at  least three issues.

The first is that it is not known 
how large the class of isometric hyperbolic polynomials is and
even hyperbolic polynomials associated to polyhedral cones may fail to be isometric. 
In a sense, this is not surprising because the same cone can be generated by different hyperbolic polynomials, e.g., $\Re^3_+$ also satisfies $\Re^3_+ = \Lambda(p,(1,1,1))$ for $p(x_1,x_2,x_3) = x_1^2x_2^2x_3^2$.
An example of non-isometric polynomial is given in \cite[Example~5.2]{BGLS01},
however it is less than ideal because there there are ``redundancies'' in the description of the underlying cone\footnote{Example~5.2 in 
\cite{BGLS01} corresponds to the restriction of the polynomial $p(x_1,x_2,x_3) = x_1x_2x_3$ to a certain two-dimensional space. Every pointed two-dimensional closed convex cone is isomorphic to $\Re^2_+$, which only requires a degree $2$ hyperbolic polynomial.
So, there is indeed a redundancy in the expression provided for Example~5.2 in 
	\cite{BGLS01}. }. 

The way to ensure that there are no redundancies is to restrict ourselves to \emph{minimal polynomials}, which 
we will now recall. 
For simplicity, assume that $\Lambda$ is a regular hyperbolicity cone.
Although $\Lambda$ can be generated by different hyperbolic polynomials, there exists a hyperbolic polynomial $p$ of minimal degree that generates $\Lambda$ with the property that $p$ divides any other hyperbolic polynomial 
$\hat p$ satisfying $\Lambda = \Lambda(\hat p,e)$, e.g., see \cite[Lemma~2.1]{HV07} for a more general result or 
see the discussion in \cite[Section~2.2]{GIL24}.
We will call such a $p$ a \emph{minimal degree polynomial for $\Lambda$}.
So the more interesting question is whether \emph{minimal} hyperbolic polynomials 
can fail to be isometric. Unfortunately, that is indeed possible.

\begin{proposition}[A non-isometric minimal polynomial associated to a polyhedral cone]\label{prop:non-isometric}
	Let $p:\mathbb{R}^3 \rightarrow \mathbb{R}$ be the polynomial defined as
	\begin{equation*}
	p(x) = (x_1+x_2+x_3)(x_1-x_2+x_3)(2x_1-x_2-x_3)(x_1+2x_2-x_3).
	\end{equation*}
	$p$ is a hyperbolic polynomial along $e = (0,0,1)$, and the hyperbolicity cone for $(p,e)$ is the polyhedral cone
	satisfying $\Lambda(p,e) = \{x \in \Re^3 \mid x_1+x_2+x_3 \geq 0,x_1-x_2+x_3 \geq 0,-2x_1+x_2+x_3 \geq 0,  -x_1-2x_2+x_3 \geq 0 \}$. 
	The polynomial $p$ is minimal for $\Lambda(p,e)$ but $p$ is not isometric.
	
\end{proposition}
\begin{proof}
	See Appendix~\ref{app:non_iso}.
\end{proof}
The second issue is that even if a polynomial is isometric, it is not necessarily the case $\lambda(\Re^d) = \orR{d}$ holds. 
Such an example has already being considered in \cite[Section~6]{BGLS01}, see the ``Singular values'' part and consider Example~6.1 in view of Theorem~5.4 in \cite{BGLS01}. 

To be fair, we have {not} shown that the formula in \eqref{eq:hyper_guess}
fails for non-isometric hyperbolic polynomials. In principle, 
it could still be the case that \eqref{eq:hyper_guess} holds but it requires another proof strategy. However, we have some numerical evidence that \eqref{eq:hyper_guess} does not hold for the 
(non-isometric) hyperbolic polynomial in Proposition~\ref{prop:non-isometric}. Indeed, for $p$ and $e$ as in Proposition~\ref{prop:non-isometric}, if we let 
$x_0\coloneqq (1,1,0)$ and we compute numerically
\[
f_1 \coloneqq \min _{x \in \Lambda(p,e)}\norm{x-x_0}_2,\,\,   f_2 \coloneqq \min _{x \in \Lambda(p,e)}\norm{x-x_0}_p,\,\, f_3\coloneqq \sqrt{\sum _{i=1}^4 \min(\lambda_i(x_0),0)^2}
\]
we obtain $f_1 \cong 1.2421$, $f_2 \cong 3.3923$, $f_3 \cong 3.1623$, which suggests that they are all distinct. 
We have that $f_1$ and $f_2$ are the distances from $x_0$ to $\Lambda(p,e)$ in the usual Euclidean norm and in the norm $\norm{\cdot}_p$, respectively.
A CVXPY \cite{CVXPY} code for these computations can be found in the following link. 

\url{https://github.com/bflourenco/dfw_projection}

%

The third and final issue is that even in the cases that the result in this section apply in full, they are only valid to the projection computed with respect to the norm $\norm{\cdot}_p$ induced by the hyperbolic polynomial in \eqref{eq:norm_poly}. 
Naturally, the norm in \eqref{eq:norm_poly} can be different from the usual Euclidean norm in $\Re^n$, which may limit the applicability of the results. This happens, for example, for the polynomial considered in Proposition~\ref{prop:non-isometric}.

In view of these limitations, we believe it is important to also consider numerical algorithms for computing projections onto hyperbolicity cones. 
This leads us to the next section.

\section{A FW algorithm for strongly convex optimization over regular cones}\label{sec:algorithm}
In this section we develop a numerical method that is able to handle
strongly convex optimization over any regular closed convex cone $\stdCone$, which includes, in particular, the case of a hyperbolicity cone corresponding to a complete hyperbolic polynomial.
More precisely, we aim to solve the following pair of primal-dual problems:
\begin{multicols}{2}
	\noindent\begin{align}
\label{prob:primal}\tag{P}
\min_{x \in \mathbb{R}^n} &\quad f(x) \\
\mathrm{s.t. } &\quad Tx+b \in \stdCone \notag
	\end{align}
	\noindent
	\begin{align}\label{prob:dual}\tag{D}
   \min_{y \in \mathbb{R}^m} &\quad (f^*\circ T^*)(y)+ \langle b,y\rangle\\
\mathrm{s.t. } &\quad y \in \stdCone^*, \notag 
	\end{align}	
\end{multicols}
\vspace*{-1\baselineskip} 
\noindent where $f:\mathbb{R}^n \rightarrow \mathbb{R}$ is a  closed proper $\mu$-strongly convex function, $f^*$ is its conjugate, $T:\mathbb{R}^n \rightarrow \mathbb{R}^m$ is a linear map, $T^*$ is its adjoint and $b\in\Re^m$ is a vector. We recall that a proper convex function is closed if and only if it is lower-semicontinuous \cite[pg.~52]{rockafellar}.
While our discussion will be general for the most part, we will also discuss the special case of hyperbolicity cones where applicable. 

Although the goal is to solve \eqref{prob:primal}, the basic idea described in this section is to apply a Frank-Wolfe method to a compactified version of  \eqref{prob:dual}, because this will lead to tractable subproblems. 
Detailed rationale and theoretical backing for that are given in Section~\ref{sec:alg_fw}. The proposed method and convergence analysis are discussed in Section~\ref{sec:convergence}. Practical considerations and an example involving quadratic optimization are discussed in Sections~\ref{sec:practical} and \ref{sec:example}, respectively.


Before we move on further, some remarks are in order. 
We suppose that  $\Re^m$ is equipped with some arbitrary inner product $\inProd{\cdot}{\cdot}$ for which the corresponding norm is given by $\norm{\cdot}$. A similar remark applies to $\Re^n$ and while $\Re^m$ and $\Re^n$ may have different inner products, for simplicity we will use the same symbols $\inProd{\cdot}{\cdot}$, $\norm{\cdot}$ to denote the inner product and the norm in both spaces. 
Furthermore, the conjugate $f^*$, the adjoint $T^*$ and the dual cone $\stdCone^*$ are, of course, computed with respect to the inner product in the spaces that they are defined.
That said, in contrast to Section~\ref{sec:projection}, one important point is that, even if $\stdCone$ is a hyperbolicity cone, we will \emph{not} require that $\norm{\cdot}$ be the norm induced by underlying  hyperbolic polynomial $p$.

As mentioned in the Section~\ref{sec:int}, \eqref{prob:primal} contains as a particular  case the projection problem, since we can take $b = 0$, let $T$ be the identity map and $f(x) \coloneqq \norm{x-x_0}^2$, where $x_0$ is fixed and $\norm{\cdot}$ is the norm induced by some inner product on $\Re^n$. 
However, \eqref{prob:primal} will also allow us to handle more general quadratic minimization problems, as we shall see in Section~\ref{sec:example}. 

The problem \eqref{prob:dual}, although written as a minimization problem, is actually equivalent to the Fenchel dual of \eqref{prob:primal} and for completeness we show below its derivation.
First, we observe that \eqref{prob:primal} can be expressed as follows using  indicator functions:
\begin{equation}
\label{prob:primal_indicator}
\min_{x \in \mathbb{R}^n} f(x) + \delta_{\stdCone-b}(Tx).
\end{equation}
Then, the Fenchel dual problem of \eqref{prob:primal_indicator} is 
\begin{equation}
\label{prob:dual_indicator}
\max_{y \in \mathbb{R}^m} -f^*(T^*y) - \delta_{\stdCone-b}^*(-y),
\end{equation}
e.g., see \cite[Corollary~31.2.1]{rockafellar}.
Here, for all $y\in\Re^m$,
\begin{align*}
\delta_{\stdCone-b}^*(y) &= \sup_{x\in\Re^m} \{\langle x,y\rangle - \delta_{\stdCone-b}(x) \} = \langle -b,y\rangle + \delta_{\stdCone^\circ}(y),
\end{align*}
where $\stdCone^\circ = -\stdCone^*$ is the polar of $\stdCone$.
Hence, $\delta_{\stdCone-b}^*(-y) =  \langle b,y\rangle + \delta_{\stdCone^*}(y)$. So, \eqref{prob:dual_indicator} is equivalent to \eqref{prob:dual}.

\paragraph{Optimality conditions for \eqref{prob:primal} and \eqref{prob:dual}}
We now briefly discuss the optimality conditions for \eqref{prob:primal} and \eqref{prob:dual}, which follow from classic convex duality theory. 
First, we show that there is no duality gap and both problems are attained under mild conditions. 
Here, we say that a problem is \emph{attained} if there exists a feasible solution whose objective function value is equal to the optimal value of the problem.

\begin{proposition}[No duality gap]\label{prop:no_gap}
Let $p^*$ and $d^*$ denote the optimal values of \eqref{prob:primal} and \eqref{prob:dual}. The following items hold.
\begin{enumerate}[$(i)$]
	\item $p^* + d^* = 0$.
	\item If \eqref{prob:primal} is feasible, then $p^*$ and $d^*$ are finite and $p^*$ is attained.
	\item If \eqref{prob:primal} satisfies Slater's condition (i.e., there exists $\bar{x}$ such that $T\bar{x} + b \in \interior \stdCone$), then $p^*$ and $d^*$ are both finite and attained.
\end{enumerate}
\end{proposition}
\begin{proof}
Since $f$ is assumed to be strongly convex over $\Re^n$, the domain of $f^*$ is $\Re^n$, see \cite[Theorem~4.2.1]{HUL96_II}.
Because $\stdCone$ is pointed, $\stdCone^\circ$ has an interior point.
We note also that $\interior (\stdCone)^\circ = \interior \dom \,\delta_{\stdCone-b}^*$.
Thus any point in $y \in \interior \mathrm{dom} \:\delta_{\stdCone-b}^*$ is
such that 
\begin{equation}\label{eq:dual_cq}
T^*y \in \reInt(\mathrm{dom} \:f^*) = \Re^n.
\end{equation}
Under  \eqref{eq:dual_cq}, we can invoke an appropriate version of Fenchel's duality theorem (e.g., \cite[Corollary~31.2.1]{rockafellar}) to conclude that 
\eqref{prob:primal} and \eqref{prob:dual_indicator} have the same optimal value and $p^*$ is attained if finite. In particular if \eqref{prob:primal} is feasible, then, since \eqref{prob:dual_indicator} is feasible as well, $p^*$ is attained. Recalling that the optimal value of \eqref{prob:dual_indicator} is $-d^*$, we conclude that 
$p^* + d^* = 0$.
This proves items $(i)$ and $(ii)$. 

Furthermore, if \eqref{prob:primal} satisfies Slater condition,
the same Fenchel's duality theorem ensures that $d^*$ is attained, which proves item~$(iii)$.
%
%
%
%
\end{proof}
From Proposition~\ref{prop:no_gap}, there is no duality gap between \eqref{prob:primal} and \eqref{prob:dual} and, as long as Slater's condition is satisfied at \eqref{prob:primal}, both problems have optimal solutions $\OPT{x}$ and $\OPT{y}$, respectively.
Since $f$ is strongly convex, $\OPT{x}$ is unique.
In any case, the solutions are related by the formulae:
\begin{align}
T^*\OPT{y} \in \partial f(\OPT{x}) \label{eq:KKT_f}\\
T\OPT{x} \in \partial {\delta^*_{\stdCone-b}}(-\OPT{y}), \label{eq:KKT_g}
\end{align}
e.g., see \cite[pg.~333 and Theorem~31.3]{rockafellar}.
Since we assumed that $f$ is a closed proper convex function, \eqref{eq:KKT_f}
is equivalent to $\OPT{x} \in \partial f^*(T^*\OPT{y})$, e.g., \cite[Theorem~23.5]{rockafellar}. Moreover, $f^*$ is differentiable because of the
strong convexity of $f$. In the end,  \eqref{eq:KKT_f} is equivalent to
\begin{equation}
\label{eq:KKT_reformulated}
\begin{aligned}
\OPT{x} = \nabla f^*(T^*\OPT{y}).
\end{aligned}
\end{equation}
Similarly, \eqref{eq:KKT_g} is equivalent to $-\OPT{y} \in \partial {\delta_{\stdCone-b}}(T\OPT{x})$.
An important consequence of \eqref{eq:KKT_reformulated} is that the unique optimal solution of \eqref{prob:primal_indicator} can be derived from \emph{any} optimal solution to the dual problem \eqref{prob:dual}.

\

\subsection{Overcoming the challenges of constructing a Frank-Wolfe based method}\label{sec:alg_fw}
When applying a Frank-Wolfe method there are a few challenges a practitioner must handle in order to obtain an efficient algorithm and to ensure convergence. 
For example, Frank-Wolfe methods typically require that the feasible region of the problem is compact, which is not necessarily the case for neither \eqref{prob:primal} nor \eqref{prob:dual}. 
It is also desirable that the subproblem that appears in the method either has a closed form solution or can be solved efficiently. 
In this subsection, we discuss these issues one by one.
Here, we recall our standing assumption that $\stdCone$ is a regular closed convex cone.

\subsubsection*{Issue 1: Primal or dual?}
The first issue is to decide which side of the problem to solve: \eqref{prob:primal} or \eqref{prob:dual}. 
Overall, we decided to solve \eqref{prob:dual} because the FW subproblem in this case can be connected to the generalized  minimum eigenvalue problem over a regular cone, as we shall discuss shortly. Furthermore, an optimal solution to \eqref{prob:primal} can be obtained via \eqref{eq:KKT_reformulated}, provided that Slater's condition for \eqref{prob:primal} is satisfied as in Proposition~\ref{prop:no_gap}.
Having settled for \eqref{prob:dual}, we discuss in the sequel two outstanding issues.

\subsubsection*{Issue 2: Compactness of the feasible region}
The classical FW method requires the compactness of the feasible region, so it can not be applied directly to \eqref{prob:dual}. 
However, this can be fixed by adding a constraint that cuts a compact slice of the feasible region of \eqref{prob:dual} in such a way that at least one optimal solution is inside the slice. 
In order to do that, we suppose we are given $e \in \Re^n$ and a constant $c_D $ as in the following assumption.

\begin{assumption}
	\label{asm:optimal_solution_bound}
We assume that $e \in \interior \stdCone$ and $c_D > 0$ are such that there exists at least one optimal solution  $y_\mathrm{opt}$  to \eqref{prob:dual} satisfying
$\inProd{e}{y_\mathrm{opt}} \leq c_D$.
\end{assumption}

Under Assumption~\ref{asm:optimal_solution_bound} we can show that
\begin{equation}
  \label{prob:dual_compact}
  \begin{array}{ll}
    \displaystyle
    \min_{y \in \mathbb{R}^m} & (f^*\circ T^*)(y) + \langle b,y\rangle\\
    \mathrm{s.t. } & \langle e, y \rangle \leq c_D \\
    & y \in {\stdCone^*}
  \end{array}
\end{equation}
has a compact feasible region, and, by assumption, \eqref{prob:dual_compact}  has at least one  optimal solution of \eqref{prob:dual} among its optimal solutions.
\begin{proposition}
  \label{prop:equivalent_compact_problem}
  The non-empty convex set 
  \begin{equation}\label{eq:compact_set}
      \{ y \in \stdCone^* \mid \inProd{e}{y} \leq c_D \}
  \end{equation}
  is compact. Moreover, \eqref{eq:compact_set} contains at least one optimal solution of \eqref{prob:dual}
  under Assumption~\ref{asm:optimal_solution_bound}.
\end{proposition}
\begin{proof}
  First, since  $0$ is always contained in \eqref{eq:compact_set}, the set in  \eqref{eq:compact_set} is non-empty.
  It is also closed and convex because $\stdCone^*$ is a closed convex cone.
 Next,  we prove the boundedness of \eqref{eq:compact_set} by checking that the recession cone of \eqref{eq:compact_set} is trivial.
 Since the set in \eqref{eq:compact_set} contains zero, its recession cone coincides with the set of elements $y$ such that $\lambda y $ belong to  \eqref{eq:compact_set} for all $\lambda \geq 0$, e.g., see \cite[Theorem~8.3]{rockafellar}.
 All such  $y$ must then belong to $\stdCone^*$ and satisfy $\inProd{e}{y} \leq 0$. However, since $e$ is an interior point of $\stdCone$, we have $\inProd{e}{y} = 0$ which forces $y = 0$. 
%
Finally, \eqref{eq:compact_set} contains at least one optimal solution of \eqref{prob:dual} by Assumption~\ref{asm:optimal_solution_bound}.
\end{proof}
In view of Proposition~\ref{prop:equivalent_compact_problem}, at least one of the optimal solutions of \eqref{prob:dual} can be obtained by solving \eqref{prob:dual_compact}. We also note that since the objective functions in \eqref{prob:dual_compact} and \eqref{prob:dual} are the same, the set of optimal solutions of the former is included in the optimal solution set of the latter.
However, \eqref{prob:dual_compact} has a compact feasible region so it is amenable to the classical FW method.

Of course, for a given problem, the crux of the issue is whether we can easily obtain $e$ and $c_{D}$ as in Assumption~\ref{asm:optimal_solution_bound}. 
As we will show in Section~\ref{sec:example}, $c_D$ can be computed explicitly from the problem data in the case of quadratic optimization, so the assumption of having $c_D$ at hand will not be problematic for our purposes.
\subsubsection*{Issue 3: Solving the FW subproblem}
Our current state of affairs is as follows. Having decided to apply a FW method to the dual side of our problem of interest, we showed that it is enough to solve \eqref{prob:dual_compact}, which is a compact version of \eqref{prob:dual} containing at least one of its optimal solutions.
Next we show that this choice indeed leads to easy subproblems.
Applying Algorithm~\ref{alg:FW_method} to the problem \eqref{prob:dual_compact} leads to the following subproblem at each iteration:
\begin{equation}
  \label{eq:FW_subproblem_aux}
  \begin{array}{ll}
    \displaystyle
    \min_{s \in \mathbb{R}^m} & \langle \nabla(f^* \circ T^*)(y_k)+b, s \rangle \\
    \mathrm{s.t. } & \langle e, s \rangle \leq c_D \\
    & s \in {\stdCone^*},
  \end{array}
\end{equation}
where we use $y_k$ in place of $x_k$ since we are working from the dual side. In what follows, it will be helpful to define
\begin{equation}
	 x_k \coloneqq \nabla f^*(T^*y_k).
\end{equation}
With that, we have $\nabla (f^* \circ T^*)(y_k)+b = T\nabla f^*(T^*y_k)+b = Tx_k+b$. Also, we transform the inequality constraint into a equality constraint by using a slack variable $\alpha$. With that, we arrive at the following subproblem, which is equivalent to \eqref{eq:FW_subproblem_aux}.
\begin{equation}
\label{eq:FW_subproblem}
\begin{array}{ll}
\displaystyle
\min_{s \in \Re^m, \alpha \in \Re} & \langle Tx_k+b, s \rangle \\
\mathrm{s.t. } & \langle e, s \rangle + \alpha = c_D \\
& s \in {\stdCone^*}, \alpha \in \Re_+.
\end{array}
\end{equation}
The problem \eqref{eq:FW_subproblem} is a common conic linear program in primal format. The next goal is to show that an optimal solution of \eqref{eq:FW_subproblem} can be written explicitly in terms of the the corresponding generalized eigenvalue function. 
In order to do so, we consider the dual problem of \eqref{eq:FW_subproblem}.

\begin{equation}
\label{eq:FW_dual_subproblem}
\begin{array}{cl}
\displaystyle
\max_{t\in\mathbb{R},z\in\mathbb{R}^m}  &c_D t \\
\mathrm{s.t. } & z  = (Tx_k+b) - te \\
& z \in \stdCone, t \leq 0.
\end{array}
\end{equation}

The problem in \eqref{eq:FW_dual_subproblem} is closely related to the minimum eigenvalue problem in \eqref{eq:min_eig_def}. 
As hinted previously, this is why solving our problem of interest from the dual side makes sense: when doing so, we arrive at a subproblem whose optimal value can be obtained from a minimum eigenvalue computation. 
We are now positioned to show our main theorem for this subsection.

\begin{theorem}[Closed-form solution of the FW subproblem]\label{theo:FW_sub_opt}
	Consider the primal-dual pair of problems \eqref{eq:FW_subproblem} and \eqref{eq:FW_dual_subproblem}.
	Then, the following statements hold.
	\begin{enumerate}[$(i)$]
		\item Both \eqref{eq:FW_subproblem} and \eqref{eq:FW_dual_subproblem} satisfy Slater's condition.
		In particular, the optimal values of both problems coincide and are attained.
		\item The optimal solution of \eqref{eq:FW_dual_subproblem} is given by
		\begin{equation}\label{eq:FW_dual_subproblem_opt}\begin{array}{l}
		t_{\mathrm{opt}} = \min(0,\lambda_{\mathrm{min}}(Tx_k+b)) \\
		z_{\mathrm{opt}} = Tx_k+b - t_{\mathrm{opt}} e,
		\end{array} 
		\end{equation}
		where $\lambda_{\min}$ is the minimum eigenvalue function along the direction $e$, as in \eqref{eq:min_eig_def}.
		\item If $t_{\mathrm{opt}} = 0$, then $(0,c_D) \in \Re^m \times \Re$ is an optimal solution to \eqref{eq:FW_subproblem}. 
		Otherwise, if $t_{\mathrm{opt}} < 0$, the optimal solution set of \eqref{eq:FW_subproblem} with respect to the $s$ variable is \begin{equation}
		\label{eq:optimal_solution_set_subproblem}
		\{ s \in {\stdCone^*} \mid \langle e,s \rangle = c_D , \:\: \langle s,z_{\mathrm{opt}} \rangle = 0 \}
		= \{ s \in \cj{{z_{\mathrm{opt}} }} \mid \langle e,s \rangle = c_D \},
		\end{equation}	
		where $\cj{{z_{\mathrm{opt}} }}$ is the conjugate face of $\stdCone$ at $z_{\mathrm{opt}}$ as in \eqref{eq:cj_def}.
\end{enumerate}
	
\end{theorem}
\begin{proof}
\fbox{Item~$(i)$}
First we check that \eqref{eq:FW_subproblem} satisfies Slater's condition.
$\stdCone$ is a regular cone by assumption, so $\interior \stdCone^*$ is not empty. Let $s \in \interior \stdCone^*$.
If $\inProd{e}{s} = 0$, then since $e$ is an interior point of $\stdCone$, we have $s = 0$, which is impossible since $0 \in \interior \stdCone^*$ implies 
$\stdCone^* = \Re^m$.
Therefore, $s$ satisfies $\inProd{e}{s}> 0$. Let
\begin{equation}
\hat{s} \coloneqq \frac{c_D}{2\langle e, s \rangle}s.
\end{equation}
Then, $\hat{s}\in \interior \stdCone^*$ and $
\langle e, \hat{s}\rangle = \frac{c_D}{2}$
hold. Therefore, $(\hat{s}, {c_D}/{2})$ is a strictly feasible solution to \eqref{eq:FW_subproblem}.

Next, we check that \eqref{eq:FW_dual_subproblem} satisfies Slater's condition. Since $e \in \interior \stdCone$, there exists a small $u > 0$ such that $(Tx_k + b)u + e \in \interior \stdCone$. Therefore, for $t \coloneqq  -1/u$, we have 
$Tx_k + b - te  \in \interior \stdCone$ and $(Tx_k + b - te, t) $ is a strictly feasible solution.

\fbox{Item~$(ii)$}
We divide the proof in two cases.

\fbox{Case $(a)$: $Tx_k+b \in \stdCone$}
In this case, 
we have $\inProd{Tx_k+b}{s} \geq 0$ for every $s \in {\stdCone^*}$. Since $(0,c_D)$ is feasible for \eqref{eq:FW_subproblem}, $(0,c_D)$ is an optimal solution of \eqref{eq:FW_subproblem}. Because of item~$(i)$
the optimal value of \eqref{eq:FW_dual_subproblem} must be zero as well, so $t_{\mathrm{opt}} = 0$ which coincides with $\min(0,\lambda_{\mathrm{min}}(Tx_k+b))$.

\fbox{Case $(b)$: $Tx_k+b \notin \stdCone$} 
Let $t_{\mathrm{opt}}$ denote the optimal solution of \eqref{eq:FW_dual_subproblem}, which exists and is finite because of item~$(i)$. 
By definition of $\lambda_{\min}$ (see \eqref{eq:min_eig_def}), we have  $t_{\mathrm{opt}} \leq \lambda_{\min}(Tx_k + b)$, since the problem in \eqref{eq:FW_dual_subproblem} has one additional constraint in comparison to the problem in \eqref{eq:min_eig_def}. However, since 
$Tx_k + b \not \in \stdCone$, the inequality $\lambda_{\min}(Tx_k+b) < 0$ holds by item~$(ii)$ of Proposition~\ref{prop:eig_p}. Therefore, 
$t \coloneqq \lambda_{\min}(Tx_k+b)$ and $z \coloneqq Tx_k+b - t e $ is feasible for \eqref{eq:FW_dual_subproblem}, so 
$t_{\mathrm{opt}}  = \lambda_{\min}(Tx_k+b)$.
Therefore, in this case too, the formula in \eqref{eq:FW_dual_subproblem_opt} holds.


\fbox{Item~$(iii)$} If $t_{\mathrm{opt}} = \min(0,\lambda_{\mathrm{min}}(Tx_k+b)) = 0$, then $\lambda_{\mathrm{min}}(Tx_k+b) \geq 0$ which implies that $Tx_k + b  \in \stdCone$, by Proposition~\ref{prop:eig_p}.
In this case, we already verified in Case (a) of item~$(ii)$ that $(0,c_D)$ is an optimal solution to \eqref{eq:FW_subproblem}.
Suppose that $t_{\mathrm{opt}} < 0$.
Since both \eqref{eq:FW_subproblem} and \eqref{eq:FW_dual_subproblem} satisfy Slater's condition, the following conditions from classical conic linear programming duality theory are necessary and sufficient for optimality
\begin{align}
\inProd{z}{s} -t\alpha &= 0\label{eq:comp}\\
\langle e, s \rangle + \alpha &= c_D \notag\\ 
 (Tx_k+b) - te & = z \notag\\
z \in \stdCone, s \in \stdCone^*, t \leq 0, \alpha &\geq 0. \notag
\end{align}
Therefore, if $t_{\mathrm{opt}} < 0$, then  complementary slackness (i.e., \eqref{eq:comp}) and $\alpha \geq 0$ implies that the optimal $\alpha^*$ in \eqref{eq:FW_subproblem} is $0$.
In particular, the $s \in \stdCone^*$ that are optimal for \eqref{eq:FW_subproblem} are exactly the ones that satisfy $\inProd{e}{s} = c_D$ and $\inProd{s}{z_{\mathrm{opt}}} = 0$.
\end{proof}
So far, we have shown that we can obtain the optimal \emph{value} of \eqref{eq:FW_dual_subproblem} through a minimum eigenvalue computation. However, we still require an optimal \emph{solution} for the case where $t_{\mathrm{opt}} < 0$.
Fortunately, from item~$(iii)$ of Theorem~\ref{theo:FW_sub_opt}, we see that it is enough to find a nonzero $s\in \cj{{z_{\mathrm{opt}} }}$ and rescale $s$ so that 
$\langle e,s \rangle = c_D $ holds. This leads us to our final point in this subsection.

\paragraph{Conjugate vector computations}
In view of Theorem~\ref{theo:FW_sub_opt}, the final piece we need to complete our discussion is a method to find a nonzero vector in $\cj{{z_{\mathrm{opt}} }}$. Unfortunately, this is a task that is very specific to the cone at hand, so now we take a look at some particular cases. 

First, if $\stdCone = \Lambda(p,e)$ is an arbitrary hyperbolicity cone, 
it seems  nontrivial to obtain a formula for $\cj{z}$ given an arbitrary $z \in \Lambda(p,e)$.
That said, a \emph{specific}  nonzero conjugate vector can be obtained easily by using the following proposition. Here, we recall that $\mult(z)$ denotes the number of zero eigenvalues of $z$ and $p^{(i)}$ is the $i$-th directional derivative of $p$ along a fixed hyperbolic direction, see Section~\ref{sec:hyp_p}.
\begin{proposition}[Conjugate vectors in hyperbolicity cones]
  \label{prop:normal_vector}
Let $\Lambda = \Lambda(p,e) \subseteq \Re^m$ be a hyperbolicity cone.
Let $z \in \Lambda$ satisfy $\mathrm{mult}(z) \geq 1$ and define $r \coloneqq \mult(z)$. Then, 
\begin{equation*}
 \nabla p^{(r-1)}(z) \in \cj{z} \setminus \{0\}.
\end{equation*}
\end{proposition}
\begin{proof}
First, we observe that $z \in \Lambda$ and $\mathrm{mult}(z) \geq 1$ implies that $z$ is in the boundary of $\Lambda$ which follows, for example, from 
item~(iii) of Proposition~\ref{prop:eig_p} and Proposition~\ref{prop:hyp_min_eig}, see also \cite[Section~3]{renegar2004hyperbolic}.

We first consider the case $r = 1$. We recall that for $y \in \Lambda$, we have $p(y) \geq 0$. In addition, $p(y) = 0$ holds if $y$ is in the boundary of $\Lambda$. In particular, we have $p(z+ty) \geq 0$ and $p(z+tz) = 0$ for every $y \in \Lambda$ and $t  \geq 0$. Then, taking the derivative with respect to $t$ at $t = 0$, we obtain
\[
\inProd{\nabla p(z)}{y} \geq 0, \forall y \in \Lambda \quad\text{ and }\quad \inProd{\nabla p(z)}{z} = 0.
\]
This shows that $\nabla p(z) \in \cj{z}$. 
However, since $\mult(z) = 1$, $\nabla p(z)$ is nonzero (e.g., \cite[Lemma~7]{renegar2004hyperbolic}). 
Therefore, $\nabla p^{(r-1)}(x) = \nabla p(x)$ is a non-zero conjugate vector.  

Next, we consider the case $r \geq 2$. 
We recall that \begin{equation*}
  \partial^i \Lambda= \{ x \in \Lambda \mid \mathrm{mult}(x) = i \} .
\end{equation*}
By definition, $x \in \partial^r \Lambda$. Using Theorem \ref{thm:mult_derivative} repeatedly,
\begin{equation*}
  \partial^r \Lambda = \partial^{(r-1)} \Lambda^{(1)} = \cdots = \partial^2 \Lambda^{(r-2)} = (\partial^1 \Lambda^{(r-1)}) \cap \Lambda^{(r-2)}
\end{equation*}
holds. Therefore, $x \in (\partial^1 \Lambda^{(r-1)}) \cap \Lambda^{(r-2)}$.
Letting $\bar \Lambda \coloneqq \Lambda^{(r-1)}$, the multiplicity of $0$ as an eigenvalue of $z$ with respect to $p^{(r-1)}$ is one, so we can apply the previous case to $\bar \Lambda$ and $p^{(r-1)}$ to conclude 
that $\nabla p^{(r-1)}(z)  \neq 0$ holds and $\nabla p^{(r-1)}(z)$ belongs to the conjugate face of $\Lambda^{(r-1)}$ at $z$. However, 
since $\Lambda \subseteq \Lambda^{(r-1)}$, we have 
${(\Lambda^{(r-1)})}^* \subseteq \Lambda^*$.
This implies that $\nabla p^{(r-1)}(z)$ belongs to the conjugate face of $\Lambda$ at $x$.
\end{proof}
We are now ready to complete our discussion on the optimal solutions of \eqref{eq:FW_subproblem} for the case of a hyperbolicity cone, which 
we note as a corollary.

\begin{corollary}\label{col:opt_sub_hyper}
Under the setting of Theorem~\ref{theo:FW_sub_opt}, suppose that 
$\stdCone = \Lambda(p,e)$ is a regular hyperbolicity cone. 
If $\OPT{t} < 0$ then
\begin{equation}
\label{eq:optimal_solution_subproblem}
\frac{c_D}{\langle e, \nabla p^{(\mathrm{mult}(z_{\mathrm{opt}})-1)}(z_{\mathrm{opt}}) \rangle}\nabla p^{(\mathrm{mult}(z_{\mathrm{opt}})-1)}(z_{\mathrm{opt}})
\end{equation}
is an optimal solution of \eqref{prob:dual_compact}.
\end{corollary}
We recall that \[
{\langle e, \nabla p^{(\mathrm{mult}(z_{\mathrm{opt}})-1)}(z_{\mathrm{opt}}) \rangle} \neq 0\]
holds in Corollary~\ref{col:opt_sub_hyper} because $e \in \interior \Lambda$ and 
$\nabla p^{(\mathrm{mult}(z_{\mathrm{opt}})-1)}(z_{\mathrm{opt}}) \neq 0$.

Before we move on, we mention other  useful cones $\stdCone$ for which $\cj{z}$ is completely known, given a particular $z \in \stdCone$. 
For symmetric cones (this includes the case of positive semidefinite matrices and second order cones), formulae are given in \cite[Theorem~2]{FB06} and \cite[Section~4.1.1]{L17}. For $p$-cones with $p \in (1,\infty)$, see \cite[Section~4.1]{LLP21}. 
Power cones and exponential cones are linearly isomorphic to their dual cones under the Euclidean inner product, so with some adjustments, the formulae 
discussed in \cite[Section~3.1]{LLLP22} and \cite[Section~4.1]{LLP22}, respectively, can also be used to determine conjugate faces for these two cones. 

\subsection{The proposed method and its convergence analysis} \label{sec:convergence}
Having solved the issues related to applying  Algorithm~\ref{alg:FW_method} to our problem of interest, we can now formally state our obtained method, 
which is described in Algorithm~\ref{alg:proposed_method}.
In Algorithm \ref{alg:proposed_method},
any step size rule which guarantees the convergence of the FW method can be used (e.g., rules discussed in Theorem~\ref{thm:convergence_FW}).
In what follows we will prove some convergence guarantees for Algorithm~\ref{alg:proposed_method}. Although the classical convergence theory of FW methods already provides some guarantees,  since we are applying the method to the \emph{dual} side of the problem, it is still necessary to show theoretical guarantees for the primal iterates $x_k$.

\begin{algorithm}[h]
	\caption{Dual Frank-Wolfe method for solving \eqref{prob:primal}}
	\label{alg:proposed_method}
	\begin{algorithmic}[1]
		\STATE Choose initial point $y_0 \in \stdCone^*$ satisfying $\inProd{e}{y_0} \leq c_D$.
		\FOR{$k = 0,1,\dots$}
		\STATE $x_k  \coloneqq \nabla f^*(T^*y_k)$
		\IF {$y_k$ or $x_k$ satisfies stopping criterion}
		\STATE Break
		\ENDIF
		
%
		\IF {$Tx_k+b \in \stdCone$}
		\STATE $s_k =  0 $
		\ELSE
		\STATE $z_k \coloneqq Tx_k+b - \lambda_{\mathrm{min}}(Tx_k+b) e $
		\STATE Let $\hat s_k \in \cj{z_k} \setminus \{0\}$ and let $s_k \coloneqq \frac{c_D}{\inProd{e}{\hat s_k}}\hat s_k$. {\color{gray}\textit{(For the case of hyperbolicity cones, see Proposition~\ref{prop:normal_vector})}}
		\ENDIF
		
		\STATE $d_k \coloneqq s_k - y_k$
		\STATE Choose $\alpha_k \in (0,1] $ by an appropriate rule
		\STATE $y_{k+1} \coloneqq y_k + \alpha_k d_k$
		\ENDFOR
	\end{algorithmic}
\end{algorithm}


%

First we recall and introduce some notation. We denote by $d^*$ the optimal value \eqref{prob:dual}. Recall that under Assumption~\ref{asm:optimal_solution_bound}, $d^*$ also coincides with the optimal value of \eqref{prob:dual_compact}. Also, let $\OPT{Y}$ denote the optimal solution set 
of \eqref{prob:dual_compact} and let $h$ denote the objective function of \eqref{prob:dual_compact}.

The analysis conducted in this subsection will, for the most part, be method agnostic. 
More precisely, suppose that $\{(x_k,y_k)\} \subseteq \Re^n\times \Re^n$ is any sequence such that the following assumption is satisfied.

\begin{assumptionAlpha}
	\label{asm:a} For every $k$, $y_k$ is feasible for \eqref{prob:dual_compact},  $x_k = \nabla f^*(T^*y_k)$ holds and $h(y_k) \to d^*$ holds.
\end{assumptionAlpha}

In particular, under Assumption~\ref{asm:optimal_solution_bound} and under an appropriate step size rule (see Theorem~\ref{thm:convergence_FW}), the iterates generated by Algorithm~\ref{alg:proposed_method} satisfy Assumption~\ref{asm:a}, although the feasibility of $y_k$ requires some comments. 
The initial iterate $y_0$ is in the feasible region of \eqref{prob:dual_compact} and all the subsequent iterates $y_{k+1}$ are obtained by taking a convex combination between $y_{k}$ and a direction $s_{k}$ that is feasible for \eqref{prob:dual_compact} as well. In particular, $y_k$ is always feasible for \eqref{prob:dual_compact} indeed.

We start with the following well-known lemma.
\begin{lemma}\label{lem:yk_opt}
Under Assumptions~\ref{asm:optimal_solution_bound} and \ref{asm:a}, the sequence $\{y_k\}$  satisfies
\begin{equation}\label{eq:yk_opt}
\lim_{k\to \infty} \dist(y_k,\OPT{Y})  = 0. \end{equation}	
\end{lemma}
\begin{proof} See Appendix~\ref{app:proof_lemyopt}.
\end{proof}
Although the sequence $\{y_k\}$  is not ensured to converge, limits of convergent subsequences must all be minimizers of \eqref{prob:dual_compact}, by Lemma~\ref{lem:yk_opt}. 
 With that, our first result is that the primal iterates $x_k$ indeed converge to the unique optimal solution.
\begin{theorem}[Convergence of primal iterates]\label{theo:xk}
Under Assumptions~\ref{asm:optimal_solution_bound} and \ref{asm:a} we have
  \begin{align*}
   \lim_{k\to \infty} x_k &= \OPT{x},
  \end{align*}
  where $\OPT{x}$ is the optimal solution of \eqref{prob:primal_indicator}.
\end{theorem}
\begin{proof}
%
 From \eqref{eq:KKT_reformulated}, we have
  \begin{equation*}
    \forall \OPT{y} \in \OPT{Y}, \:\: \OPT{x} = \nabla f^*(T^* \OPT{y}).
  \end{equation*}
  Therefore,
  \begin{align*}
    \forall \OPT{y} \in \OPT{Y}, \:\: \norm{x_k-\OPT{x}} &= \norm{x_k- \nabla f^*(T^* \OPT{y})} \\
    &=\norm{\nabla f^*(T^* y_k)- \nabla f^*(T^* \OPT{y})}.
  \end{align*}
  Taking the infimum with respect to $\OPT{y}$ in $\OPT{Y}$, we obtain 
  \begin{equation*}
    \norm{x_k-\OPT{x}} = \inf_{\OPT{y} \in \OPT{Y}}\norm{\nabla f^*(T^* y_k) - \nabla f^*(T^* \OPT{y})}.
  \end{equation*}
  We recall that $\nabla f^*$ is $1/\mu$-Lipschitz continuous because of the $\mu$-strong convexity of $f$, see \cite[Theorem~4.2.1]{HUL96_II}. Thus,
  \begin{align*}
   \norm{x_k-\OPT{x}} &= \inf_{\OPT{y} \in \OPT{Y}}\norm{\nabla f^*(T^* y_k) - \nabla f^*(T^* \OPT{y})} \\
   &\leq  \inf_{\OPT{y} \in \OPT{Y}}\frac{1}{\mu}\norm{T^* y_k - T^* \OPT{y}} \\
   &\leq  \inf_{\OPT{y} \in \OPT{Y}}\frac{\|T^*\|_{\mathrm{op}}}{\mu} \norm{y_k - \OPT{y}} \\
   &= \frac{\|T^*\|_{\mathrm{op}}}{\mu} \dist(y_k,\OPT{Y}).
  \end{align*}
 However, under the stated assumptions, Lemma~\ref{lem:yk_opt} ensures that $\lim_{k\to \infty} \dist(y_k,\OPT{Y}) = 0$ holds, which leads to $ \lim_{k\to \infty}  \norm{x_k-\OPT{x}} = 0$.
\end{proof}

In spite of the potential lack of convergence of $\{y_k\}$,  the primal sequence $\{x_k\}$ \emph{is} ensured to converge by Theorem~\ref{theo:xk}.
Related to that, under a mild assumption on $T$, next we prove that the rate of convergence of $x_k$ to $\OPT{x}$ is no worse than the square-root of the rate of convergence of $h(y_k)$ to $d^*$.

\begin{theorem}[Rate of convergence]\label{thm:convergence_rate}
	Suppose that Assumptions~\ref{asm:optimal_solution_bound} and \ref{asm:a} hold,
  $f$ is $\mu$-strongly convex function and  that $T^*T $ is positive definite. Then, for every $k$ we have
  \begin{equation}
    \|x_k - \OPT{x}\| \leq \sqrt{\frac{2\|T\|_{\mathrm{op}}^2}{\mu\lambda_{\min}(T^*T)}}
    \sqrt{h(y_k) - d^*},
    \end{equation}
    where $\lambda_{\min}(T^*T)\:\:(>0)$ is the minimum eigenvalue of $T^*T$, $\OPT{x}$ is the optimal solution of
    \eqref{prob:primal_indicator} and $d^*$ is the optimal value of  \eqref{prob:dual}.
\end{theorem}

\begin{proof}
First, we check that $\nabla h$ is Lipschitz continuous with constant $\|T\|_{\mathrm{op}}^2/\mu$.
 Since $f$ is $\mu$-strongly convex,  $\nabla f^*$ is $1/\mu$-Lipschitz continuous (\cite[Theorem~4.2.1]{HUL96_II}). 
 Thus, for $y,y' \in \Re^n$ we have
  \begin{align*}
    \| \nabla h(y) - \nabla h(y') \| &= \| T\nabla f^*(T^*y) - T\nabla f^*(T^*y')\| \\
    & \leq \|T\|_{\mathrm{op}} \| \nabla f^*(T^*y) - \nabla f^*(T^*y')\| \\
    & \leq \|T\|_{\mathrm{op}} \frac{1}{\mu}\| T^*y - T^*y' \| \\
    & \leq \|T\|_{\mathrm{op}} \frac{1}{\mu} \|T^*\|_{\mathrm{op}}\| y - y' \| \\
    & = \frac{1}{\mu} \|T\|_{\mathrm{op}}^2 \| y - y' \|.
  \end{align*}
  From the convexity of $h$ and the  Lipschitz continuity of $\nabla h$ with constant $\|T\|_{\mathrm{op}}^2/\mu$, we obtain
  \begin{equation*}
  h(y_k) \geq h(\OPT{y}) + \langle y_k - \OPT{y},\nabla h(\OPT{y})\rangle 
  + \frac{1}{2}\frac{\mu}{\|T\|_{\mathrm{op}}^2}\|\nabla h(y_k) - \nabla h(\OPT{y}) \|^2 ,
  \end{equation*}
  see, e.g., \cite[Theorem~2.1.5, Equation (2.1.10)]{Nes18}. Recalling that  $\nabla h(y_k) = T\nabla f^*(T^* y_k) + b = T x_k+b$ and readjusting the inequality, we obtain
    \begin{align}
  \|Tx_k - T\OPT{x} \|^2 \leq  \frac{2\|T\|_{\mathrm{op}}^2}{\mu}\left(  h(y_k) - h(\OPT{y})
  + \langle y_k - \OPT{y},-\nabla h(\OPT{y})\rangle\right).
  \label{eq:norm_grad_dual_obj}
  \end{align}  
By Assumption~\ref{asm:a}, $y_k$ is always feasible for \eqref{prob:dual_compact}. From the first-order optimality conditions for the problem \eqref{prob:dual_compact} (e.g., \cite[Theorem~2.2.9]{Nes18}), we obtain the inequality
  \begin{equation}
    \label{eq:first_order_optimality}
    \langle y_k - \OPT{y},-\nabla h(\OPT{y})\rangle \leq 0.
  \end{equation}
  Also, from the assumption $T^*T \succ 0$,
  \begin{equation}
    \label{eq:lower_bound_T^*T}
      \|Tx_k - T\OPT{x} \|^2 \geq \lambda_{\min}(T^*T)  \norm{x_k - \OPT{x}}^2 .
  \end{equation}
  From \eqref{eq:norm_grad_dual_obj}, \eqref{eq:first_order_optimality}, and \eqref{eq:lower_bound_T^*T}, we obtain 
  \begin{align*}
     \lambda_{\min}(T^*T) \norm{x_k - \OPT{x}}^2 
      &\leq \frac{2\|T\|_{\mathrm{op}}^2}{\mu}\left(  h(y_k) - h(\OPT{y})\right), 
      \end{align*}
  which leads to
  \[
  \norm{x_k - \OPT{x}}
  \leq \sqrt{\frac{2\|T\|_{\mathrm{op}}^2}{\mu\lambda_{\min}(T^*T)}}\sqrt{h(y_k) - h(\OPT{y})}.
  \]
\end{proof}
As a corollary of Theorem~\ref{thm:convergence_rate} and of the fact that a convex function is locally Lipschitz continuous  over the relative interior of its  domain, we can also get a rate of convergence for the primal objective function.
\begin{corollary}[Convergence of the primal objective function]\label{col:primal_obj}
Suppose that Assumptions~\ref{asm:optimal_solution_bound} and \ref{asm:a} hold. Then, there exists a positive constant $L$ such that the output of Algorithm~\ref{alg:FW_method} after $k$ iterations satisfies
\[
f(x_k) - f(\OPT{x}) \leq L \norm{x_k - \OPT{x}}.
\]
In particular, under the assumptions of Theorem~\ref{thm:convergence_rate} we have 
\[
f(x_k) - f(\OPT{x}) \leq L\sqrt{\frac{2\|T\|_{\mathrm{op}}^2}{\mu\lambda_{\min}(T^*T)}}
\sqrt{h(y_k) - d^*}.
\]
\end{corollary}
\begin{proof}
By Assumption~\ref{asm:a} and the compactness of the feasible region of \eqref{prob:dual_compact} (Proposition~\ref{prob:dual_compact}), the sequence of iterates $\{y_k\}$ is bounded. 
By definition $x_k =  \nabla f^*(T^* y_{k})$ holds and  $\nabla f^*$ is Lipschitz continuous, because of the strong convexity of $f$. In particular, $\{x_k\}$ is the image of the bounded set $\{y_k\}$ via a continuous map with closed domain, so $\{x_k\}$ is bounded as well\footnote{Of course, the
	 fact that the domain is closed is important, otherwise the  image of the interval $(0,1)$ by the function $x \mapsto 1/x$ would be a counter-example.} . 
Although  a given $x_k$ might fail to be feasible for \eqref{prob:primal}, it belongs to the domain of $f$, which is $\Re^n$ by assumption. 
Now, a convex function is Lipschitz continuous relative to any bounded set whose closure is contained in the relative interior of its domain (e.g., \cite[Theorem~10.4]{rockafellar}).
Then, since $\{x_k\}$ is bounded, there exists a constant $L > 0$ such that 
\[
f(x_k) - f(\OPT{x}) \leq L \norm{x_k - \OPT{x}}
\]
holds for every $k$, as we wanted to show. The remainder of the corollary follows directly from Theorem~\ref{thm:convergence_rate}.
\end{proof}
The summary of Theorems~\ref{theo:xk}, \ref{thm:convergence_rate} and Corollary~\ref{col:primal_obj} is the following. 
Under Assumptions~\ref{asm:optimal_solution_bound} and \ref{asm:a},
$x_k$ indeed converges to $\OPT{x}$. 
Furthermore, the convergence rates of $x_k$ to $\OPT{x}$ and of $f(x_k)$ to $f(\OPT{x})$ are no worse than the square root of the convergence rate of $h(y_k)$ to $d^*$, provided that $T^*T$ is positive definite.
These results are not specific to Frank-Wolfe methods, since they hold for any approach that generate sequences as in Assumption~\ref{asm:a} and arise as consequences of the relations between a strongly convex optimization problem and its dual.

In the particular case that $\{(x_k,y_k)\}$ are the iterates generated by Algorithm~\ref{alg:proposed_method}, if the step size is chosen as to ensure $h(y_k) - d^* = O(1/k)$ (see Theorem~\ref{thm:convergence_FW}), we have the following guarantees on the primal iterates and the primal objective function.
\begin{gather}
\lim_{k\to \infty } x_k = \OPT{x}, \label{eq:global_convergence}\\
\| x_k - \OPT{x}\| = O(1/\sqrt{k}), \label{eq:seqencial_convergence}\\
f(x_k)-f(x_{\text{opt}}) = O (1/\sqrt{k}) \label{eq:obj_vals_convergence},
\end{gather}
with the caveat that \eqref{eq:seqencial_convergence} and
\eqref{eq:obj_vals_convergence} 
require the extra assumption that $T^*T$ is positive definite (i.e., $T$ is injective).

Before we move on, we should remark that it was recently shown that the iterates of the Frank-Wolfe method may fail to converge in nontrivial settings, see \cite{BCP23}. We note that there is no contradiction with \eqref{eq:global_convergence}, since we proved convergence for the \emph{primal} iterate $x_k$. 
We recall that in our approach Frank-Wolfe is applied to the dual problem \eqref{prob:dual_compact} and, indeed, for the corresponding iterates $y_k$ we were not able to say anything more than what is expressed in Lemma~\ref{lem:yk_opt}.

\subsection{Practical considerations}\label{sec:practical}
Having discussed the theoretical properties of Algorithm~\ref{alg:proposed_method}, we now take a look at some implementation issues that may arise.

\paragraph{Choice of stopping criteria}
There is some level of flexibility regarding the choice of stopping criterion. Typically, a maximum iteration number can be set or, for example, the  Frank-Wolfe gap can be used to stop the algorithm when it becomes too small and the $x_k$ iterates are close to being feasible. In our own implementation used in Section~\ref{sec:numerical_experiment}, we allow multiple user-controlled stopping criteria.
\paragraph{Computation of the minimum eigenvalue function and conjugate vectors}
For an arbitrary regular cone $\stdCone$ (whether a hyperbolicity cone or not), 
the minimum eigenvalue function (see \eqref{eq:min_eig_def}) required in Algorithm~\ref{alg:proposed_method} depends on $\stdCone$ and  the chosen direction $e \in \interior \stdCone$. 
If a closed form expression is not readily available, as long as a procedure to decide membership in $\stdCone$ is available, a binary search approach can be used to find $\lambda_{\mathrm{\min}}(x)$ for a given $x$.
Obtaining conjugate vectors, however, is more challenging and depends on having a good understanding of the facial structure of $\stdCone$.

In the particular case where $\stdCone$ is a hyperbolicity cone, there is more structure one can exploit in order to efficiently compute $\lambda_{\min}$ and conjugate vectors, so let us take a look at this case.
Suppose that $\stdCone = \Lambda(p,e)$, where $p: \Re^n \to \Re$ has degree $d$. By Proposition~\ref{prop:hyp_min_eig}, $\lambda_{\min}$ is the smallest root of the univariate polynomial
$t \mapsto p(x-te)$. 
Or, equivalently,  $\lambda_{\min}$ is the \emph{negative} of the \emph{largest root} of the polynomial  $p_x : \Re \to \Re$ such that 
$p_x(t) \coloneqq p(x+te)$.
In order to compute the roots of $p_x$, we first need to determine its coefficients. In theory, this could be done via its Taylor expansion, since
\begin{equation*}
p_x(t) = p(x+te) =  \sum_{i=0}^d \frac{1}{i!} p^{(i)}(x) t^i,
\end{equation*}
where we recall that $p^{(i)} = D_e^ip$. In practice, however, a naive evaluation of $p^{(i)}$ may be computationally prohibitive. 
To address this issue, we follow Renegar's suggestion in \cite[Section~9]{renegar2004hyperbolic} to evaluate the terms $p^{(i)}(x)$ using the inverse Fast Fourier Transform as follows.
 \begin{theorem}\cite[Section~9]{renegar2004hyperbolic}
	\label{thm:higher_directional_derivative}
	Let $p: \Re^n \rightarrow \mathbb{R}$ be a $d$-degree hyperbolic polynomial whose directional vector
	is $e \in \Re^n$ and $\omega$ be a primitive $d$-th root of unity. 
	Then,
	\begin{equation*}
	\frac{1}{i!}p^{(i)}(x) = \frac{1}{d} \sum_{j=1}^d \omega^{-ij}p(x+\omega^j e) \quad (i=1,\cdots,d-1)
	\end{equation*}
	holds, and hence
	\begin{align*}
	\nabla p^{(i)}(x) = \frac{i!}{d} \sum_{j=1}^d  \omega^{-ij}\nabla p(x+\omega^j e) \quad (i=1,\cdots,d-1).
	\end{align*}
\end{theorem}

\begin{remark}
At the end of Section~9 in \cite{renegar2004hyperbolic},
equations for 	$p^{(i)}(x)$ and $\nabla p^{(i)}(x) $ are given in which $i$ appears in place of $-i$ in the ``$\omega^{-ij}$'' term. 
We believe this is a typo, which can be verified by considering $p(x_1,x_2,x_3) = x_1x_2x_3$, $e = (1,1,1)$ and computing $p^{(1)}(x_1,x_2,x_3)$ which is $x_1x_2 + x_1x_3 + x_2x_3$. 
That said, the inverse of the Vandermonde matrix given in Section~9 is correct.
\end{remark}
Once the coefficients of $p_x(t)$ or $p_x(-t)$ are identified, we can numerically obtain its roots which allow us to compute the eigenvalues of $x$.
The computation of the roots of a polynomial is itself a nontrivial problem and there are a few choices on how to handle it. 
For example,  one simple approach is to form the companion matrix and compute its eigenvalues.
   
There is one extra outstanding issue regarding eigenvalue computations.
The terms $p(x+w^je)$ appearing in Theorem~\ref{thm:higher_directional_derivative} may be large even if $x$ itself is not very large\footnote{For example, for $p(x) \coloneqq x_1\cdots x_n$, $e = (1,\ldots, 1)$ and $x_0 \coloneqq (2,\ldots, 2)$, we have $p(x_0+e) = 3^n$.}. 
In order to ameliorate possible overflow issues, we may exploit the fact that $\lambda(\alpha x) = \alpha \lambda(x)$ holds for $\alpha  > 0$ and scale $x$ suitably before computing its eigenvalues. 
For example, we may divide $x$ by $\norm{x}$ if $\norm{x} > 1$ holds.

Finally, we note that Theorem~\ref{thm:higher_directional_derivative} also provides a formula for the computation of conjugate vectors as in Proposition~\ref{prop:normal_vector}.

\subsection{Minimizing a positive definite quadratic function}\label{sec:example}
In the previous subsections we discussed Algorithm~\ref{alg:proposed_method} and 
its properties. 
In this subsection, we discuss how Assumption~\ref{asm:optimal_solution_bound} is satisfied for the problem of minimizing a positive definite quadratic function under conic constraints.
More precisely, we will show that $c_D$ can be explicitly obtained from the problem data and that a dual optimal solution to \eqref{prob:dual} must 
satisfy $\inProd{e}{\OPT{y}} \leq c_D$ for some $e \in \reInt \stdCone$ as described below.
We consider the following problem
\begin{equation}\label{ex:quadratic}
  \begin{array}{ll}
    \displaystyle
    \min_{x\in\mathbb{R}^n} &f(x) = \frac{1}{2}\langle x,Qx \rangle + \langle c, x\rangle \\
    \mathrm{s.t. } & Tx+b \in \stdCone,
  \end{array}
\end{equation}
where $Q $ is a symmetric positive definite matrix (i.e., $Q \succ 0$), $\stdCone \subset \Re^m$ is a regular cone and the inner product is the usual Euclidean one. Additionally we assume that there exists $\hat{e}$ such that $T\hat{e} = e \in \reInt \stdCone$.
We can construct a feasible solution of \eqref{ex:quadratic} from $\hat{e}$ as follows. Since $T\hat{e} \in \reInt{\stdCone}$,
there exists $\epsilon > 0$ such that $T\hat{e}+\epsilon b \in \stdCone$. For this $\epsilon$, $T(\hat{e}/\epsilon)+ b \in \stdCone$,
so $\hat{e}/\epsilon$ is a feasible solution to \eqref{ex:quadratic}.

The conjugate function of $f$ is 
\begin{equation*}
f^*(x) = \frac{1}{2}\langle x-c,Q^{-1}(x-c) \rangle,
\end{equation*}
and hence,
\begin{equation} \label{eq:grad_conj_quad}
  \nabla f^*(x) = Q^{-1}(x-c).
\end{equation}

Next, we will show that
\begin{equation*}
  c_D \coloneqq \norm{\hat{e}}\sqrt\frac{\left( 2f(\hat{e}/\epsilon) + \inProd{c}{Q^{-1}c} \right)}{\lambda_{\mathrm{min}}(Q^{-1})}
\end{equation*}
satisfies Assumption \ref{asm:optimal_solution_bound}.
Let $\OPT{x}$ and $\OPT{y}$ be optimal solutions of \eqref{ex:quadratic} and the dual problem of \eqref{ex:quadratic}, respectively.
Since $\hat{e}/\epsilon$ is a feasible solution of \eqref{ex:quadratic}, 
\begin{align*}
  f(\hat{e}/\epsilon) &\geq f(\OPT{x}) \\
  &= \frac{1}{2}\inProd{\OPT{x}}{Q\OPT{x}} + \inProd{c}{\OPT{x}}.
\end{align*}
From $\OPT{x} = \nabla f^* (T^*\OPT{y})$ (see \eqref{eq:KKT_reformulated}), we have
\begin{align*}
  \frac{1}{2}\inProd{\OPT{x}}{Q\OPT{x}} + \inProd{c}{\OPT{x}}
  &= \frac{1}{2}\inProd{\nabla f^*(T^*\OPT{y})}{Q\nabla f^*(T^*\OPT{y})} + \inProd{c}{\nabla f^*(T^*\OPT{y})}\\
  &= \frac{1}{2}\inProd{Q^{-1}(T^*\OPT{y}-c)}{T^*\OPT{y}-c} + \inProd{c}{Q^{-1}(T^*\OPT{y}-c)}\\
  &= \frac{1}{2}\inProd{T^*\OPT{y}}{Q^{-1}T^*\OPT{y}}- \frac{1}{2}\inProd{c}{Q^{-1}c}.
\end{align*}
Therefore,
\begin{align*}
  &\inProd{T^*\OPT{y}}{Q^{-1}T^*\OPT{y}} \leq 2f(\hat{e}/\epsilon) + \inProd{c}{Q^{-1}c}
\end{align*}
which implies
\[ 
\norm{T^*\OPT{y}}^2 \leq \frac{\left( 2f(\hat{e}/\epsilon) + \inProd{c}{Q^{-1}c} \right)}{\lambda_{\mathrm{min}}(Q^{-1})},
\]
where $\lambda_{\mathrm{min}}(Q^{-1}) \:\:(>0)$ is the minimum eigenvalue of $Q$ as a matrix. From this inequality, $c_D$ is derived as follows.
\begin{align*}
  \inProd{e}{\OPT{y}} &= \inProd{T\hat{e}}{\OPT{y}} \\
  &= \inProd{\hat{e}}{T^*\OPT{y}} \\
  &\leq \norm{\hat{e}}\norm{T^*\OPT{y}}\\
  &\leq \norm{\hat{e}}\sqrt\frac{\left( 2f(\hat{e}/\epsilon) + \inProd{c}{Q^{-1}c} \right)}{\lambda_{\mathrm{min}}(Q^{-1})}.
\end{align*}
We note that \eqref{ex:quadratic} contains the particular case of the 
projection problem. That is, if we wish to project an arbitrary $x_0 \in \Re^m$ onto $\stdCone$, 
we may take  $c \coloneqq  -x_0$, let $T$ be the identity map and $Q$ be the identity matrix in \eqref{ex:quadratic}. In this case, it is enough to let 
$e = \hat e$ be any element in interior of $\stdCone$ and let $\epsilon \coloneqq 1$. With that, $c_D$ simplifies to 
\begin{equation}\label{eq:cd_projection}
  c_D = \norm{e}\norm{e-x_0}.
\end{equation}

\section{Numerical experiments}\label{sec:numerical_experiment}
In order to test our ideas, we wrote a MATLAB implementation of Algorithm~\ref{alg:proposed_method} and conducted two sets of experiments.
\begin{itemize}
	\item In Section~\ref{sec:exp_rel} we describe our experiments and implementation details on the problem of projecting a point onto a family of hyperbolicity cones given by elementary symmetric polynomials.
	\item In Section~\ref{sec:experiment_proj_p} we consider the problem of projecting a point onto $p$-cones, for different values of $p$.
\end{itemize}

Naturally, numerical experiments typically involve some sort of comparison.
However, there are  few methods and solvers that can handle directly the problems we discuss in this paper. 
Regarding the problems in Section~\ref{sec:exp_rel}, the most direct competitor is the accelerated gradient method (AGM) developed by Renegar in \cite{renegar2019accelerated}, which can handle conic linear programs over hyperbolicity cones. 
The DDS solver, which implements an interior point method, is  capable of handling hyperbolicity cone constraints and $p$-cone constraints. 
Mosek \cite{MC2020}, a commercial interior point method solver, can solve the problems in Section~\ref{sec:experiment_proj_p} but does not handle hyperbolicity cone constraints.
With this in mind, our goal in this section is to answer the following questions.

\begin{enumerate}[$(1)$]
	\item Is our dual Frank-Wolfe method competitive against Renegar's AGM in the case of hyperbolicity cones?
	\item Is our method ``competitive'' against IPMs?
\end{enumerate}
As we will see, the answer for question $(1)$ is a relatively straightforward ``yes'' as we found that our proposed method significantly outperforms Renegar's AGM although it should be stressed that Renegar's AGM is applicable to a more general class of problems.

The second question is more delicate. 
We wrote ``competitive'' (in quotation marks) because IPMs and first-order methods have different design goals. Generally speaking, first-order methods have a small cost per iteration. 
They struggle to get accurate solutions but they may be a good choice if the goal is to obtain solutions with low-to-medium accuracy fast. 
Conversely, IPMs seem to excel at getting accurate solutions, but often have more numerically expensive iterations.

Taking heed of this difference, both in Sections~\ref{sec:exp_rel} and \ref{sec:experiment_proj_p}, we designed experiments to check how long does it take for our method to obtain solutions that are ``somewhat close'' to the solutions obtained by IPMs.
This is consistent with the idea that the computation of a projection is often used as a subroutine in another method, so there are cases where getting a less accurate solution fast is more desirable.
Taking this into consideration, the answer to question (2)~is a qualified yes, as we will see in the results. Roughly speaking, our approach consistently obtains
solutions that are within $1\%$ to $5\%$ of the solutions obtained by IPMs in a fraction of the time.

All files can be found in the following link.

\url{https://github.com/bflourenco/dfw_projection}

All experiments were done in a PC with a Intel Xeon W-2145 CPU, 128GB RAM and Windows 10 Pro. The code was implemented in Matlab 2024b.
We used DDS version 2.0 and Mosek version 10 in our experiments.
For convenience,  terms and abbreviations used in the following discussion are summarized in Table~\ref{table:abbr}.

\begin{table}
	\centering
	\begin{tabular}{l|p{10cm}}
		 & Meaning\\ \hline \hline
		 Abs.~time (sec.) & Absolute time in seconds. \\
		AGM & Accelerated Gradient Method  \cite{renegar2019accelerated}.\\
		AGM EleSym & Accelerated Gradient Method \cite{renegar2019accelerated} using a divide-and-conquer approach to evaluate elementary symmetric polynomials and their gradients. \\ \hline
		DDS & Domain-Driven Solver \cite{karimi2019domain}. \\\hline
		Error & In the tables, this describes the target relative error with respect to the reference solver, which is DDS in Section~\ref{sec:exp_rel} and Mosek in Section~\ref{sec:experiment_proj_p}. See also explanation in Section~\ref{sec:comparison} and the discussion around \eqref{eq:crit}. \\ \hline
		FW & Algorithm~\ref{alg:proposed_method}.\\
		FW EleSym & Algorithm~\ref{alg:proposed_method} using a divide-and-conquer approach to evaluate elementary symmetric polynomials and their gradients.\\\hline
		Iterations & Average number of iterations to reach error target with respect to the reference solver. \\ \hline
		Rel.~Time & Average relative time to reach the error target with respect to the reference solver. Values close to $0$ are better. See explanation in Section~\ref{sec:comparison}. 
		\\ \hline
		S(\%) & Success rate, see explanation in Section~\ref{sec:comparison}.\\		\hline
		$\norm{x_k - x_{\text{DDS}}}_\infty$ & Average of the distance to the solution obtained by DDS. The distance is computed using the infinity norm. \\ 
		$\norm{x_k - x_{\text{Mos}}}_\infty$ & Average of the distance to the solution obtained by Mosek. The distance is computed using the infinity norm. \\ 
	\end{tabular}\caption{Abbreviations and terms used in 
	this section and in Tables~\ref{table:hyper_high_deg}-\ref{table:pcone_low}.
}\label{table:abbr}
\end{table}

\subsection{Projection onto derivative relaxations}\label{sec:exp_rel}
In this subsection our paper comes full circle and we address again the problem of 
projecting a point onto a hyperbolicity cone, this time from a numerical point of view.
We focus on hyperbolicity cones for which there are no (known)
closed form expressions in terms of the underlying eigenvalues.
Perhaps the simplest cones of this type correspond to 
the derivative relaxations of the nonnegative orthant. 
We remark that, more generally, derivative relaxations are often used to test ideas in the theory of hyperbolic polynomials and have been extensively studied, e.g., \cite{Zin08,Sa13,Bra14,SP15,Sa18,Kum21}.

In particular, the $k$-th derivative relaxation of $\Re^n_+$ (see Section~\ref{sec:hyp_p}), denoted by $\Re_+^{n,(k)}$ satisfies
\[
\Re_+^{n,(k)} = \Lambda(\sigma_{n,{n-k}},e),
\]
where $e\coloneqq 	 (1,\ldots,1)$ and $\sigma_{n,k}$ is the $k$-th elementary symmetric polynomial in $n$ variables which 
is given by 
\[
\sigma _{n,k}(x) \coloneqq \sum_{1\leq i_1<\cdots<i_k \leq n} x_{i_1}\cdots x_{i_k}.
\]
For more details on elementary symmetric polynomials and its applications to the study of hyperbolicity cones see, e.g., \cite[Section~5]{renegar2004hyperbolic} and \cite[Section~2]{BGLS01}.

In this section, our target problem is 
\begin{align}
\label{prob:proj_original}
\min_{x\in\Re^n} &\:\: \frac{1}{2}\| x-c \|^2 \\
\mathrm{s.t. } & \:\: x \in \Re_+^{n,(k)} \notag
\end{align}
where $\Re_+^{n,(k)}$ is the $k$-th derivative cone of $\mathbb{R}^n_+$ along  $e = \bm{1}_n$ and $c$ is the vector 
we wish to project onto $\Re_+^{n,(k)}$.
Since the objective function of \eqref{prob:proj_original} is a positive definite quadratic function, \eqref{prob:proj_original} is implementable as discussed in Section~\ref{sec:example}.
The derivative relaxations $\Re^{n,(k)}_+$ for $0 < k < n-3$, $n \geq 4$ are non-polyhedral and there are no known formulae for their orthogonal projections.

\paragraph{Implementation remarks on Algorithm~\ref{alg:proposed_method}}
We implemented Algorithm~\ref{alg:proposed_method} fairly straightforwardly following the discussion in Section~\ref{sec:algorithm}. 
In particular, the constant term in \eqref{prob:proj_original} is removed and we consider the equivalent problem
\begin{align}
\label{prob:proj_original_const}
\min_{x\in\Re^n} &\:\: \frac{1}{2}\norm{x}^2 - \inProd{c}{x} \\
\mathrm{s.t. } & \:\: x \in \Re_+^{n,(k)}. \notag
\end{align}
The problem \eqref{prob:proj_original_const} is regarded as the primal problem \eqref{prob:primal} and the constant $c_D$ is computed as described in \eqref{eq:cd_projection} with $e = \hat e = (1,\ldots, 1)$, $x_0 = c$.

The code for our implementation are in  the files \texttt{FW\_HP\_exp.m} and \texttt{FW\_HP.m}. 
The former is the one we actually use in the experiments and it returns all the iterates generated by the method and other useful experimental information. 
It is, however, quite memory intensive, so we also provide the file \texttt{FW\_HP.m} which only returns the best solution obtained during the course of the algorithm.
For users that wish a quick way to compute a projection onto a given  hyperbolicity cone, we also provide the file \texttt{poly\_proj.m} that is a wrapper around \texttt{FW\_HP.m} specialized for projection computations, see examples in \texttt{poly\_proj\_examples.m}.

We also implemented special functions to handle elementary symmetric polynomials and their gradients. 
Even for small $n$, the description of $\sigma _{n,k}$ 
can be quite large. 
For example, for $n = 20$, $\sigma _{10}$ is a sum of $184756$ monomials. Rather than store $\sigma_{n,k}$ as a matrix, we use an approach based on a simple divide-and-conquer algorithm to evaluate $\sigma_{n,k}$ and its gradients directly. 
The corresponding files are \texttt{eleSym.m} and \texttt{grad\_eleSym.m}.
In the numerical experiments we compare both the naive approach (i.e., storing the polynomials directly) and the implicit approach tailored for elementary symmetric polynomials.


\paragraph{Experimental and implementation remarks on Renegar's AGM and DDS}
In order to make use of Renegar's AGM and DDS we  considered the following equivalent formulations of \eqref{prob:proj_original} using an additional second-order cone constraint.
\begin{equation}\label{eq:hypcone2cone}
\eqref{prob:proj_original} \iff
\left\lbrace\begin{array}{ll}\displaystyle
\min_{x,y} & y\\
\mathrm{s.t. } &  y \geq \|x - c \| \\
& x \in \Re_+^{n,(k)} 
\end{array}\right. \iff
\left\lbrace\begin{array}{ll}\displaystyle
\min_{x,y,z} & y\\
\mathrm{s.t. } &  y \geq \|z \| \\
& z = x -c \\
& x \in \Re_+^{n,(k)} 
\end{array}\right. 
\end{equation}

Given a hyperbolicity cone and an underlying hyperbolic polynomial,
Renegar's AGM also requires the computation of  $p^{(i)}(x)$, $\nabla p^{(i)}(x)$ $(i=1,\dots,d)$ and the hyperbolic eigenvalues. In our implementation we took similar precautions as the ones discussed in Section~\ref{sec:practical}.
Additionally, Renegar's algorithm requires the computation of the following expression which corresponds to the gradient of a smoothed version of the maximum eigenvalue function:
\begin{equation}
  \label{eq:f_mu_original}
  \nabla f_\mu(x) = \frac{1}{\sum_j m_j \exp(\lambda_j(x)/\mu)}\sum_j\frac{m_j \exp(\lambda_j(x)/\mu)}{p^{(m_j)}(x-\lambda_j(x)e)}\nabla p^{(m_j-1)}(x-\lambda_j(x)e),
\end{equation}
where $\{\lambda_j(x)\}$ is the set of distinct eigenvalues of $x$ and $m_j$ is the multiplicity of $\lambda_j(x)$ and
$\mu \:(> 0)$ is a parameter determining the accuracy of the algorithm,  see \cite[Proposition 3.3]{renegar2019accelerated}. 
The smaller $\mu$ is, the smaller the error is guaranteed to be.
However, if $\mu$ is too small, there may be numerical issues if \eqref{eq:f_mu_original} is evaluated naively. 
To address this problem, we use the idea in \cite[Section 5.2]{nesterov2005smooth} in order to reformulate  \eqref{eq:f_mu_original} as \eqref{eq:f_mu_reformulated}.
\begin{equation}
  \label{eq:f_mu_reformulated}
  \nabla f_\mu(x) = \frac{1}{\sum_j m_j \exp(\lambda_j(x)-\lambda_{\max}/\mu)}\sum_j\frac{m_j \exp(\lambda_j(x)-\lambda_{\max}/\mu)}{p^{(m_j)}(x-\lambda_j(x)e)}\nabla p^{(m_j-1)}(x-\lambda_j(x)e),
\end{equation}
where $\lambda_{\max} = \max_j \{\lambda_j(x)\}$.

As in the case of our proposed method, we also adjusted the implementation of Renegar's AGM to make it possible to exploit the structure of elementary symmetric polynomials. 
Finally, we remark that Renegar's main algorithm (``MainAlgo'' in \cite{renegar2019accelerated}) prescribes that two accelerated gradient sub-methods run in parallel and, then, if at given point a certain condition is met \emph{for the iterates of the first sub-method}, \emph{both} sub-methods are stopped, a certain outer update is conducted and the sub-methods are then restarted.
Here, in order to simplify the implementation, instead of running the sub-methods in parallel, we perform one iteration of each sub-method and check if the condition for the outer update is satisfied. During the discussion of the results of the numerical experiments we will revisit this issue.

\subsubsection{A comparison between Renegar's AGM, Algorithm~\ref{alg:proposed_method} and DDS}\label{sec:comparison}
In this series of experiments we proceed as follows. 
We fix the values of $n$ and $k$ in \eqref{prob:proj_original} and then we 
generate $30$ normally distributed points in $\Re^n$. These are the $c$'s we would like to project onto $\Re_+^{n,(k)}$. 
For each generated point $c$ we check if the minimum eigenvalue of $c$ with respect to $\Re_+^{n,(k)}$ and $e = (1,\ldots, 1)$ is greater than $-10^{-4}$. If this happens, then $c$ is deemed to be too close to the cone, so we discard it and generate a new point. 

Once the $30$ points are generated, we solve the problem \eqref{prob:proj_original} with DDS, with our proposed method 
and with an implementation of Renegar's AGM. 
For Algorithm~\ref{alg:proposed_method} and Renegar's AGM we also considered variants that use code specialized to elementary symmetric polynomials. 
So, in total, each of the $30$ instances is solved through $5$ different methods, which will be, henceforth denoted by ``DDS'', ``FW'', ``FW EleSym'', ``AGM'' and ``AGM EleSym''.
Here, we recall that FW and AGM correspond to Algorithm~\ref{alg:proposed_method} and Renegar's accelerated gradient method, respectively. 
``EleSym'' indicates the usage of special methods to handle elementary symmetric polynomials as discussed previously.

We consider the objective function value obtained by DDS as the baseline to which we compare the performance of the other algorithms. 
The results are described in Tables~\ref{table:hyper_high_deg} and \ref{table:hyper_low_deg}. We now explain the meaning of the data.
For example, consider the first line in Table~\ref{table:hyper_hig_deg_10_1}, so that the ``Error'' column indicates ``$10\%$''.
Roughly speaking, for this line, we checked how much time does each one of the 4 tested methods need to get a solution that has a value that is within $10\%$ 
of the function valued obtained by DDS.

More formally, for each instance $i$, denote by $f_{\text{DDS}}^i$ and $t_{\text{DDS}}^i$  the objective function value obtained by DDS and the corresponding running time, respectively. 
Analogously, denote by $f_{\text{FW}}^{i,j}$ the function value obtained by Algorithm~\ref{alg:proposed_method}  for the $i$-th instance at the $j$-th iteration. 
Denote by  $t_{\text{FW}}^{i,j}$ the time elapsed after the $j$-th iteration. 
For each instance $k$ and a given error tolerance $E$  (e.g., $10\%$), we checked the amount of time that the FW method
required to reach an iteration $j$ for which 
\begin{equation}\label{eq:crit}
f_{\text{FW}}^{i,j} \leq f_{\text{DDS}}^i \times \left(\frac{100+E}{100}\right)
\end{equation}
holds and the minimal eigenvalue of corresponding primal iterate (the $x_k$ in Algorithm~\ref{alg:proposed_method}) is at least $-10^{-8}$ (i.e., $x_k$ is sufficiently close to being feasible). 
 Then, we record the ratio 
$\frac{t_{\text{FW}}^{i,j}}{t_{\text{DDS}}^i}$ and in the column ``Rel.~time'' we register the average of these ratios together with their standard deviation.
This average is what we henceforth call the \emph{mean relative time}.
In the ``S'' (for \emph{Success}) column, we indicate the percentage of instances for which the method was able to find a solution that satisfies  \eqref{eq:crit} within the feasibility requirements. 
The mean relative time is computed using successful instances only.

For the other methods FW EleSym, AGM and AGM EleSym we proceed similarly.
Although each algorithm uses a different equivalent formulation for the problem \eqref{prob:proj_original}, objective function value comparisons are always done with respect to the objective function of \eqref{prob:proj_original}.
For each instance, for all methods except DDS, we 
set the maximum running time to be equal to the time spent by DDS. 
The rationale is that it does not make sense to run a first order method longer than the time required by an IPM for the same problem. 

For example, for the first line in Table~\ref{table:hyper_hig_deg_10_1} (which corresponds to $n = 10$, $k =1$), the entry 
``$0.25 \pm 0.08$'' under ``FW'' means that, on average,  our proposed method was able to find a solution whose objective value is within $10\%$ of the value found by DDS using $0.25\%$ of the time that DDS needed to find $f_{\text{DDS}}^i$.
The corresponding standard deviation was $0.08$. It also succeeds for all $30$ instances, i.e., for each one of the instances there was at least one iteration that satisfied \eqref{eq:crit} with $E=10$.
The data in the other columns ``FW EleSym'', ``AGM'' and ``AGM EleSym'' have analogous meaning.

As we go down Table~\ref{table:hyper_hig_deg_10_1}, the mean relative time increases and success rate decreases. 
As the error $E$ decreases, it becomes harder to approach the values obtained by DDS within the allowed time budget.
 Still,  we believe it is notable that for $E = 1\%$, FW is able to get solutions whose values are within $1\%$ of the value obtained by DDS using less than $1\%$ of DDS's running time. 
 For the FW EleSym code, which is the variant optimized for elementary symmetric polynomials, we were able to get even more mileage with $100\%$ success for $E = 0.5\%$ and mean relative time of less than $2\%$.
 
 Overall, our impression is that a bottleneck in the 4 methods is the computation of minimum eigenvalues and, in the case of our proposed method, the computation of conjugate vectors too. 
 Both  are heavily influenced by the degree of the underlying hyperbolic polynomial. 
 Indeed, 
 the results for FW and FW EleSym for $n = 10$ and $k = 2$ (Table~\ref{table:hyper_hig_deg_10_2}) seem slightly better than the ones for 
 $n = 10$ and $k = 1$ (Table~\ref{table:hyper_hig_deg_10_1}) in the sense that the success rates were higher.
 
 In contrast for $n = 20$ and $k \in \{1,2\}$, we have polynomials of degrees 19 and 18 respectively. In those cases, the performance of FW, AGM and AGM EleSym plummet and these three methods seem to struggle to even get low accuracy solutions. 
 However, FW EleSym is still competitive and is able to get high success rates up to $E = 0.1$ with reasonable mean relative times. 
 
 In Table~\ref{table:hyper_low_deg}, we have the results for $(n,k) \in \{(30,27),(40,37),(50,47)\}$. In all those cases, the degree of the hyperbolic polynomial is just three.  Both FW and FW EleSym had particularly strong performances, which leads further credence to the idea that the degree is an important factor.
 In the case of FW EleSym, we were able to consistently get within $0.05\%$ or less of the objective value obtained by DDS with just a small fraction of the required running time. 
 For example, for $n = 50, k = 47$ (Table~\ref{table:hyper_low_deg_50_47}), on average, we needed no more than $0.04\%$ of the running time of DDS in order to get within 
 $0.005\%$ of the objective value.
 In this case, $\sigma_{n,n-k}$ has $19600$ monomials, so using routines specialized to elementary symmetric polynomials leads to a boost in performance. 
 We believe that is why in this set of experiments, FW EleSym was better than FW. Similarly, AGM EleSym had a superior performance when compared to the pure AGM.

 In Tables~\ref{table:hyper_high_deg} and \ref{table:hyper_low_deg}, we configured DDS to run with the default stopping criterion tolerance of $10^{-8}$. 
 This means that DDS  is actively looking for relatively high-accuracy solutions.
 One may then reasonably wonder what would happen if we configure DDS to run with a lower accuracy.
 To address that, we considered the same exact experiment but with the DDS stopping tolerance set to $10^{-3}$.
 
 When DDS runs with a lower tolerance there are two opposing effects that appear. 
 On one hand, the ratio of the running times tend to increase, since the denominator (i.e., the DDS running time) decreases as DDS stops earlier.  
 On the other hand,  since the solutions obtained by the DDS are less accurate, intuition would suggest that it would be easier to approach the solutions obtained by DDS using a first-order method. 
 This would mean  that the numerator of the mean relative times would get smaller.
 The former effect should lead to ``worse'' results and the latter effect should lead to ``better'' (i.e., higher success rates and/or decreased mean relative times) results.

 The results are described in Tables~\ref{table:hyper_high_deg_low_acc} and \ref{table:hyper_low_deg_low_acc}. 
 Overall, the results were largely similar to the ones in Tables~\ref{table:hyper_high_deg} and \ref{table:hyper_low_deg} and seem to allow for similar conclusions. 
 A notable difference is that, indeed, for some choices of $n,k$ the success rate of FW is higher than in the case where DDS is run with high accuracy. For example, for $n = 10$ and $k = 1$ (Table~\ref{table:hyper_high_deg_low_acc_10_1}), we can see that the success rate of FW stays above $60\%$ throughout the 30 instances, although, naturally, the mean relative times increase accordingly.  
 For $n = 20, k = 1$, FW EleSym was able to get more than $80\%$ success rate up until $E  = 0.001\%$ with reasonable mean relative times, see Table~\ref{table:hyper_high_deg_low_acc_20_1}. 
 In contrast, in the high-accuracy setting (Table~\ref{table:hyper_hig_deg_20_1}),  $E = 0.05\%$ seems to be the best we could obtain with success rate above $80\%$.
 In the setting of Table~\ref{table:hyper_high_deg_low_acc}, the fact that the solution obtained by DDS are easier to reach seems to be the stronger factor here.
 
In the case where the polynomials are of smaller degree (Table~\ref{table:hyper_low_deg_low_acc}), the fact that DDS stops faster seems to be the preponderant effect, as the success rates for FW and FW EleSym are no longer $100\%$  even at $E = 1\%$. 
Nevertheless, they still stay above $90\%$ up to $E = 0.5\%$ with quite reasonable mean relative times. Overall, for FW EleSym we still need less than $0.3\%$ of the running time of DDS in order to find solutions whose values are within $0.5\%$ of the objective value found by DDS.

Taking Tables~\ref{table:hyper_high_deg}--\ref{table:hyper_low_deg_low_acc} in consideration, our conclusion is that: $(a)$ in most cases Algorithm~\ref{alg:proposed_method} indeed succeeds in getting low-to-medium accuracy solutions in a reasonable time; $(b)$ Algorithm~\ref{alg:proposed_method} seems to be faster than Renegar's AGM; $(c)$ when the hyperbolic polynomial has many monomials and/or is of higher degree, FW EleSym tends to be significantly better than FW.
A caveat is that, as mentioned previously, our implementation of Renegar's AGM is sequential rather than parallel, but even taking into consideration a parallel speed-up factor of 2, our approach still seems to outperform Renegar's AGM by a large margin.

On the other hand, it seems that Algorithm~\ref{alg:proposed_method} is indeed quite sensitive to the precision of eigenvalue and conjugate vector computations. 
In particular, for high degree hyperbolic polynomials  careful  implementation of  the eigenvalue computation routines is important. 

Finally, for the experiments described in Tables~\ref{table:hyper_high_deg} and \ref{table:hyper_low_deg} we provide extra statistics about the FW and FW EleSym approaches, see Tables~\ref{table:hyper_high_deg_det} and \ref{table:hyper_low_deg_det} respectively.
In those tables, we provide the average number of iterations and the absolute time required to reach each error target.
Also, in the column ``$\norm{x_k - x_{\text{DDS}}}_\infty$'' we  provide the average of the distances (in the infinity norm) to the solutions obtained by the DDS solver.
For example, in Table~\hyperref[table:hyper_hig_deg_10_1_det]{\ref*{table:hyper_high_deg_det}\subref*{table:hyper_hig_deg_10_1_det}}, in the block associated to FW, we have ``$7.7\pm2.9$'' in the  first row  of the ``Iterations'' column. 
This indicates that considering all the $30$ points generated in Table~\ref{table:hyper_hig_deg_10_1}, FW required an average of $7.7$ iterations to reach a solution with $E = 10\%$ and the standard deviation was $2.9$. The average absolute time in seconds to obtain such a solution was $2.49\times 10^{-3}$ and the standard deviation was   $8.14\times 10^{-4}$. 
Here we use ``absolute time'' to distinguish from the relative time
described previously.
Measuring the distance to the solution obtained by the DDS solver we obtained an average of $2.58 \times 10^{-2}$ and standard deviation of $1.92 \times 10^{-2}$.

As the error threshold decreases, the number of iterations and the absolute time both increase, as expected. 
The most notable phenomenon is that for polynomials of small degree a very small number of iterations is enough to get solutions with $E=0.1\%$ or less, as indicated
in Table~\ref{table:hyper_low_deg_det}.
At this moment, we are not aware of a deeper theoretical explanation for that.

\begin{table}\small
\begin{subtable}{\textwidth}
\begin{filecontents*}{experiments/proj_hyper_n10_d1_30_tol_high.csv}
error,FWx,FWxstd,FWit,FWitstd,FWabs,FWabsstd,FW,FWstd,FWsuccess,FWElx,FWElxstd,FWElit,FWElitstd,FWElabs,FWElabsstd,FWEl,FWElstd,FWElsuccess,AGM,AGMstd,AGMsuccess,AGMEl,AGMElstd,AGMElsuccess
10,2.58e-02,1.92e-02,7.66667,2.92826,2.49e-03,8.14e-04,0.246279,0.0798749,100,2.58e-02,1.92e-02,7.66667,2.92826,4.92e-03,2.10e-03,0.484122,0.200443,100,52.0792,28.4263,56.6667,98.1182,8.93955,23.3333
5,2.45e-02,1.94e-02,7.8,2.77178,2.55e-03,9.27e-04,0.252049,0.0910607,100,2.45e-02,1.94e-02,7.8,2.77178,4.97e-03,2.03e-03,0.489347,0.19376,100,72.8895,24.7874,46.6667,NaN,NaN,0
1,1.50e-02,9.47e-03,11.2333,5.8113,3.37e-03,1.58e-03,0.334283,0.156648,100,1.50e-02,9.47e-03,11.2333,5.8113,6.72e-03,3.45e-03,0.662368,0.334091,100,93.3906,0,3.33333,NaN,NaN,0
0.5,1.26e-02,9.18e-03,18.7333,15.3487,5.35e-03,3.82e-03,0.557384,0.492673,100,1.26e-02,9.18e-03,18.6,15.3322,1.03e-02,7.39e-03,1.06193,0.968908,100,NaN,NaN,0,NaN,NaN,0
0.1,6.44e-03,4.77e-03,437.222,340.39,9.74e-02,7.22e-02,9.4916,6.8538,90,6.09e-03,4.60e-03,399.033,345.376,1.98e-01,1.72e-01,19.1794,15.8699,100,NaN,NaN,0,NaN,NaN,0
0.05,3.73e-03,3.20e-03,1275.54,657.04,2.88e-01,1.38e-01,28.1416,12.9961,86.6667,3.75e-03,2.89e-03,1029.33,619.448,5.15e-01,3.05e-01,51.2435,28.6172,90,NaN,NaN,0,NaN,NaN,0
0.01,1.69e-04,0.00e+00,2140,0,4.28e-01,0.00e+00,42.6169,0,3.33333,1.69e-04,0.00e+00,2140,0,8.23e-01,0.00e+00,81.9811,0,3.33333,NaN,NaN,0,NaN,NaN,0
0.005,8.70e-05,0.00e+00,4320,0,8.59e-01,0.00e+00,85.5422,0,3.33333,NaN,NaN,NaN,NaN,NaN,NaN,NaN,NaN,0,NaN,NaN,0,NaN,NaN,0
0.001,NaN,NaN,NaN,NaN,NaN,NaN,NaN,NaN,0,NaN,NaN,NaN,NaN,NaN,NaN,NaN,NaN,0,NaN,NaN,0,NaN,NaN,0

\end{filecontents*}
\csvreader[hyperExp]{experiments/proj_hyper_n10_d1_30_tol_high.csv}
{}{\printExpDataLine }
\caption{$n=10$, $k=1$}\label{table:hyper_hig_deg_10_1}
\end{subtable}
\begin{subtable}{\textwidth}
	{\begin{filecontents*}{experiments/proj_hyper_n10_d2_30_tol_high.csv}
error,FWx,FWxstd,FWit,FWitstd,FWabs,FWabsstd,FW,FWstd,FWsuccess,FWElx,FWElxstd,FWElit,FWElitstd,FWElabs,FWElabsstd,FWEl,FWElstd,FWElsuccess,AGM,AGMstd,AGMsuccess,AGMEl,AGMElstd,AGMElsuccess
10,5.79e-02,2.92e-02,5.33333,1.29544,3.41e-03,5.38e-04,0.273174,0.0465845,100,5.79e-02,2.92e-02,5.33333,1.29544,3.17e-03,8.72e-04,0.253956,0.0709806,100,42.3043,19.381,90,70.5856,9.38624,73.3333
5,4.40e-02,1.70e-02,6.13333,0.937102,3.65e-03,4.11e-04,0.293602,0.0428946,100,4.40e-02,1.70e-02,6.13333,0.937102,3.48e-03,6.42e-04,0.279585,0.0560586,100,50.8756,22.7274,56.6667,80.6609,10.9922,36.6667
1,2.57e-02,1.55e-02,18.0333,11.1525,7.19e-03,3.21e-03,0.584851,0.27414,100,2.57e-02,1.55e-02,18.0333,11.1525,8.20e-03,4.16e-03,0.668673,0.353509,100,63.1491,17.1687,13.3333,NaN,NaN,0
0.5,1.93e-02,1.41e-02,52.1667,39.8602,1.74e-02,1.14e-02,1.39592,0.90758,100,1.93e-02,1.41e-02,52.1667,39.8602,2.20e-02,1.50e-02,1.76434,1.1879,100,76.0023,12.6012,6.66667,NaN,NaN,0
0.1,4.86e-03,3.30e-03,644.033,261.706,1.96e-01,8.01e-02,15.6537,6.13324,100,4.87e-03,3.32e-03,642.167,261.273,2.81e-01,1.32e-01,22.3804,10.1049,100,NaN,NaN,0,NaN,NaN,0
0.05,2.57e-03,1.80e-03,1457.21,517.427,4.40e-01,1.62e-01,34.8553,12.1291,96.6667,2.37e-03,1.65e-03,1465.59,457.005,6.45e-01,2.51e-01,51.582,19.9751,96.6667,NaN,NaN,0,NaN,NaN,0
0.01,7.30e-05,3.39e-05,2748,649.124,7.86e-01,1.80e-01,59.9334,16.2991,6.66667,7.30e-05,3.39e-05,2748,649.124,1.05e+00,1.94e-01,79.8568,18.2072,6.66667,NaN,NaN,0,NaN,NaN,0
0.005,6.34e-05,0.00e+00,4631,0,1.33e+00,0.00e+00,97.459,0,3.33333,NaN,NaN,NaN,NaN,NaN,NaN,NaN,NaN,0,NaN,NaN,0,NaN,NaN,0
0.001,NaN,NaN,NaN,NaN,NaN,NaN,NaN,NaN,0,NaN,NaN,NaN,NaN,NaN,NaN,NaN,NaN,0,NaN,NaN,0,NaN,NaN,0

\end{filecontents*}
\csvreader[hyperExp]{experiments/proj_hyper_n10_d2_30_tol_high.csv}
		{}{\printExpDataLine
		}
	}\caption{$n=10$, $k=2$}\label{table:hyper_hig_deg_10_2}
\end{subtable}
\begin{subtable}{\textwidth}
	{\begin{filecontents*}{experiments/proj_hyper_n20_d1_30_tol_high.csv}
error,FWx,FWxstd,FWit,FWitstd,FWabs,FWabsstd,FW,FWstd,FWsuccess,FWElx,FWElxstd,FWElit,FWElitstd,FWElabs,FWElabsstd,FWEl,FWElstd,FWElsuccess,AGM,AGMstd,AGMsuccess,AGMEl,AGMElstd,AGMElsuccess
10,2.28e-02,0.00e+00,11,0,1.64e-02,0.00e+00,0.379741,0,3.33333,4.06e-02,2.31e-02,17.6,6.71899,4.16e-02,1.66e-02,0.95906,0.406564,100,NaN,NaN,0,NaN,NaN,0
5,2.28e-02,0.00e+00,11,0,1.64e-02,0.00e+00,0.379741,0,3.33333,3.95e-02,2.30e-02,17.8333,6.53417,4.20e-02,1.63e-02,0.969604,0.400458,100,NaN,NaN,0,NaN,NaN,0
1,9.82e-03,0.00e+00,18,0,3.45e-02,0.00e+00,0.799114,0,3.33333,2.69e-02,1.39e-02,26.2333,10.5885,5.93e-02,2.32e-02,1.38529,0.620204,100,NaN,NaN,0,NaN,NaN,0
0.5,NaN,NaN,NaN,NaN,NaN,NaN,NaN,NaN,0,1.81e-02,1.00e-02,40.9333,18.4184,9.03e-02,4.16e-02,2.12011,1.09602,100,NaN,NaN,0,NaN,NaN,0
0.1,NaN,NaN,NaN,NaN,NaN,NaN,NaN,NaN,0,6.15e-03,2.19e-03,166.433,138.91,3.46e-01,2.65e-01,7.95513,5.93727,100,NaN,NaN,0,NaN,NaN,0
0.05,NaN,NaN,NaN,NaN,NaN,NaN,NaN,NaN,0,4.44e-03,1.73e-03,545.31,497.302,1.13e+00,9.92e-01,25.6005,22.2241,96.6667,NaN,NaN,0,NaN,NaN,0
0.01,NaN,NaN,NaN,NaN,NaN,NaN,NaN,NaN,0,NaN,NaN,NaN,NaN,NaN,NaN,NaN,NaN,0,NaN,NaN,0,NaN,NaN,0
0.005,NaN,NaN,NaN,NaN,NaN,NaN,NaN,NaN,0,NaN,NaN,NaN,NaN,NaN,NaN,NaN,NaN,0,NaN,NaN,0,NaN,NaN,0
0.001,NaN,NaN,NaN,NaN,NaN,NaN,NaN,NaN,0,NaN,NaN,NaN,NaN,NaN,NaN,NaN,NaN,0,NaN,NaN,0,NaN,NaN,0

\end{filecontents*}
\csvreader[hyperExp]{experiments/proj_hyper_n20_d1_30_tol_high.csv}
		{}{\printExpDataLine
		}
	}\caption{$n=20$, $k=1$}\label{table:hyper_hig_deg_20_1}
\end{subtable}
\begin{subtable}{\textwidth}
	{\begin{filecontents*}{experiments/proj_hyper_n20_d2_30_tol_high.csv}
error,FWx,FWxstd,FWit,FWitstd,FWabs,FWabsstd,FW,FWstd,FWsuccess,FWElx,FWElxstd,FWElit,FWElitstd,FWElabs,FWElabsstd,FWEl,FWElstd,FWElsuccess,AGM,AGMstd,AGMsuccess,AGMEl,AGMElstd,AGMElsuccess
10,4.89e-02,1.42e-02,8,1.22474,2.73e-02,2.28e-03,0.1878,0.00637637,16.6667,6.14e-02,3.32e-02,12.3667,3.44897,2.78e-02,8.22e-03,0.180683,0.0566776,100,NaN,NaN,0,NaN,NaN,0
5,4.05e-02,6.79e-03,9.5,1.91485,3.01e-02,3.87e-03,0.206258,0.035675,13.3333,5.21e-02,2.52e-02,13.4667,3.29821,2.99e-02,7.78e-03,0.195206,0.0575714,100,NaN,NaN,0,NaN,NaN,0
1,3.54e-02,5.82e-03,10.5,3.53553,3.19e-02,6.74e-03,0.208645,0.0354982,6.66667,2.58e-02,1.11e-02,24.2,10.0735,5.08e-02,2.09e-02,0.333941,0.154166,100,NaN,NaN,0,NaN,NaN,0
0.5,3.32e-02,2.63e-03,21,18.3848,6.81e-02,5.79e-02,0.439021,0.361299,6.66667,1.81e-02,7.06e-03,35.5667,14.9751,7.27e-02,3.08e-02,0.479348,0.222694,100,NaN,NaN,0,NaN,NaN,0
0.1,2.65e-03,0.00e+00,347,0,1.05e+00,0.00e+00,7.06555,0,3.33333,8.39e-03,3.07e-03,431.2,329.433,8.39e-01,6.10e-01,5.58211,4.2206,100,NaN,NaN,0,NaN,NaN,0
0.05,NaN,NaN,NaN,NaN,NaN,NaN,NaN,NaN,0,5.03e-03,1.93e-03,1518.27,680.344,3.02e+00,1.35e+00,19.9977,9.88767,100,NaN,NaN,0,NaN,NaN,0
0.01,NaN,NaN,NaN,NaN,NaN,NaN,NaN,NaN,0,1.42e-03,1.01e-03,7802,2132.63,1.44e+01,6.65e+00,84.0576,17.3628,6.66667,NaN,NaN,0,NaN,NaN,0
0.005,NaN,NaN,NaN,NaN,NaN,NaN,NaN,NaN,0,NaN,NaN,NaN,NaN,NaN,NaN,NaN,NaN,0,NaN,NaN,0,NaN,NaN,0
0.001,NaN,NaN,NaN,NaN,NaN,NaN,NaN,NaN,0,NaN,NaN,NaN,NaN,NaN,NaN,NaN,NaN,0,NaN,NaN,0,NaN,NaN,0

\end{filecontents*}
\csvreader[hyperExp]{experiments/proj_hyper_n20_d2_30_tol_high.csv}
		{}{\printExpDataLine
		}
	}\caption{$n=20$, $k=2$}\label{table:hyper_hig_deg_20_2}
\end{subtable}	
\caption{All experiments were done with $30$ randomly generated points. The polynomials in the experiments described here have degrees $9,8,19,18$, respectively. 
A bold entry in a row indicates the method with best mean relative time among the ones that were $100\%$ successful. Averages and standard deviations are computed using successful points only.
More detailed experimental data regarding FW and  FW EleSym is given in Table~\ref{table:hyper_high_deg_det}.}\label{table:hyper_high_deg}		
\end{table}



\begin{table}\footnotesize
\begin{DIFnomarkup}
	\setcounter{subtable}{0}
	\begin{subtable}{\textwidth}
		\csvreader[hyperExpDet]{experiments/proj_hyper_n10_d1_30_tol_high.csv}
		{}{\printExpDataLineDet }
		\caption{$n=10$, $k=1$}\label{table:hyper_hig_deg_10_1_det}
	\end{subtable}
	\begin{subtable}{\textwidth}
		{\csvreader[hyperExpDet]{experiments/proj_hyper_n10_d2_30_tol_high.csv}
			{}{\printExpDataLineDet
			}
		}\caption{$n=10$, $k=2$}\label{table:hyper_hig_deg_10_2_det}
	\end{subtable}
	\begin{subtable}{\textwidth}
		{\csvreader[hyperExpDet]{experiments/proj_hyper_n20_d1_30_tol_high.csv}
			{}{\printExpDataLineDet
			}
		}\caption{$n=20$, $k=1$}\label{table:hyper_hig_deg_20_1_det}
	\end{subtable}
	\begin{subtable}{\textwidth}
		{\csvreader[hyperExpDet]{experiments/proj_hyper_n20_d2_30_tol_high.csv}
			{}{\printExpDataLineDet
			}
		}\caption{$n=20$, $k=2$}\label{table:hyper_hig_deg_20_2_det}
	\end{subtable}
\end{DIFnomarkup}	
	\caption{Additional data regarding the experiments in Table~\ref{table:hyper_high_deg} and the setting is the same as in Table~\ref{table:hyper_high_deg}. 
	Rows for which the success rate of both FW and FW EleSym are less than $15\%$ are omitted. Greyed out entries correspond to the cases where success rates were less than $90\%$.
	 All averages and standard deviations are computed using successful instances only.  }\label{table:hyper_high_deg_det}		
\end{table}

\begin{table}\small
\begin{subtable}{\textwidth}
	\begin{filecontents*}{experiments/proj_hyper_n30_d27_30_tol_high.csv}
error,FWx,FWxstd,FWit,FWitstd,FWabs,FWabsstd,FW,FWstd,FWsuccess,FWElx,FWElxstd,FWElit,FWElitstd,FWElabs,FWElabsstd,FWEl,FWElstd,FWElsuccess,AGM,AGMstd,AGMsuccess,AGMEl,AGMElstd,AGMElsuccess
10,5.78e-03,5.72e-03,3,0,1.35e-01,2.63e-03,0.744116,0.116124,100,5.78e-03,5.72e-03,3,0,1.27e-02,8.04e-04,0.0705413,0.0121513,100,29.2751,14.1498,100,5.9313,2.61593,100
5,5.78e-03,5.72e-03,3,0,1.35e-01,2.63e-03,0.744116,0.116124,100,5.78e-03,5.72e-03,3,0,1.27e-02,8.04e-04,0.0705413,0.0121513,100,35.3457,14.3991,100,7.02221,2.6461,100
1,5.78e-03,5.72e-03,3,0,1.35e-01,2.63e-03,0.744116,0.116124,100,5.78e-03,5.72e-03,3,0,1.27e-02,8.04e-04,0.0705413,0.0121513,100,44.6704,12.0589,93.3333,11.747,13.1499,100
0.5,5.78e-03,5.72e-03,3,0,1.35e-01,2.63e-03,0.744116,0.116124,100,5.78e-03,5.72e-03,3,0,1.27e-02,8.04e-04,0.0705413,0.0121513,100,46.9856,12.7713,93.3333,12.1379,13.0758,100
0.1,5.78e-03,5.72e-03,3,0,1.35e-01,2.63e-03,0.744116,0.116124,100,5.78e-03,5.72e-03,3,0,1.27e-02,8.04e-04,0.0705413,0.0121513,100,61.308,16.2297,63.3333,26.8733,25.6851,100
0.05,5.61e-03,5.15e-03,14.6,63.5358,2.19e-01,4.62e-01,1.13036,2.09408,100,5.61e-03,5.15e-03,14.6,63.5358,4.60e-02,1.82e-01,0.222709,0.831104,100,68.1174,15.5175,50,36.2732,30.9533,96.6667
0.01,3.54e-03,1.97e-03,174.44,504.562,1.39e+00,3.71e+00,7.63353,19.3809,83.3333,4.06e-03,2.22e-03,870.333,1732.87,2.51e+00,4.98e+00,12.1227,23.5003,100,NaN,NaN,0,65.3412,23.1373,63.3333
0.005,2.09e-03,1.03e-03,362.714,824.772,2.77e+00,6.02e+00,13.1794,26.8548,46.6667,2.51e-03,1.11e-03,1477.3,1969.36,4.30e+00,5.70e+00,23.9591,33.0376,66.6667,NaN,NaN,0,70.2264,22.3292,46.6667
0.001,3.05e-04,4.24e-04,358,502.046,2.78e+00,3.74e+00,16.6222,22.1498,6.66667,4.19e-04,3.59e-04,1749.67,2436.44,5.07e+00,7.07e+00,29.4844,40.8832,10,NaN,NaN,0,73.2417,20.875,33.3333

\end{filecontents*}
\csvreader[hyperExp]{experiments/proj_hyper_n30_d27_30_tol_high.csv}
	{}{\printExpDataLine
	}
	\caption{$n=30$, $k=27$}\label{table:hyper_low_deg_30_27}
\end{subtable}
\begin{subtable}{\textwidth}
\begin{filecontents*}{experiments/proj_hyper_n40_d37_30_tol_high.csv}
error,FWx,FWxstd,FWit,FWitstd,FWabs,FWabsstd,FW,FWstd,FWsuccess,FWElx,FWElxstd,FWElit,FWElitstd,FWElabs,FWElabsstd,FWEl,FWElstd,FWElsuccess,AGM,AGMstd,AGMsuccess,AGMEl,AGMElstd,AGMElsuccess
10,2.70e-03,2.11e-03,3,0,4.69e-01,1.12e-02,0.886906,0.14539,100,2.70e-03,2.11e-03,3,0,2.31e-02,6.47e-04,0.0436854,0.00737931,100,29.044,4.91303,96.6667,5.57214,10.6805,100
5,2.70e-03,2.11e-03,3,0,4.69e-01,1.12e-02,0.886906,0.14539,100,2.70e-03,2.11e-03,3,0,2.31e-02,6.47e-04,0.0436854,0.00737931,100,34.1619,6.10439,96.6667,6.12312,10.5838,100
1,2.70e-03,2.11e-03,3,0,4.69e-01,1.12e-02,0.886906,0.14539,100,2.70e-03,2.11e-03,3,0,2.31e-02,6.47e-04,0.0436854,0.00737931,100,52.4085,18.1507,96.6667,8.15202,10.638,100
0.5,2.70e-03,2.11e-03,3,0,4.69e-01,1.12e-02,0.886906,0.14539,100,2.70e-03,2.11e-03,3,0,2.31e-02,6.47e-04,0.0436854,0.00737931,100,53.7216,15.3514,90,10.7454,16.737,100
0.1,2.70e-03,2.11e-03,3,0,4.69e-01,1.12e-02,0.886906,0.14539,100,2.70e-03,2.11e-03,3,0,2.31e-02,6.47e-04,0.0436854,0.00737931,100,67.2277,17.9404,73.3333,14.4263,14.059,96.6667
0.05,2.70e-03,2.11e-03,3,0,4.69e-01,1.12e-02,0.886906,0.14539,100,2.70e-03,2.11e-03,3,0,2.31e-02,6.47e-04,0.0436854,0.00737931,100,65.3028,15.4573,40,21.8035,15.9925,96.6667
0.01,2.70e-03,2.11e-03,3,0,4.69e-01,1.12e-02,0.886906,0.14539,100,2.70e-03,2.11e-03,3,0,2.31e-02,6.47e-04,0.0436854,0.00737931,100,68.3194,10.3555,10,35.7584,20.8074,93.3333
0.005,2.23e-03,1.83e-03,57.3077,230.383,2.32e+00,7.85e+00,4.71661,16.3915,86.6667,2.44e-03,1.79e-03,752.8,1894.52,4.05e+00,1.02e+01,7.11889,17.8721,100,NaN,NaN,0,50.7322,20.9792,90
0.001,3.93e-04,3.32e-04,3,0,4.73e-01,9.78e-03,0.861279,0.172599,30,4.45e-04,3.55e-04,505.4,1588.73,2.74e+00,8.58e+00,5.68962,17.8598,33.3333,NaN,NaN,0,65.3726,22.741,60

\end{filecontents*}
\csvreader[hyperExp]{experiments/proj_hyper_n40_d37_30_tol_high.csv}
{}{\printExpDataLine
}
\caption{$n=40$, $k=37$}\label{table:hyper_low_deg_40_37}
\end{subtable}
\begin{subtable}{\textwidth}
	\begin{filecontents*}{experiments/proj_hyper_n50_d47_30_tol_high.csv}
error,FWx,FWxstd,FWit,FWitstd,FWabs,FWabsstd,FW,FWstd,FWsuccess,FWElx,FWElxstd,FWElit,FWElitstd,FWElabs,FWElabsstd,FWEl,FWElstd,FWElsuccess,AGM,AGMstd,AGMsuccess,AGMEl,AGMElstd,AGMElsuccess
10,1.79e-03,1.48e-03,3,0,1.25e+00,4.44e-02,1.37198,0.178168,100,1.79e-03,1.48e-03,3,0,3.75e-02,1.19e-03,0.0410283,0.00554729,100,47.4358,7.54428,100,3.58871,0.564212,100
5,1.79e-03,1.48e-03,3,0,1.25e+00,4.44e-02,1.37198,0.178168,100,1.79e-03,1.48e-03,3,0,3.75e-02,1.19e-03,0.0410283,0.00554729,100,55.823,15.7436,93.3333,6.76701,10.3515,100
1,1.79e-03,1.48e-03,3,0,1.25e+00,4.44e-02,1.37198,0.178168,100,1.79e-03,1.48e-03,3,0,3.75e-02,1.19e-03,0.0410283,0.00554729,100,69.3626,17.2009,83.3333,10.0871,12.8841,100
0.5,1.79e-03,1.48e-03,3,0,1.25e+00,4.44e-02,1.37198,0.178168,100,1.79e-03,1.48e-03,3,0,3.75e-02,1.19e-03,0.0410283,0.00554729,100,76.4254,17.0834,66.6667,12.1603,13.2901,100
0.1,1.79e-03,1.48e-03,3,0,1.25e+00,4.44e-02,1.37198,0.178168,100,1.79e-03,1.48e-03,3,0,3.75e-02,1.19e-03,0.0410283,0.00554729,100,75.0221,21.6316,46.6667,17.328,17.8112,100
0.05,1.79e-03,1.48e-03,3,0,1.25e+00,4.44e-02,1.37198,0.178168,100,1.79e-03,1.48e-03,3,0,3.75e-02,1.19e-03,0.0410283,0.00554729,100,85.8435,22.4406,36.6667,22.5949,19.5098,100
0.01,1.79e-03,1.48e-03,3,0,1.25e+00,4.44e-02,1.37198,0.178168,100,1.79e-03,1.48e-03,3,0,3.75e-02,1.19e-03,0.0410283,0.00554729,100,89.1729,10.0205,6.66667,47.3715,21.7521,93.3333
0.005,1.79e-03,1.48e-03,3,0,1.25e+00,4.44e-02,1.37198,0.178168,100,1.79e-03,1.48e-03,3,0,3.75e-02,1.19e-03,0.0410283,0.00554729,100,NaN,NaN,0,52.5724,17.5979,93.3333
0.001,6.12e-04,4.45e-04,3,0,1.26e+00,4.59e-02,1.39831,0.191452,50,6.65e-04,4.79e-04,93.125,360.5,8.24e-01,3.15e+00,0.956394,3.65925,53.3333,NaN,NaN,0,62.7794,17.4481,76.6667

\end{filecontents*}
\csvreader[hyperExp]{experiments/proj_hyper_n50_d47_30_tol_high.csv}
	{}{\printExpDataLine
	}
	\caption{$n=50$, $k=47$}\label{table:hyper_low_deg_50_47}
\end{subtable}\caption{All experiments were done with $30$ randomly generated points. The polynomials in the experiments described here all have degree $3$. For Table~\ref{table:hyper_low_deg_40_37}, the last row (Error = 0.001\%) is not directly comparable to the previous rows as the success rates for FW and FW EleSym are much smaller and the relative time averages for each algorithm only consider successful instances.
More detailed experimental data regarding FW and  FW EleSym is given in Table~\ref{table:hyper_low_deg_det}.  }\label{table:hyper_low_deg}
\end{table}

\begin{table}\footnotesize
\begin{DIFnomarkup}
	\setcounter{subtable}{0}
	\begin{subtable}{\textwidth}
		\csvreader[hyperExpDet]{experiments/proj_hyper_n30_d27_30_tol_high.csv}
		{}{\printExpDataLineDet
		}
		\caption{$n=30$, $k=27$}\label{table:hyper_low_deg_30_27_det}
	\end{subtable}
	\begin{subtable}{\textwidth}
		\csvreader[hyperExpDet]{experiments/proj_hyper_n40_d37_30_tol_high.csv}
		{}{\printExpDataLineDet
		}
		\caption{$n=40$, $k=37$}\label{table:hyper_low_deg_40_37_det}
	\end{subtable}
	\begin{subtable}{\textwidth}
		\csvreader[hyperExpDet]{experiments/proj_hyper_n50_d47_30_tol_high.csv}
		{}{\printExpDataLineDet
		}
		\caption{$n=50$, $k=47$}\label{table:hyper_low_deg_50_47_det}
	\end{subtable}
\end{DIFnomarkup}
\caption{Additional data regarding the experiments in Table~\ref{table:hyper_low_deg}. The setting is the same as in Table~\ref{table:hyper_low_deg} and the meaning of the columns is the same as in Table~\ref{table:hyper_high_deg_det}. }\label{table:hyper_low_deg_det}
\end{table}

\begin{table}\footnotesize
	\begin{subtable}{\textwidth}
		\begin{filecontents*}{experiments/proj_hyper_n10_d1_30_tol_low.csv}
error,FWx,FWxstd,FWit,FWitstd,FWabs,FWabsstd,FW,FWstd,FWsuccess,FWElx,FWElxstd,FWElit,FWElitstd,FWElabs,FWElabsstd,FWEl,FWElstd,FWElsuccess,AGM,AGMstd,AGMsuccess,AGMEl,AGMElstd,AGMElsuccess
10,2.43e-02,1.96e-02,7.66667,2.92826,2.46e-03,8.10e-04,0.426937,0.155426,100,2.43e-02,1.96e-02,7.66667,2.92826,5.21e-03,2.64e-03,0.887016,0.413549,100,53.2747,16.2682,33.3333,NaN,NaN,0
5,2.29e-02,1.98e-02,7.8,2.77178,2.51e-03,9.24e-04,0.438994,0.19073,100,2.29e-02,1.98e-02,7.8,2.77178,5.27e-03,2.58e-03,0.897662,0.402873,100,67.6228,13.8881,13.3333,NaN,NaN,0
1,1.25e-02,8.19e-03,11.1667,5.8018,3.32e-03,1.58e-03,0.577735,0.298226,100,1.25e-02,8.19e-03,11.1667,5.8018,7.04e-03,3.99e-03,1.19814,0.637863,100,NaN,NaN,0,NaN,NaN,0
0.5,9.54e-03,7.31e-03,17.5,13.9525,4.84e-03,3.09e-03,0.862096,0.628811,100,9.54e-03,7.32e-03,17.4333,13.9647,1.02e-02,7.29e-03,1.78405,1.38681,100,NaN,NaN,0,NaN,NaN,0
0.1,6.68e-03,5.40e-03,248.481,247.903,5.37e-02,5.22e-02,9.7732,9.79603,90,6.39e-03,5.21e-03,226.9,243.831,1.11e-01,1.16e-01,20.0223,21.4458,100,NaN,NaN,0,NaN,NaN,0
0.05,7.07e-03,6.40e-03,523.852,489.925,1.11e-01,1.01e-01,20.0311,18.9724,90,7.19e-03,6.44e-03,343.577,365.84,1.69e-01,1.76e-01,28.9597,30.1816,86.6667,NaN,NaN,0,NaN,NaN,0
0.01,8.67e-03,7.59e-03,842.85,682.43,1.81e-01,1.41e-01,29.9564,22.9499,66.6667,5.31e-03,3.67e-03,454.389,461.042,2.40e-01,2.40e-01,40.3067,40.0953,60,NaN,NaN,0,NaN,NaN,0
0.005,8.70e-03,7.69e-03,954.6,780.094,2.05e-01,1.60e-01,33.9089,26.0563,66.6667,5.56e-03,3.57e-03,267.769,296.983,1.45e-01,1.59e-01,23.4755,25.3904,43.3333,NaN,NaN,0,NaN,NaN,0
0.001,8.78e-03,7.99e-03,946.632,735.598,2.05e-01,1.54e-01,34.0839,25.355,63.3333,5.56e-03,3.78e-03,308.692,339.177,1.67e-01,1.82e-01,27.0615,29.4577,43.3333,NaN,NaN,0,NaN,NaN,0

\end{filecontents*}
\csvreader[hyperExp]{experiments/proj_hyper_n10_d1_30_tol_low.csv}
		{}{\printExpDataLine
		}
		\caption{$n=10$, $k=1$, 30 points}\label{table:hyper_high_deg_low_acc_10_1}
	\end{subtable}
	\begin{subtable}{\textwidth}
		{\begin{filecontents*}{experiments/proj_hyper_n10_d2_30_tol_low.csv}
error,FWx,FWxstd,FWit,FWitstd,FWabs,FWabsstd,FW,FWstd,FWsuccess,FWElx,FWElxstd,FWElit,FWElitstd,FWElabs,FWElabsstd,FWEl,FWElstd,FWElsuccess,AGM,AGMstd,AGMsuccess,AGMEl,AGMElstd,AGMElsuccess
10,5.50e-02,3.01e-02,5.33333,1.29544,3.82e-03,2.07e-03,0.562391,0.280068,100,5.50e-02,3.01e-02,5.33333,1.29544,3.41e-03,1.44e-03,0.500599,0.194968,100,61.9899,13.0385,73.3333,106.673,3.06337,20
5,4.03e-02,1.57e-02,6.13333,0.937102,4.10e-03,2.18e-03,0.606638,0.296899,100,4.03e-02,1.57e-02,6.13333,0.937102,3.72e-03,1.26e-03,0.550984,0.174362,100,71.7756,12.6391,43.3333,NaN,NaN,0
1,1.90e-02,8.99e-03,18.3333,11.9289,7.83e-03,4.00e-03,1.17964,0.627666,100,1.90e-02,8.99e-03,18.3333,11.9289,9.15e-03,7.29e-03,1.38015,1.08096,100,85.9117,3.81475,6.66667,NaN,NaN,0
0.5,1.36e-02,8.78e-03,50.6333,42.7022,1.76e-02,1.22e-02,2.66282,1.89002,100,1.36e-02,8.78e-03,50.6333,42.7022,2.20e-02,1.80e-02,3.34005,2.74111,100,NaN,NaN,0,NaN,NaN,0
0.1,7.17e-03,6.24e-03,550.552,311.378,1.65e-01,8.81e-02,24.5454,14.4739,96.6667,7.19e-03,6.32e-03,555.345,309.628,2.38e-01,1.27e-01,35.3033,20.3093,96.6667,NaN,NaN,0,NaN,NaN,0
0.05,7.70e-03,8.48e-03,1077,643.072,3.20e-01,1.81e-01,45.9477,26.24,86.6667,8.30e-03,9.31e-03,833.476,430.632,3.72e-01,1.85e-01,52.8043,25.4467,70,NaN,NaN,0,NaN,NaN,0
0.01,9.93e-03,1.03e-02,1308.23,730.858,4.03e-01,2.18e-01,56.2534,29.0708,43.3333,1.32e-02,1.20e-02,906.625,633.389,4.16e-01,2.75e-01,58.4777,36.6323,26.6667,NaN,NaN,0,NaN,NaN,0
0.005,1.06e-02,1.12e-02,1316.45,785.898,4.07e-01,2.36e-01,57.0618,31.4884,36.6667,9.42e-03,1.02e-02,604.2,553.108,2.79e-01,2.31e-01,42.3919,34.36,16.6667,NaN,NaN,0,NaN,NaN,0
0.001,1.22e-02,1.20e-02,1216.11,796.214,3.75e-01,2.35e-01,53.4408,32.0663,30,9.49e-03,1.04e-02,658,602.435,3.04e-01,2.51e-01,46.1686,37.4169,16.6667,NaN,NaN,0,NaN,NaN,0

\end{filecontents*}
\csvreader[hyperExp]{experiments/proj_hyper_n10_d2_30_tol_low.csv}
			{}{\printExpDataLine
			}
		}\caption{$n=10$, $k=2$}
	\end{subtable}
	\begin{subtable}{\textwidth}
		{\begin{filecontents*}{experiments/proj_hyper_n20_d1_30_tol_low.csv}
error,FWx,FWxstd,FWit,FWitstd,FWabs,FWabsstd,FW,FWstd,FWsuccess,FWElx,FWElxstd,FWElit,FWElitstd,FWElabs,FWElabsstd,FWEl,FWElstd,FWElsuccess,AGM,AGMstd,AGMsuccess,AGMEl,AGMElstd,AGMElsuccess
10,2.12e-02,0.00e+00,11,0,1.68e-02,0.00e+00,0.659661,0,3.33333,4.03e-02,2.28e-02,17.6,6.71899,4.05e-02,1.63e-02,1.52774,0.597934,100,NaN,NaN,0,NaN,NaN,0
5,2.12e-02,0.00e+00,11,0,1.68e-02,0.00e+00,0.659661,0,3.33333,3.92e-02,2.27e-02,17.8333,6.53417,4.09e-02,1.60e-02,1.54727,0.589207,100,NaN,NaN,0,NaN,NaN,0
1,1.60e-02,0.00e+00,18,0,3.41e-02,0.00e+00,1.3411,0,3.33333,2.76e-02,1.40e-02,25.3667,9.20451,5.63e-02,2.05e-02,2.15568,0.842248,100,NaN,NaN,0,NaN,NaN,0
0.5,NaN,NaN,NaN,NaN,NaN,NaN,NaN,NaN,0,2.05e-02,1.06e-02,36.2333,16.2961,7.89e-02,3.54e-02,3.00993,1.37533,100,NaN,NaN,0,NaN,NaN,0
0.1,NaN,NaN,NaN,NaN,NaN,NaN,NaN,NaN,0,1.10e-02,6.44e-03,92.2667,53.9821,1.94e-01,1.08e-01,7.32839,3.93734,100,NaN,NaN,0,NaN,NaN,0
0.05,NaN,NaN,NaN,NaN,NaN,NaN,NaN,NaN,0,1.09e-02,6.89e-03,141.267,138.716,2.91e-01,2.66e-01,10.9534,9.6723,100,NaN,NaN,0,NaN,NaN,0
0.01,NaN,NaN,NaN,NaN,NaN,NaN,NaN,NaN,0,1.09e-02,7.27e-03,212.143,232.333,4.38e-01,4.59e-01,16.3754,16.8853,93.3333,NaN,NaN,0,NaN,NaN,0
0.005,NaN,NaN,NaN,NaN,NaN,NaN,NaN,NaN,0,1.09e-02,7.29e-03,248.036,316.101,5.12e-01,6.38e-01,19.0896,23.253,93.3333,NaN,NaN,0,NaN,NaN,0
0.001,NaN,NaN,NaN,NaN,NaN,NaN,NaN,NaN,0,1.13e-02,7.24e-03,250.593,334.038,5.11e-01,6.57e-01,19.1022,24.0041,90,NaN,NaN,0,NaN,NaN,0

\end{filecontents*}
\csvreader[hyperExp]{experiments/proj_hyper_n20_d1_30_tol_low.csv}
			{}{\printExpDataLine
			}
		}\caption{$n=20$, $k=1$}\label{table:hyper_high_deg_low_acc_20_1}
	\end{subtable}
	\begin{subtable}{\textwidth}
		{\begin{filecontents*}{experiments/proj_hyper_n20_d2_30_tol_low.csv}
error,FWx,FWxstd,FWit,FWitstd,FWabs,FWabsstd,FW,FWstd,FWsuccess,FWElx,FWElxstd,FWElit,FWElitstd,FWElabs,FWElabsstd,FWEl,FWElstd,FWElsuccess,AGM,AGMstd,AGMsuccess,AGMEl,AGMElstd,AGMElsuccess
10,4.33e-02,2.05e-02,8,1.22474,2.76e-02,2.75e-03,0.309603,0.0281935,16.6667,5.95e-02,3.46e-02,12.3667,3.44897,2.76e-02,8.21e-03,0.290195,0.0776649,100,NaN,NaN,0,NaN,NaN,0
5,3.34e-02,1.28e-02,9.5,1.91485,3.06e-02,3.25e-03,0.347885,0.0718034,13.3333,5.03e-02,2.64e-02,13.4667,3.29821,2.96e-02,7.80e-03,0.312903,0.0764137,100,NaN,NaN,0,NaN,NaN,0
1,2.22e-02,3.84e-03,10.5,3.53553,3.16e-02,6.96e-03,0.330805,0.0726771,6.66667,2.63e-02,1.24e-02,22.0667,8.3622,4.61e-02,1.70e-02,0.48922,0.183007,100,NaN,NaN,0,NaN,NaN,0
0.5,2.02e-02,9.77e-04,21,18.3848,6.80e-02,5.84e-02,0.711305,0.610785,6.66667,1.78e-02,8.23e-03,32.1333,13.4977,6.57e-02,2.77e-02,0.698751,0.297732,100,NaN,NaN,0,NaN,NaN,0
0.1,1.81e-02,0.00e+00,347,0,1.04e+00,0.00e+00,10.8837,0,3.33333,9.66e-03,7.31e-03,159.7,212.566,2.97e-01,3.62e-01,3.24443,4.21905,100,NaN,NaN,0,NaN,NaN,0
0.05,2.42e-02,0.00e+00,413,0,1.74e+00,0.00e+00,18.2,0,3.33333,9.72e-03,7.70e-03,300.167,401.508,5.57e-01,6.88e-01,6.09906,8.0335,100,NaN,NaN,0,NaN,NaN,0
0.01,NaN,NaN,NaN,NaN,NaN,NaN,NaN,NaN,0,1.09e-02,7.94e-03,804.233,990.473,1.53e+00,1.77e+00,16.5186,20.3328,100,NaN,NaN,0,NaN,NaN,0
0.005,NaN,NaN,NaN,NaN,NaN,NaN,NaN,NaN,0,1.13e-02,7.94e-03,789.724,964.133,1.51e+00,1.70e+00,15.7334,17.6481,96.6667,NaN,NaN,0,NaN,NaN,0
0.001,NaN,NaN,NaN,NaN,NaN,NaN,NaN,NaN,0,1.15e-02,7.96e-03,912.931,1123.25,1.74e+00,1.98e+00,18.192,20.6367,96.6667,NaN,NaN,0,NaN,NaN,0

\end{filecontents*}
\csvreader[hyperExp]{experiments/proj_hyper_n20_d2_30_tol_low.csv}
			{}{\printExpDataLine
			}
		}\caption{$n=20$, $k=2$}
	\end{subtable}	
	\caption{The setting is the same as in Table~\ref{table:hyper_high_deg} except that DDS is run with $10^{-3}$ tolerance.}\label{table:hyper_high_deg_low_acc}		
\end{table}

\begin{table}\small
	\begin{subtable}{\textwidth}
		\begin{filecontents*}{experiments/proj_hyper_n30_d27_30_tol_low.csv}
error,FWx,FWxstd,FWit,FWitstd,FWabs,FWabsstd,FW,FWstd,FWsuccess,FWElx,FWElxstd,FWElit,FWElitstd,FWElabs,FWElabsstd,FWEl,FWElstd,FWElsuccess,AGM,AGMstd,AGMsuccess,AGMEl,AGMElstd,AGMElsuccess
10,8.17e-03,7.04e-03,3,0,1.34e-01,4.34e-03,1.44708,0.15363,100,8.17e-03,7.04e-03,3,0,1.24e-02,3.41e-04,0.134381,0.0129784,100,53.8634,8.68164,96.6667,11.3773,2.9821,100
5,8.17e-03,7.04e-03,3,0,1.34e-01,4.34e-03,1.44708,0.15363,100,8.17e-03,7.04e-03,3,0,1.24e-02,3.41e-04,0.134381,0.0129784,100,66.4212,16.1949,96.6667,16.1017,16.2685,100
1,8.43e-03,7.02e-03,3,0,1.34e-01,4.42e-03,1.45297,0.152857,96.6667,8.43e-03,7.02e-03,3,0,1.24e-02,2.51e-04,0.134537,0.0131793,96.6667,83.7198,16.8925,76.6667,19.3299,9.57615,96.6667
0.5,8.43e-03,7.02e-03,3,0,1.34e-01,4.42e-03,1.45297,0.152857,96.6667,8.43e-03,7.02e-03,3,0,1.24e-02,2.51e-04,0.134537,0.0131793,96.6667,82.2161,17.3476,66.6667,21.7275,13.313,96.6667
0.1,6.58e-03,3.75e-03,3,0,1.34e-01,5.33e-03,1.45823,0.15372,60,6.50e-03,3.66e-03,146.263,624.47,4.24e-01,1.79e+00,4.92876,20.8945,63.3333,94.7576,8.92626,13.3333,31.4101,16.8179,46.6667
0.05,6.96e-03,3.29e-03,3,0,1.33e-01,6.49e-03,1.44277,0.152465,40,6.96e-03,3.29e-03,3,0,1.24e-02,2.32e-04,0.134399,0.0111445,40,NaN,NaN,0,37.8775,21.4879,13.3333
0.01,NaN,NaN,NaN,NaN,NaN,NaN,NaN,NaN,0,NaN,NaN,NaN,NaN,NaN,NaN,NaN,NaN,0,NaN,NaN,0,NaN,NaN,0
0.005,NaN,NaN,NaN,NaN,NaN,NaN,NaN,NaN,0,NaN,NaN,NaN,NaN,NaN,NaN,NaN,NaN,0,NaN,NaN,0,NaN,NaN,0
0.001,NaN,NaN,NaN,NaN,NaN,NaN,NaN,NaN,0,NaN,NaN,NaN,NaN,NaN,NaN,NaN,NaN,0,NaN,NaN,0,NaN,NaN,0

\end{filecontents*}
\csvreader[hyperExp]{experiments/proj_hyper_n30_d27_30_tol_low.csv}
		{}{\printExpDataLine
		}\caption{$n=30$, $k=27$}
	\end{subtable}
	\begin{subtable}{\textwidth}
		\begin{filecontents*}{experiments/proj_hyper_n40_d37_30_tol_low.csv}
error,FWx,FWxstd,FWit,FWitstd,FWabs,FWabsstd,FW,FWstd,FWsuccess,FWElx,FWElxstd,FWElit,FWElitstd,FWElabs,FWElabsstd,FWEl,FWElstd,FWElsuccess,AGM,AGMstd,AGMsuccess,AGMEl,AGMElstd,AGMElsuccess
10,4.47e-03,3.36e-03,3,0,4.67e-01,7.06e-03,1.73987,0.230934,100,4.47e-03,3.36e-03,3,0,2.31e-02,7.95e-04,0.0862111,0.0119341,100,58.2599,8.56706,96.6667,10.0624,15.6024,100
5,4.47e-03,3.36e-03,3,0,4.67e-01,7.06e-03,1.73987,0.230934,100,4.47e-03,3.36e-03,3,0,2.31e-02,7.95e-04,0.0862111,0.0119341,100,67.7361,13.5758,93.3333,11.2606,15.4503,100
1,4.62e-03,3.33e-03,3,0,4.67e-01,7.19e-03,1.74675,0.231877,96.6667,4.62e-03,3.33e-03,3,0,2.31e-02,7.92e-04,0.0864518,0.012071,96.6667,81.9246,21.4179,56.6667,12.8296,5.00968,96.6667
0.5,4.62e-03,3.33e-03,3,0,4.67e-01,7.19e-03,1.74675,0.231877,96.6667,4.62e-03,3.33e-03,3,0,2.31e-02,7.92e-04,0.0864518,0.012071,96.6667,79.346,22.4185,40,18.4894,17.6193,96.6667
0.1,4.14e-03,3.32e-03,3,0,4.67e-01,6.82e-03,1.68457,0.191077,66.6667,4.14e-03,3.32e-03,3,0,2.31e-02,6.83e-04,0.0832924,0.0095951,66.6667,109.243,0,3.33333,38.1082,25.6731,53.3333
0.05,2.31e-03,1.68e-03,3,0,4.66e-01,7.77e-03,1.75664,0.179338,40,2.31e-03,1.68e-03,3,0,2.27e-02,2.98e-04,0.0856821,0.00898973,40,NaN,NaN,0,60.4733,34.0719,26.6667
0.01,NaN,NaN,NaN,NaN,NaN,NaN,NaN,NaN,0,NaN,NaN,NaN,NaN,NaN,NaN,NaN,NaN,0,NaN,NaN,0,NaN,NaN,0
0.005,NaN,NaN,NaN,NaN,NaN,NaN,NaN,NaN,0,NaN,NaN,NaN,NaN,NaN,NaN,NaN,NaN,0,NaN,NaN,0,NaN,NaN,0
0.001,NaN,NaN,NaN,NaN,NaN,NaN,NaN,NaN,0,NaN,NaN,NaN,NaN,NaN,NaN,NaN,NaN,0,NaN,NaN,0,NaN,NaN,0

\end{filecontents*}
\csvreader[hyperExp]{experiments/proj_hyper_n40_d37_30_tol_low.csv}
		{}{\printExpDataLine
		}\caption{$n=40$, $k=37$}
	\end{subtable}
	\begin{subtable}{\textwidth}
		\begin{filecontents*}{experiments/proj_hyper_n50_d47_30_tol_low.csv}
error,FWx,FWxstd,FWit,FWitstd,FWabs,FWabsstd,FW,FWstd,FWsuccess,FWElx,FWElxstd,FWElit,FWElitstd,FWElabs,FWElabsstd,FWEl,FWElstd,FWElsuccess,AGM,AGMstd,AGMsuccess,AGMEl,AGMElstd,AGMElsuccess
10,3.20e-03,2.56e-03,3,0,1.26e+00,5.95e-02,2.4214,0.235483,100,3.20e-03,2.56e-03,3,0,3.91e-02,4.83e-03,0.0754989,0.0134723,100,83.5472,13.4833,100,6.42133,1.04027,100
5,3.20e-03,2.56e-03,3,0,1.26e+00,5.95e-02,2.4214,0.235483,100,3.20e-03,2.56e-03,3,0,3.91e-02,4.83e-03,0.0754989,0.0134723,100,87.2376,13.3771,70,11.9379,17.0521,100
1,3.28e-03,2.56e-03,3,0,1.26e+00,5.95e-02,2.42383,0.239267,96.6667,3.28e-03,2.56e-03,3,0,3.91e-02,4.92e-03,0.0754631,0.0137093,96.6667,89.7555,11.8259,26.6667,15.5735,13.6817,96.6667
0.5,3.28e-03,2.56e-03,3,0,1.26e+00,5.95e-02,2.42383,0.239267,96.6667,3.28e-03,2.56e-03,3,0,3.91e-02,4.92e-03,0.0754631,0.0137093,96.6667,93.1406,7.71815,16.6667,19.3949,18.1608,96.6667
0.1,3.07e-03,2.27e-03,3,0,1.26e+00,5.98e-02,2.39161,0.22428,73.3333,3.07e-03,2.27e-03,3,0,3.94e-02,5.54e-03,0.0750412,0.0153082,73.3333,92.4444,10.3275,10,32.0263,28.6418,70
0.05,2.82e-03,1.98e-03,3,0,1.26e+00,5.48e-02,2.39897,0.227079,70,2.82e-03,1.98e-03,3,0,3.95e-02,5.66e-03,0.0757253,0.0153378,70,NaN,NaN,0,53.7675,27.6326,53.3333
0.01,NaN,NaN,NaN,NaN,NaN,NaN,NaN,NaN,0,NaN,NaN,NaN,NaN,NaN,NaN,NaN,NaN,0,NaN,NaN,0,NaN,NaN,0
0.005,NaN,NaN,NaN,NaN,NaN,NaN,NaN,NaN,0,NaN,NaN,NaN,NaN,NaN,NaN,NaN,NaN,0,NaN,NaN,0,NaN,NaN,0
0.001,NaN,NaN,NaN,NaN,NaN,NaN,NaN,NaN,0,NaN,NaN,NaN,NaN,NaN,NaN,NaN,NaN,0,NaN,NaN,0,NaN,NaN,0

\end{filecontents*}
\csvreader[hyperExp]{experiments/proj_hyper_n50_d47_30_tol_low.csv}
		{}{\printExpDataLine}\caption{$n=50$, $k=47$}
	\end{subtable}\caption{The setting is the same as in Table~\ref{table:hyper_low_deg} except that DDS is run with $10^{-3}$ tolerance.}\label{table:hyper_low_deg_low_acc}	
\end{table}

\subsubsection{Hyperbolic polynomials with many monomials}\label{sec:num_mil}
For certain choices of $n$ and $k$, the corresponding 
$(n-k)$-th elementary symmetric polynomial has hundreds of thousands of monomials.  
For example, for $(n,k) = (20,10)$ and $(n,k) = (30,15)$, $\sigma_{n,n-k}$ has, respectively, $184756$ and $155117520$ monomials.
In this subsection, our goal is to check whether our proposed algorithm can still function properly when the underlying hyperbolic polynomial has a huge number of monomials.
Here, we focus on the FW EleSym method and we recall that the ``EleSym'' suffix indicates the usage of special methods to handle elementary symmetric polynomials.
We will also use this opportunity to check the behavior of Algorithm~\ref{alg:proposed_method} over a single instance. 

The reason for focusing on FW EleSym only is that for $(n,k) = (20,10)$ and $(n,k) = (30,15)$, DDS struggles to complete a single iteration. And, from the previous discussion we saw that FW EleSym was significantly faster than either version of Renegar's AGM. 

For  $(n,k) \in \{(20,10), (30,15)\}$ we generated $10$ random instances using the same procedure as before. Then, for each instance, we ran FW EleSym for $10$ seconds and we plotted the Frank-Wolfe gap and the relative objective function value in Figures~\ref{fig:hyper_mon_20_10} and \ref{fig:hyper_mon_30_15} using log-log plots.
For each instance the relative objective function value was computed as follows: we compute the smallest function value obtained through the $10$ seconds among the primal iterates whose minimal eigenvalues were at least $-10^{-8}$. We call this value $\hat f_{\text{opt}}$. 
Then, denoting the objective function of \eqref{prob:primal} by $f$ and $k$-th primal iterate by $x_k$ the relative objective function value at the $k$-th iteration is
\[
\frac{\min_{1\leq i \leq k }f(x_i)-\hat f_{\text{opt}}}{\hat f_{\text{opt}}},
\]
with the caveat that ``$\min$'' is only considered over primal iterates whose minimal eigenvalues were at least $-10^{-8}$.
The goal is to measure empirically how fast the primal objective value is converging. 
Using $\hat f_{\text{opt}}$ may seem odd, but the issue is that we do not know the true optimal values and have no other baselines to compare  since we were not able to solve the problem with DDS. 

Both Figures \ref{fig:hyper_mon_20_10} and \ref{fig:hyper_mon_30_15} suggest that the Frank-Wolfe gap and the function values are decreasing sublinearly, which is consistent with the convergence results described in Section~\ref{sec:convergence}.
Denoting the common optimal value of \eqref{prob:dual} and \eqref{prob:dual_compact} by $d^*$, we recall that the Frank-Wolfe gap  at the $k$-th iteration is an upper bound to the difference $h(y_k) - d^*$, where $h(y_k)$ is the value of the dual objective function at the $k$-iterate. 
In view of Theorem~\ref{thm:convergence_rate}, the square root of the Frank-Wolfe gap times a constant can be used to bound the distance of the primal iterate to the
primal optimal solution. More precisely, denoting the $k$-th iterate by $x_k$, the primal optimal solution by $x_{\text{opt}}$ (i.e., the projection) and the Frank-Wolfe gap at the $k$-th iteration by $G_k$ we have 
\[
\norm{x_k-x_{\text{opt}}} \leq \sqrt{2}\sqrt{G_k},
\]
because the objective function of \eqref{prob:proj_original} is $1$-strongly convex and $T$ (as in Theorem~\ref{thm:convergence_rate}) is the identity matrix.

So the fact that in both plots the Frank-Wolfe gap is indeed decreasing for all instances, gives us some numerical confidence that Algorithm~\ref{alg:proposed_method} is indeed approaching the true optimal solution of \eqref{prob:primal} in spite of the challenging circumstances. 
This suggests that even if the hyperbolic polynomial has millions of monomials, Algorithm~\ref{alg:proposed_method} can still work properly provided that the underlying computational algebra for the polynomial is carefully implemented.
 

\begin{figure}	
	\begin{subfigure}{1\textwidth}
			\centering\includegraphics[scale=0.5]{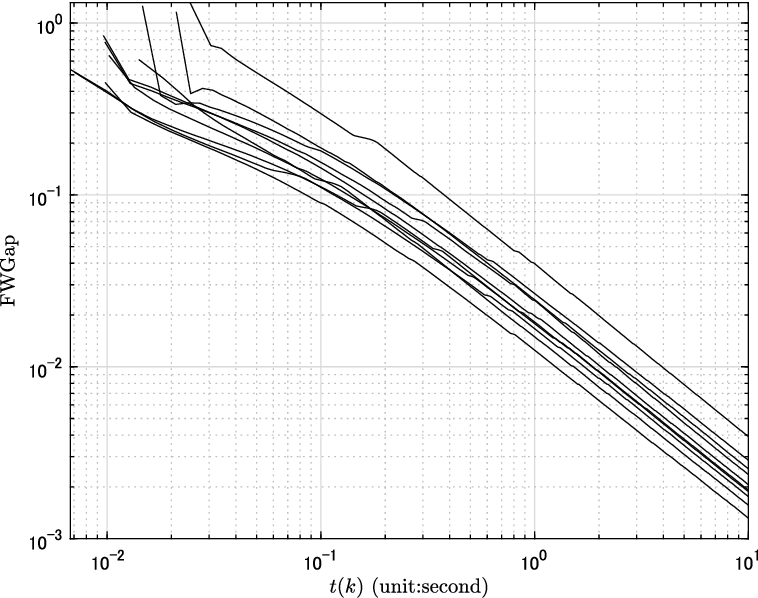}\includegraphics[scale=0.5]{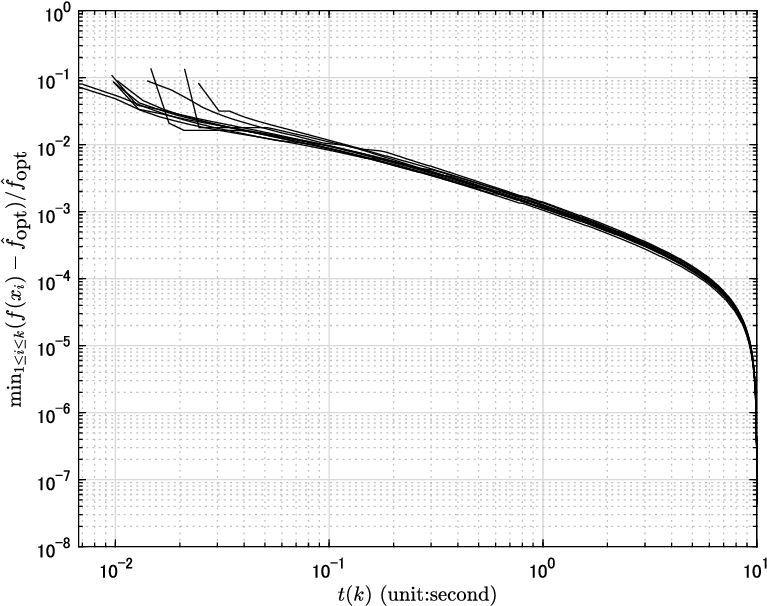}
			\caption{$n = 20$, $k = 10$. The corresponding hyperbolic polynomial has $184756$ monomials.}\label{fig:hyper_mon_20_10}
	\end{subfigure}
	\begin{subfigure}{1\textwidth}
	\centering\includegraphics[scale=0.5]{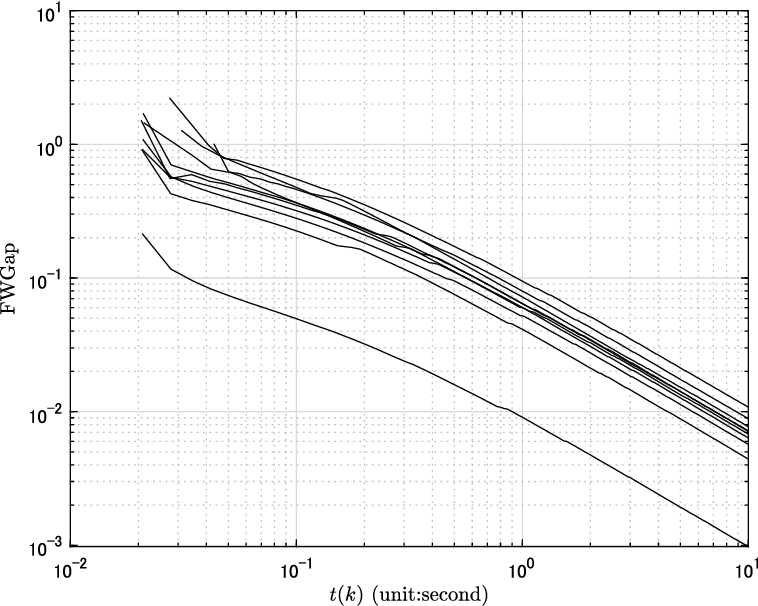}\includegraphics[scale=0.5]{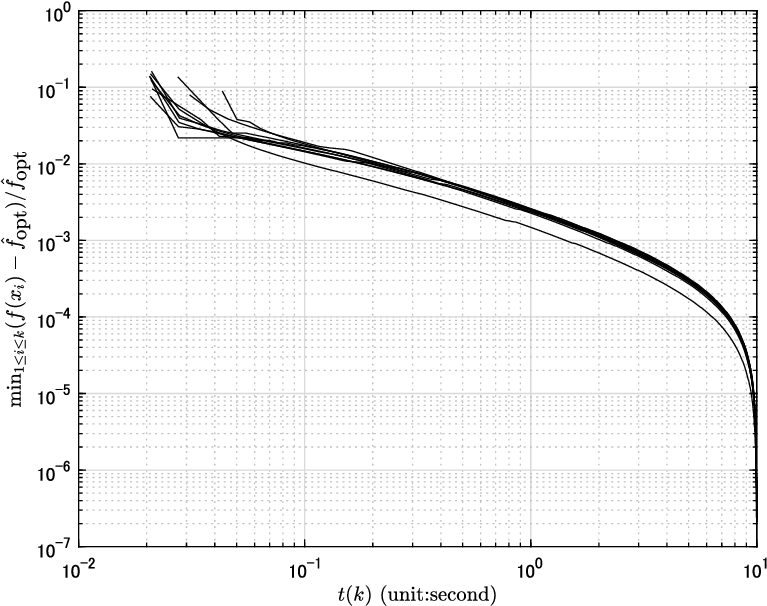}
		\caption{$n = 30$, $k = 15$. The corresponding hyperbolic polynomial has $155117520$ monomials.}\label{fig:hyper_mon_30_15}
	\end{subfigure}

\caption{Frank-Wolfe gap and relative function values log-log plots for the cases $(n,k) \in \{(20,10),(30,15)\}$,  $10$ instances each. 
	For relative function values (the plots on the right), since $\hat f_{\text{opt}}$ is the best solution obtained during the $10$ seconds, it is natural that the relative error computed empirically goes to $0$. Still, the fact that the graph is almost a straight line before that suggests that the convergence is indeed sublinear as predicted by \eqref{eq:obj_vals_convergence}. }
\end{figure}

\subsection{Projection onto $p$-cones} \label{sec:experiment_proj_p}
Algorithm~\ref{alg:proposed_method} is also applicable in more general settings so in this subsection our goal is to examine its behavior beyond the case of hyperbolicity cones. 
Here, we consider the problem of projecting a given $c \in \Re^{n+1}$ onto a $p$-cone.
\begin{align}
  \label{prob:proj_Pcone}
  \min_{x\in\Re^{n+1}} &\:\: \frac{1}{2}\| x-c \|^2 \\
  \mathrm{s.t. } & \:\: x \in \pK, \notag
\end{align}
where $\pK = \{(x_0,x_1,\ldots, x_n) \in \Re^{n+1}\mid x_0 \geq 0, x_0^p \geq \abs{x_1}^p + \cdots +\abs{x_n}^p\}$ for some $p \in (1,\infty)$.

The problem \eqref{prob:proj_Pcone} can be solved with both DDS and Mosek \cite{MC2020}. However, Mosek  does not handle $p$-norm constraints directly, so we need to reformulate 
\eqref{prob:proj_Pcone} using power cone constraints as follows. 
%
%
\begin{align*}
  x \in \stdCone_p^{n+1}
  &\iff \left\lbrace \begin{array}{l}
    x_0^p \geq \sum_{i=1}^n \abs{x_i}^p \\
    x_0 \geq 0
  \end{array}\right.\\
  &\iff \left\lbrace \begin{array}{l}
    x_0^p \geq \sum_{i=1}^n y_i x_0^{p-1} \\
    y_i x_0^{p-1} \geq |x_i|^p \qquad (i=1,\dots,n) \\
    x_0 \geq 0
  \end{array}\right.\\
  &\iff \left\lbrace \begin{array}{l}
    x_0 \geq \sum_{i=1}^n y_i  \\
    y_i^{1/p} x_0^{1-1/p} \geq |x_i| \qquad (i=1,\dots,n) \\
    x_0 \geq 0
  \end{array}\right.
\end{align*}
Using this transformation and dropping the quadratic objective function to a second-order cone constraint, \eqref{prob:proj_Pcone} is
transformed into the following equivalent problem.
\begin{align*}
  \min_{t\in\Re,x\in\Re^{n+1}} \:\: & t \\
  \mathrm{s.t. }  \:\: & t \geq \| x-c \| \\
  &x_0 \geq \sum_{i=1}^n y_i\\
  & x_0^{1-1/p}y_i^{1/p} \geq |x_i| \qquad (i=1,\dots,n)\\
  & x_0 \geq 0.
\end{align*}
In our preliminary tests, Mosek was significantly faster than DDS, so in the following experiments we only compare with Mosek.

As in Section~\ref{sec:exp_rel}, our  implementation of Algorithm~\ref{alg:proposed_method} for solving \eqref{prob:proj_Pcone} is relatively straightforward.
For more details, see the file \texttt{FW\_GCP\_exp.m}.
 The element $e$ is given by $(1,0,\ldots, 0)$ and, with that, the generalized minimum eigenvalue function is such that 
\[
\lambda_{\min}(x) = x_0 - \sqrt{\sum_{i=1}^n \abs{x_i}^p}.
\]
The computation of conjugate vectors is done using the formulae described in \cite[Section~4.1]{LLP21} and the constant $c_D$ is computed as in \eqref{eq:cd_projection}.

We follow the same experimental procedures as in Section~\ref{sec:exp_rel}. We generate $c$ by sampling from the  standard normal distribution, discarding points that are too close to the cone and repeating until $30$ points were generated.

We tested our implementation of Algorithm~2 with $p \in \{1.1,1.3,3,5\}$ and $n \in \{100,300,500,1000\}$.
The results are described in Tables~\ref{table:pcone_high} and \ref{table:pcone_low}. Table entries have the same meaning as in 
Tables~\ref{table:hyper_high_deg}--\ref{table:hyper_low_deg_low_acc}. 
 Analogously to Section~\ref{sec:comparison}, our goal was to examine how long does take it take on average to obtain a solution that has a value that is close to the one obtained by Mosek.
For example, the entry ``$2.72 \pm 1.05$'' at the column $p = 3$ at Table~\ref{table:pcone_high_1000} means that, on average over $30$ points, Algorithm~\ref{alg:proposed_method} required $2.72\%$ of the time that Mosek needed in order to find a solution whose value is within $0.5\%$ of the optimal value found by Mosek. 
As before, we only consider iterates that satisfy 
$\lambda_{\min}(x_k) \geq -10^{-8}$ and all objective function value computations are considered with respect the formulation in \eqref{prob:proj_Pcone}.
In all instances, we set the maximum running time to be equal to the time spent by Mosek.

For the experiments in Table~\ref{table:pcone_high}, Mosek was configured to run with its default accuracy settings. 
For most values of $p$ and $n$, Algorithm~\ref{alg:proposed_method} was able to obtain solutions having objective value between $1\%$ and $0.5\%$ of the value obtained by Mosek in a fraction of the time. 
The case $p = 1.1$ seems to be the most challenging and Algorithm~\ref{alg:proposed_method} often requires  at least $10\%$ of the running time Mosek to reach solutions with $E = 0.5\%$. 
The performance for the other $p$'s was better and for, say, $p = 3$, even for $n = 1000$ no more than $3\%$ of the running time of Mosek was required to reach solutions with $E = 0.5\%$.

Extra statistics for the cases $n \in \{100,1000\}$ are given in Table~\ref{table:pcone_high_det}. The meaning of the columns is analogous to the columns of Table~\ref{table:hyper_high_deg_det}.
The ``Iterations'' and ``Abs. time (sec.)'' columns contain the average and the standard deviation of the number of iterations and the absolute time required to reach each error target, respectively. 
The   ``$\norm{x_k - x_{\text{Mos}}}_\infty$'' column correspond to the average distance (in the infinity norm) to the solution obtained by Mosek. 
As expected, for each $n$ and $p$ the number of iterations and the absolute time increase as the error threshold decreases. 

We also performed experiments where Mosek is configured to run with a lower optimality threshold of $10^{-3}$, 
analogously to Tables~\ref{table:hyper_high_deg_low_acc}-\ref{table:hyper_low_deg_low_acc}\footnote{More precisely, the parameter \texttt{MSK\_DPAR\_INTPNT\_CO\_TOL\_MU\_RED} of Mosek which controls the relative complementarity gap tolerance is set to $10^{-3}$. }. 
These experiments are described in Table~\ref{table:pcone_low} and the results are largely similar to the ones reported in Table~\ref{table:pcone_high}.

Again, it should be emphasized that the goal of experimental setting described in Tables~\ref{table:pcone_high} and \ref{table:pcone_low} is to understand the trade-off between the accuracy afforded by a second-order approach and the fast iterations of a first-order method for this particular class of problems.
In this sense, Algorithm~\ref{alg:proposed_method} seems to be competitive since it consistently obtain relatively close solutions within a fraction of the time required by Mosek. 
On the other hand, it struggles to get closer than $0.1\%$ of the objective value obtained by Mosek within the allotted time budget.



\begin{table}
\begin{subtable}{\textwidth}\small
\csvreader[pconeExp]%
{experiments/proj_pcone_n100_tol_high_all.csv}{
}%
{\printpConeExpDataLine}\caption{$n = 100$}%
\end{subtable}
\begin{subtable}{\textwidth}\small
	\csvreader[pconeExp]%
	{experiments/proj_pcone_n300_tol_high_all.csv}{
	}%
	{\printpConeExpDataLine}\caption{$n = 300$}%
\end{subtable}
\begin{subtable}{\textwidth}\small
	\csvreader[pconeExp]%
	{experiments/proj_pcone_n500_tol_high_all.csv}{
	}%
	{\printpConeExpDataLine}\caption{$n = 500$}%
\end{subtable}
\begin{subtable}{\textwidth}\small
	\csvreader[pconeExp]%
	{experiments/proj_pcone_n1000_tol_high_all.csv}{
	}%
	{\printpConeExpDataLine}\caption{$n = 1000$}\label{table:pcone_high_1000}%
\end{subtable}
\caption{Mean relative times in comparison with Mosek using default accuracy for $ p\in \{1.1,1.3,3,5\}$ and $n \in \{100,300,500,1000\}$. We wrote in bold the entries that correspond to the cases where Algorithm~\ref{alg:proposed_method} has mean relative time less than $15\%$ and the success rate is $100\%$. More detailed experimental data regarding the $n \in \{100,1000\}$ cases is given in Table~\ref{table:pcone_high_det}.
}\label{table:pcone_high}
\end{table}

	\begin{table}
\begin{DIFnomarkup}%
	\setcounter{subtable}{0}
	\begin{subtable}{\textwidth}\small
		\csvreader[pconeExpDet]%
		{experiments/proj_pcone_n100_tol_high_all.csv}{
		}%
		{\printpConeExpDataLineDet}\caption{$n = 100$}%
	\end{subtable}
	\begin{subtable}{\textwidth}\small
		\csvreader[pconeExpDetAlt]%
		{experiments/proj_pcone_n100_tol_high_all.csv}{
		}%
		{\printpConeExpDataLineDetAlt}\caption{$n = 100$}%
	\end{subtable}
	\begin{subtable}{\textwidth}\small
		\csvreader[pconeExpDet]%
		{experiments/proj_pcone_n1000_tol_high_all.csv}{
		}%
		{\printpConeExpDataLineDet}\caption{$n = 1000$}%
	\end{subtable}
	\begin{subtable}{\textwidth}\small
		\csvreader[pconeExpDetAlt]%
		{experiments/proj_pcone_n1000_tol_high_all.csv}{
		}%
		{\printpConeExpDataLineDetAlt}\caption{$n = 1000$}%
	\end{subtable}
\end{DIFnomarkup}%
	\caption{Additional data regarding the experiments in Table~\ref{table:pcone_high}, for $n \in \{100,1000\}$. The experimental setting is the same as in Table~\ref{table:pcone_high}. 
		Greyed out entries correspond to the cases where success rates were less than $90\%$.
		All averages and standard deviations are computed using successful instances only. 	
}\label{table:pcone_high_det}
\end{table}


\begin{table}
	\begin{subtable}{\textwidth}\small
		\csvreader[pconeExp]%
		{experiments/proj_pcone_n100_tol_low_all.csv}{
		}%
		{\printpConeExpDataLine}\caption{$n = 100$}%
	\end{subtable}
	\begin{subtable}{\textwidth}\small
		\csvreader[pconeExp]%
		{experiments/proj_pcone_n300_tol_low_all.csv}{
		}%
		{\printpConeExpDataLine}\caption{$n = 300$}%
	\end{subtable}
	\begin{subtable}{\textwidth}\small
		\csvreader[pconeExp]%
		{experiments/proj_pcone_n500_tol_low_all.csv}{
		}%
		{\printpConeExpDataLine}\caption{$n = 500$}%
	\end{subtable}
	\begin{subtable}{\textwidth}\small
		\csvreader[pconeExp]%
		{experiments/proj_pcone_n1000_tol_low_all.csv}{
		}%
		{\printpConeExpDataLine}\caption{$n = 1000$}%
	\end{subtable}
	\caption{The setting is the same as in Table~\ref{table:pcone_high} except that Mosek is run with $10^{-3}$ complementarity gap tolerance.}\label{table:pcone_low}
\end{table}

\section{Conclusion}\label{sec:conc}
The initial motivation for this paper was the problem of computing projections onto hyperbolicity cones. We explored this question from both theoretical and numerical perspectives. 
As seen in Section~\ref{sec:projection}, there are limits to what can be done for arbitrary hyperbolicity cones and formulae analogous to the ones for the positive semidefinite cone are only available in certain special cases (Propositions~\ref{prop:dist_hyperbolicity}, \ref{prop:projection} and \ref{prop:proj}). In particular, the closest analogue (Proposition~\ref{prop:proj}) requires that the underlying polynomial be complete, diagonalizable and satisfy a special condition on the eigenvalue map.
In face of these limitations, we also proposed an algorithm based on the classic Frank-Wolfe method for computing projections, see Section~\ref{sec:algorithm}. 
In fact, our method can handle more general problems including the case where the underlying cone is not necessarily a hyperbolicity cone.

As discussed in Section~\ref{sec:algorithm}, a novel point is that the Frank-Wolfe method is actually applied to the \emph{dual problem}, since this leads to subproblems whose solutions can be expressed in terms of minimum eigenvalues computations and conjugate vectors.
In the particular case of hyperbolicity cones, we show how all the necessary 
objects are computable from the underlying hyperbolic polynomial.
Then, in Section~\ref{sec:numerical_experiment} we presented some numerical experiments that suggest that our approach has a better performance than an earlier algorithm proposed by Renegar \cite{renegar2019accelerated}. 
We also compared against interior point methods. 
As expected, IPMs excel at getting accurate solutions but we found that our approach was often able to obtain close enough solutions with a fraction of the running time. 

Still, there are a few outstanding issues that we believe could be addressed in future works. In particular, there have been many interesting works regarding Frank-Wolfe method and variants, including nonconvex extensions \cite{LT13,BPT20,ZZLP21}. In particular, one of the most common ways to improve Frank-Wolfe methods is via the so-called \emph{away steps} and it could be interesting to try to port some techniques to our setting by making use of the geometric properties of hyperbolicity cones and their duals. 
The modifications described in \cite{FGM17}, for instance, could be promising, but they seem to require a deep knowledge of the facial structure of the underlying convex set, which may present a challenge for a set obtained by taking a compact slice of the {dual} of a hyperbolicity cone as we do in our approach.

Finally, a limitation of our approach is the requirement that a constant 
$c_D$ (see Assumption~\ref{asm:optimal_solution_bound}) is known. 
To conclude this paper, we offer some thoughts on this point.
\paragraph{When $c_D$ in Assumption~\ref{asm:optimal_solution_bound} is unknown}
Suppose that $\stdCone$ is a regular convex cone and $e \in \interior \stdCone$ is arbitrary.
Suppose also that \eqref{prob:dual} has an optimal solution.
Since  $\inProd{e}{y} > 0$ always holds for $ y \in \stdCone^*\setminus\{0\}$, for any such problem, there is always some $c_D$ for which Assumption~\ref{asm:optimal_solution_bound} is satisfied.
As we saw in Section~\ref{sec:example}, $c_D$ is readily available when minimizing a positive definite quadratic function. But 
suppose that we have a problem for which $c_D$ is not available.

Then, one can start with any $e \in \interior \stdCone$, an initial guess for $c_D$, say, $c_D^0 \coloneqq 1$ and run Algorithm~\ref{alg:proposed_method} with $c_D^0$ and $e$.
Let $\bar{x}^0$ and $\bar{y}^0$ denote the output of the algorithm.
As remarked in the discussion about the stopping criteria, if $c_D^0$ is large enough so that Assumption~\ref{asm:optimal_solution_bound} is satisfied, then $\bar{x}^0$ and $\bar{y}^0$  should be close to being zero duality gap pairs of optimal solutions to \eqref{prob:primal} and \eqref{prob:dual}.
On the other hand, if $c_D^0$ is too small, then either $\bar{x}^0$ is far from being feasible to \eqref{prob:primal} or the sum of the objective values associated to $\bar{x}^0$ and $\bar{y}^0$ has large absolute value or both phenomena happen at the same time. 
In that case, we may increase $c_D^0$ by, say, setting $c_D^{1} \coloneqq 2c_D^0$ and try Algorithm~\ref{alg:proposed_method} again with $c_D^1$ and $e$ in order to obtain new solutions $\bar{x}^1$ and $\bar{y}^1$.

The summary of this discussion is that, in theory, one could handle problems for which $c_D$ is unknown by repeatedly invoking Algorithm~\ref{alg:proposed_method} with increasingly larger guesses of $c_D$  and stopping when the obtained solutions are sufficiently close to being primal-dual optimal for \eqref{prob:primal} and \eqref{prob:dual}.
This, of course, would require a very careful calibration of the stopping criterion in Algorithm~\ref{alg:proposed_method}. Although we did not explore this possibility in this paper, this might be an interesting future direction to consider.

This also raises the question of the behavior of Algorithm~\ref{alg:proposed_method} if $c_D$ is large. 
As long as Assumptions~\ref{asm:optimal_solution_bound} and \ref{asm:a} are satisfied, the convergence results in Section~\ref{sec:convergence} still hold. 
However, the convergence rates in \eqref{eq:seqencial_convergence} and \eqref{eq:obj_vals_convergence} are obtained under the assumption that the step-size is chosen as to ensure that 
$h(y_k) - d^* = O(1/k)$ holds.
Hidden in the big O notation there are constants that depend on the \emph{diameter} of the feasible solution set of \eqref{prob:dual_compact}, e.g., see the proof of Theorem~6.1 in \cite{levitin1966constrained} or \cite[Theorem~1]{bomze2021frank}.
This suggests that, in theory, taking an overly large $c_D$ may have an adverse effect on the running time. 
However, for the problems discussed in Section~\ref{sec:exp_rel}, even increasing in $100$ times the value of $c_D$ used therein did not significantly degrade the performance, see more details in Appendix~\ref{app:cd}. This suggests that, in practice, 
the actual influence of the value of $c_D$ may be quite problem dependent.

\subsection*{Acknowledgements}
We thank the referees and the associate editor for their  comments, which helped to improve the paper.

\bibliographystyle{abbrv} 
\bibliography{bib.bib}

\appendix
\section{Omitted proofs}\label{app:proof}

\subsection{Proof of Lemma~\ref{lem:iso}}\label{app:lem_iso}
\begin{proof}
Let $z$ be such that $\dist(\lambda^{-1}(u),x) = \norm{z-x}_p$ and $\lambda(z) = u$. By the isometric property, there exists $y$ such that $\lambda(y) = \lambda(z) = u$ and $\lambda(y+x) = \lambda(y)+\lambda(x)$.
Simplifying the equality $\norm{y+x}_p^2 = \norm{\lambda(y+x)}_2^2 = \norm{\lambda(y)+\lambda(x)}_2^2$ leads to
\begin{equation}\label{eq:inner}
\inProd{x}{y}_p = \inProd{\lambda(x)}{\lambda(y)}_2 = \inProd{\lambda(x)}{\lambda(z)}_2.
\end{equation}
Recalling that $\inProd{z}{x}_p \leq \inProd{\lambda(z)}{\lambda(x)}_2$ holds (see \cite[Proposition~4.4]{BBEGG10}), we have
\begin{align*}
\dist(\lambda^{-1}(u),x)^2 & = \norm{z-x}_p^2 \\
&= \norm{z}_p^2 - 2\inProd{z}{x}_p + \norm{x}_p^2 \\
& \geq \norm{z}_p^2 - 2\inProd{\lambda(z)}{\lambda(x)}_2 + \norm{x}_p^2\\
& = \norm{y}_p^2 - 2\inProd{\lambda(y)}{\lambda(x)}_2 + \norm{x}_p^2,
\end{align*}
where the last equality follows from $\norm{y}_p^2 = \norm{\lambda(y)}_2^2 = \norm{\lambda(z)}_2^2 = \norm{z}_p^2$ and $\lambda(y) = \lambda(z)$. Then, in view of \eqref{eq:inner}, we 
obtain
\begin{equation*}
\dist(\lambda^{-1}(u),x)^2 \geq \norm{x-y}_p^2.
\end{equation*}
Since   $\lambda(y)=\lambda(z) = u$ holds we have in fact $\dist(\lambda^{-1}(u),x)^2 = \norm{x-y}_p^2$. Recalling \eqref{eq:inner}, this leads to 
\[
\dist(\lambda^{-1}(u),x)^2 = \norm{x-y}_p^2 = \norm{\lambda(y)}_2^2 - 2\inProd{\lambda(y)}{\lambda(x)}_2 + \norm{\lambda(x)}_2^2 = \norm{\lambda(x)-\lambda(y)}_2^2 = \dist(u,\lambda(x))^2.
\]
\end{proof}

\subsection{Proof of Proposition~\ref{prop:non-isometric}}\label{app:non_iso}
Let
\begin{equation}
  p(x) \coloneqq (x_1+x_2+x_3)(x_1-x_2+x_3)(2x_1-x_2-x_3)(x_1+2x_2-x_3).
\end{equation}
We start with the following lemma.

\begin{lemma}\label{lem:piso}
The polynomial $p$ is hyperbolic with respect to $e \coloneqq (0,0,1)$, but is not isometric.
\end{lemma}
\begin{proof}
The roots of $p(x-te)$ are
  \begin{alignat*}{2}
    &r_1(x) = x_1+x_2+x_3, &&r_2(x) = x_1-x_2+x_3, \\
    &r_3(x) = -2x_1+x_2+x_3, \:\:&&r_4(x) = -x_1-2x_2+x_3, 
  \end{alignat*}
which are all real for $x\in \Re^3$.
Since $p(e) > 0$, $p$ is hyperbolic along $e$.

To prove $p$ is not isometric, we show that, for $z = (3,1,0)$ and $y=(-1,0,0)$, there is no $w\in \mathbb{R}^3$
such that $\lambda(w) = \lambda(z)$ and $\lambda(w+y) = \lambda(w) + \lambda(y)$.
First we show that
\begin{equation*}
  \lambda(w) = \lambda(z) \Rightarrow w = z.
\end{equation*}
Let $w \in \Re^3$ be such that $\lambda(w) = \lambda(z)$ holds.
We start by observing that $\lambda(z) = (4,2,-5,-5)$, so $z$ has an eigenvalue of multiplicity two. Therefore, if $\lambda(w) = \lambda(z)$, at least two of $r_1(w),r_2(w),r_3(w),r_4(w)$ must be the same.
We consider all possible cases.
\begin{enumerate}[$(i)$]
  \item \fbox{$r_1(w)=r_2(w)$.} This case happens if and only if $w_2  = 0$.
  We have the following subcases.
  \begin{enumerate}
    \item If $w_1>0$, then $r_1(w) = r_2(w) > r_4(w) > r_3(w)$ holds, i.e., the two largest eigenvalues of $w$ are equal. Therefore, $\lambda(w)$ can not be $\lambda(z)$, because the two smallest components of $\lambda(z)$ are equal.
    \item If $w_1=0$, then $r_1(w) = r_2(w) = r_3(w) = r_4(w)$ holds. Similarly, $\lambda(w)$ can not be $\lambda(z)$.
    \item If $w_1<0$, then $r_3(w) > r_4(w) > r_2(w) = r_1(w)$ holds. Therefore,\begin{equation*}
      \lambda(w) = \lambda(z) \Rightarrow \left\lbrace\begin{array}{l}
        r_3(w) = 4\\
        r_4(w) = 2\\
        r_2(w) = r_1(w) = -5
      \end{array}\right.
    \end{equation*}
    However, there does not exist $w$ which satisfies these equalities.
  \end{enumerate}
  \item \fbox{$r_1(w)=r_3(w)$.} This case happens if and only if $w_1 = 0$. We have the following subcases.
  \begin{enumerate}
    \item If $w_2>0$, then $r_1(w) = r_3(w) > r_2(w) > r_4(w)$ holds. Therefore, $\lambda(w)$ can not be $\lambda(z)$.
    \item If $w_2=0$, then $r_1(w) = r_2(w) = r_3(w) = r_4(w)$ holds. Therefore, $\lambda(w)$ can not be $\lambda(z)$.
    \item If $w_2<0$, then $r_4(w) > r_2(w) > r_3(w) = r_1(w)$ holds. Therefore,
    \begin{equation*}
      \lambda(w) = \lambda(z) \Rightarrow \left\lbrace\begin{array}{l}
        r_4(w) = 4\\
        r_2(w) = 2\\
        r_1(w) = r_3(w) = -5
        \end{array}\right.
    \end{equation*}
    However, there does not exist $w$ which satisfies these equalities.
  \end{enumerate}
  \item \fbox{$r_1(w)=r_4(w)$.} This case happens if and only if $2w_1=-3w_2$. We have the following subcases.
  \begin{enumerate}
    \item If $w_1>0$, then $r_2(w) > r_1(w) = r_4(w) > r_3(w)$ holds. Therefore, $\lambda(w)$ can not be $\lambda(z)$.
    \item If $w_1=0$, then $r_1(w) = r_2(w) = r_3(w) = r_4(w)$ holds. Therefore, $\lambda(w)$ can not be $\lambda(z)$.
    \item If $w_1<0$, then $r_3(w) > r_4(w) = r_1(w) > r_2(w)$ holds. Therefore, $\lambda(w)$ can not be $\lambda(z)$.
  \end{enumerate}
  \item \fbox{$r_2(w)=r_3(w)$.} This case happens if and only if $3w_1= 2w_2$. We have the following subcases.
  \begin{enumerate}
    \item If $w_1>0$, then $r_1(w) > r_2(w) = r_3(w) > r_4(w)$ holds. Therefore, $\lambda(w)$ can not be $\lambda(z)$.
    \item If $w_1=0$, then $r_1(w) = r_2(w) = r_3(w) = r_4(w)$ holds. Therefore, $\lambda(w)$ can not be $\lambda(z)$.
    \item If $w_1<0$, then $r_4(w) > r_3(w) = r_2(w) > r_1(w)$ holds. Therefore, $\lambda(w)$ can not be $\lambda(z)$.
  \end{enumerate}
  \item \fbox{$r_2(w)=r_4(w)$.} This case happens if and only if $2w_1= -w_2$. We have the following subcases.
  \begin{enumerate}
    \item If $w_1>0$, then $r_2(w) = r_4(w) > r_1(w) > r_3(w)$ holds. Therefore, $\lambda(w)$ can not be $\lambda(z)$.
    \item If $w_1=0$, then $r_1(w) = r_2(w) = r_3(w) = r_4(w)$ holds. Therefore, $\lambda(w)$ can not be $\lambda(z)$.
    \item If $w_1<0$, then $r_3(w) > r_1(w) > r_2(w) = r_4(w)$ holds. Therefore,
    \begin{equation*}
    \lambda(w) = \lambda(z) \Rightarrow \left\lbrace\begin{array}{l}
      r_3(w) = 4\\
      r_1(w) = 2\\
      r_2(w) = r_4(w) = -5
    \end{array}\right.
  \end{equation*}
  However, there does not exist $w$ which satisfies these equalities. 
  \end{enumerate}
  \item \fbox{$r_3(w)=r_4(w)$.} This case happens if and only if $w_1 = 3w_2$. We have the following subcases.
  \begin{enumerate}
    \item If $w_1>0$, then $r_1(w) > r_2(w) > r_3(w) = r_4(w)$ holds. Therefore,
    \begin{equation*}
    \lambda(w) = \lambda(z) \Rightarrow \left\lbrace\begin{array}{l}
      r_1(w) = 4\\
      r_2(w) = 2\\
      r_3(w) = r_4(w) = -5
    \end{array}\right.
    \iff \left\lbrace\begin{array}{l}
      w_1=3\\
      w_2=1\\
      w_3=0
    \end{array}\right.
    \iff w=z
    \end{equation*}
    \item If $w_1=0$, then $r_1(w) = r_2(w) = r_3(w) = r_4(w)$ holds. Therefore, $\lambda(w)$ can not be $\lambda(z)$.
    \item If $w_1<0$, then $r_4(w) = r_3(w) > r_2(w) > r_1(w)$ holds. Therefore, $\lambda(w)$ can not be $\lambda(z)$.
  \end{enumerate}
\end{enumerate}
The summary of all the six cases and subcases above is that the sole possibility for $\lambda(w) = \lambda(z)$ is case $(vi).(a)$ where we have $w=z$. 
That is, 
\begin{equation*}
    \lambda(w) = \lambda(z) \Rightarrow w = z.
\end{equation*}
  Moreover, $(3,1,-3,-4) =\lambda(z+y)$ is different from  $\lambda(z) + \lambda(y) = (6, 3, -6, -6)$. Therefore, $p$ is not isometric.
\end{proof}

\begin{proof}[Proof of Proposition~\ref{prop:non-isometric}]
	By Lemma~\ref{lem:piso}, $p$ is hyperbolic along $e$ and is not isometric. It remains to check that  $p$ is minimal and the corresponding hyperbolicity cone is as described in the statement of the proposition.
	
		The roots of $p(x-te)$ are
	\begin{align*}
	r_1(x) = x_1+x_2+x_3, \qquad & r_2(x) = x_1-x_2+x_3, \\
	r_3(x) = -2x_1+x_2+x_3,\qquad & r_4(x) = -x_1-2x_2+x_3. 
	\end{align*}
	In order for $x$ to belong to $\Lambda(p,e)$ all the roots must be nonnegative. This gives the expression for $\Lambda(p,e)$ in the statement of the proposition.
	
	Next, let $q$ be a minimal hyperbolic polynomial for $\Lambda(p,e)$ so that 
	$q$ divides $p$ and $\Lambda(p,e) = \Lambda(q,e)$ holds. Since $p$ is a product of four degree $1$ polynomials, $q$  must be a product of \emph{some} of those 
	four polynomials. Suppose that $p$ is not of minimal degree.
	Then, $q$ cannot have degree $4$, so, up to a constant,  it must be a product of \emph{strictly 
		less than} four degree $1$ polynomials among the ones that appear in the decomposition of $p$.
	Therefore, in order to show that $p$ is of minimal degree, we only need to argue that removing \emph{any} of the polynomials that appear in the decomposition of 
	$p$ will result in a larger cone. We do this case by case.
	
	If $q$ omits the factor $x_1+x_2+x_3$, then $\Lambda(q,e)$ contains $(-1,-1,1) \not \in \Lambda(p,e)$.
	If $q$ omits $x_1-x_2+x_3$, then $\Lambda(q,e)$ contains $(-1,1,1) \not \in \Lambda(p,e)$.
	If $q$ omits $2x_1-x_2-x_3$, then $\Lambda(q,e)$ contains $(1,-1,1) \not \in \Lambda(p,e)$.
	Finally, if $q$ omits $x_1+2x_2-x_3 $, then $\Lambda(q,e)$ contains $(1,1,2) \not \in \Lambda(p,e)$.
	
	The conclusion is that in order for $q$ to be minimal it cannot omit any of the degree $1$ factors of $p$, 
	so $q$ must have degree $4$ and $p$ is a minimal degree polynomial as well.	
\end{proof}

\subsection{Proof of Lemma~\ref{lem:yk_opt}}\label{app:proof_lemyopt}
\begin{proof}
Suppose that \eqref{eq:yk_opt} fails.
Then, there exists  an $\epsilon>0$ and  a subsequence $\{y_{k_j}\}$ for which \[\dist(y_{k_j},\OPT{Y}) \geq \epsilon\]  holds for all $j$. 
Since the feasible region of \eqref{prob:dual_compact} is compact (Proposition~\ref{prop:equivalent_compact_problem}) and all the $y_k$ are feasible, 
there exists a subsequence of $\{y_{k_j}\}$ that converges to some $\bar{y}$ that is also feasible to \eqref{prob:dual_compact}. 
This leads to $\lim _{j \to \infty }h(y_{k_j}) = h(\bar{y})  = d^*$ (by Assumption~\ref{asm:a}) and 
$\dist(\bar{y},\OPT{Y}) \geq \epsilon$, which is  a contradiction.	
\end{proof}

\section{Increasing the value of $c_D$}\label{app:cd}
In Section~\ref{sec:numerical_experiment} the value of $c_D$ was always taken to be as in \eqref{eq:cd_projection}. 
Considering the setting of Table~\ref{table:hyper_high_deg} and of Section~\ref{sec:exp_rel}, we analyse the behavior of FW EleSym when the value of $c_D$ is increased for the case where $(n,k)  \in \{(10,1), (20,2)\}$. 

Letting $\hat c_D$ denote \eqref{eq:cd_projection} we repeated
the experiments described in Tables~\ref{table:hyper_hig_deg_10_1} and \ref{table:hyper_hig_deg_20_2} using  $\hat c_D$, $2 \hat c_D$, $4 \hat c_D$, $8 \hat c_D$, $16 \hat c_D$ and $100 \hat c_D$. The results are given in Table~\ref{table:hyper_cd}. 
The meaning of the columns is the same as in Tables~\ref{table:hyper_high_deg} and \ref{table:hyper_high_deg_det}.

These results suggest that for the problems considered in Section~\ref{sec:exp_rel}, the influence of the value of $c_D$ is relatively mild, as we do not see a significant decrease of performance for larger values. 
Until $E = 0.5\%$, as $c_D$ increases, we see that the average relative time is essentially the same  or increases slightly.
However, starting from $E = 0.1\%$, we start to see some counter-intuitive behaviour in that the average of relative times and iterations seem to decrease as $c_D$ increases.
For example, for $E = 0.05\%$, if we compare 
$\hat c_D$ and $100 \hat c_D$, the averages of the relative time and iterations seem to be better for the latter case in both sets of experiments.
We should be careful in drawing conclusions because for those cases, the standard deviation is also high.
Nevertheless, it seems interesting that for $E= 0.01\%$ the success rate for $100 \hat c_D$ is higher than the success rate for $\hat c_D$.
We are not certain about the cause of this phenomenon.
A guess is that this may be related to numerical precision issues after a large number of iterations and the computation of eigenvalues and conjugate vectors, which are operations that are very susceptive to numerical problems.
In particular, in the computation of conjugate vectors, it is necessary to compute the number of zero eigenvalues (i.e., the multiplicity) of certain elements, see Corollary~\ref{col:opt_sub_hyper}.
A larger value of $c_D$ may have a regularizing effect and may help to avoid making a wrong decision on whether an eigenvalue is numerically zero or not as this is implemented by checking whether its absolute value is smaller than a certain predefined threshold.

	\begin{table}\scriptsize

	\begin{subtable}{\textwidth}
	\begin{DIFnomarkup}
		\begin{filecontents*}{experiments/proj_hyper_cd_n10_d1_30_tol_high.csv}
error,cdiFWElit,cdiFWElitstd,cdiFWElabs,cdiFWElabsstd,cdiFWEl,cdiFWElstd,cdiFWElsuccess,cdiiFWElit,cdiiFWElitstd,cdiiFWElabs,cdiiFWElabsstd,cdiiFWEl,cdiiFWElstd,cdiiFWElsuccess,cdiiiFWElit,cdiiiFWElitstd,cdiiiFWElabs,cdiiiFWElabsstd,cdiiiFWEl,cdiiiFWElstd,cdiiiFWElsuccess,cdivFWElit,cdivFWElitstd,cdivFWElabs,cdivFWElabsstd,cdivFWEl,cdivFWElstd,cdivFWElsuccess,cdvFWElit,cdvFWElitstd,cdvFWElabs,cdvFWElabsstd,cdvFWEl,cdvFWElstd,cdvFWElsuccess,cdviFWElit,cdviFWElitstd,cdviFWElabs,cdviFWElabsstd,cdviFWEl,cdviFWElstd,cdviFWElsuccess
10,7.66667,2.92826,5.59e-03,3.66e-03,0.551841,0.366379,100,6.63333,1.80962,4.04e-03,1.16e-03,0.402753,0.12942,100,7.36667,2.23581,4.43e-03,1.46e-03,0.440668,0.157135,100,7.6,2.26822,4.53e-03,1.43e-03,0.451531,0.153287,100,8.2,2.32527,4.85e-03,1.43e-03,0.480699,0.151332,100,8.86667,2.55604,5.71e-03,2.26e-03,0.563293,0.214912,100
5,7.8,2.77178,5.64e-03,3.61e-03,0.557274,0.361629,100,7.33333,2.00574,4.37e-03,1.27e-03,0.435209,0.140354,100,8.53333,2.84948,5.02e-03,1.78e-03,0.496266,0.182723,100,9.1,2.70823,5.29e-03,1.66e-03,0.525825,0.182961,100,9.8,2.80885,5.69e-03,1.74e-03,0.565119,0.19217,100,10.9,3.2942,7.04e-03,3.63e-03,0.693047,0.340131,100
1,11.2333,5.8113,7.56e-03,5.30e-03,0.748555,0.539652,100,15.7333,8.23338,8.60e-03,4.64e-03,0.845367,0.450339,100,15.8667,7.34252,8.73e-03,4.24e-03,0.858708,0.41235,100,15.4667,6.1405,8.58e-03,3.58e-03,0.847665,0.361416,100,15.8333,6.60242,8.86e-03,3.87e-03,0.874163,0.387737,100,17.1333,7.1522,1.07e-02,6.70e-03,1.05069,0.628521,100
0.5,18.6,15.3322,1.12e-02,8.70e-03,1.16387,1.09769,100,21.9,13.1026,1.18e-02,7.26e-03,1.14919,0.680904,100,21.6667,13.2752,1.18e-02,7.61e-03,1.15253,0.720351,100,20.7,11.3689,1.14e-02,6.47e-03,1.12041,0.631189,100,21.6,11.3612,1.20e-02,6.51e-03,1.17697,0.644928,100,22.6333,11.4936,1.37e-02,8.82e-03,1.34854,0.84823,100
0.1,399.033,345.376,1.98e-01,1.69e-01,19.2759,15.7068,100,254.067,236.221,1.29e-01,1.19e-01,12.6541,11.3017,100,171.633,174.158,8.86e-02,8.96e-02,8.47534,7.88265,100,114.7,127.549,5.99e-02,6.61e-02,5.48459,5.31865,100,89.7,107.476,4.74e-02,5.59e-02,4.33511,4.49473,100,81.7667,109.227,4.55e-02,5.96e-02,4.10588,4.69738,100
0.05,1029.33,619.448,5.17e-01,2.99e-01,51.5576,28.2971,90,837.667,604.474,4.30e-01,3.15e-01,42.4698,29.8367,100,668.1,529.474,3.52e-01,2.82e-01,34.4901,25.3358,100,542.533,459.485,2.88e-01,2.48e-01,28.2,22.9298,100,449.2,402.563,2.39e-01,2.15e-01,23.009,19.4509,100,321.8,394.233,1.75e-01,2.16e-01,16.0958,18.5167,100
0.01,2140,0,8.79e-01,0.00e+00,87.0608,0,3.33333,1094,0,4.21e-01,0.00e+00,41.7158,0,3.33333,573,0,2.22e-01,0.00e+00,22.0225,0,3.33333,1040.5,990.657,5.61e-01,5.99e-01,48.7569,49.6195,6.66667,957.667,604.56,4.98e-01,2.96e-01,52.2047,34.3425,20,684.5,551.178,3.40e-01,2.73e-01,37.0375,31.362,13.3333
0.005,NaN,NaN,NaN,NaN,NaN,NaN,0,2218,0,8.52e-01,0.00e+00,84.4052,0,3.33333,1147,0,4.41e-01,0.00e+00,43.7081,0,3.33333,642,0,2.54e-01,0.00e+00,25.1677,0,3.33333,451,0,1.94e-01,0.00e+00,19.1766,0,3.33333,470,0,2.25e-01,0.00e+00,22.337,0,3.33333
0.001,NaN,NaN,NaN,NaN,NaN,NaN,0,NaN,NaN,NaN,NaN,NaN,NaN,0,NaN,NaN,NaN,NaN,NaN,NaN,0,NaN,NaN,NaN,NaN,NaN,NaN,0,1665,0,6.62e-01,0.00e+00,65.6128,0,3.33333,1295,0,6.19e-01,0.00e+00,61.3817,0,3.33333

\end{filecontents*}
\csvreader[hyperExpcd]{experiments/proj_hyper_cd_n10_d1_30_tol_high.csv}
		{}{\printExpDatacd}
	\end{DIFnomarkup}
		\caption{$n=10$, $k=1$. The columns under ``$\hat c_D$'' correspond to the same exact experimental setting as in the ``FW EleSym'' columns of Tables~\ref{table:hyper_hig_deg_10_1} and \hyperref[table:hyper_hig_deg_10_1_det]{\ref*{table:hyper_high_deg_det}\subref*{table:hyper_hig_deg_10_1_det}}, so the relative time  averages are similar but not the same as they correspond to a rerun of the experiments. }\label{table:hyper_high_cd}
	\end{subtable}
	\begin{subtable}{\textwidth}\begin{DIFnomarkup}
		\csvreader[hyperExpcdAlt]{experiments/proj_hyper_cd_n10_d1_30_tol_high.csv}
		{}{\printExpDatacdAlt}\end{DIFnomarkup}
		\caption{$n=10$, $k=1$.}\label{table:hyper_high_cdalt}
	\end{subtable}
	
	\begin{subtable}{\textwidth}\begin{DIFnomarkup}
		\begin{filecontents*}{experiments/proj_hyper_cd_n20_d2_30_tol_high.csv}
error,cdiFWElit,cdiFWElitstd,cdiFWElabs,cdiFWElabsstd,cdiFWEl,cdiFWElstd,cdiFWElsuccess,cdiiFWElit,cdiiFWElitstd,cdiiFWElabs,cdiiFWElabsstd,cdiiFWEl,cdiiFWElstd,cdiiFWElsuccess,cdiiiFWElit,cdiiiFWElitstd,cdiiiFWElabs,cdiiiFWElabsstd,cdiiiFWEl,cdiiiFWElstd,cdiiiFWElsuccess,cdivFWElit,cdivFWElitstd,cdivFWElabs,cdivFWElabsstd,cdivFWEl,cdivFWElstd,cdivFWElsuccess,cdvFWElit,cdvFWElitstd,cdvFWElabs,cdvFWElabsstd,cdvFWEl,cdvFWElstd,cdvFWElsuccess,cdviFWElit,cdviFWElitstd,cdviFWElabs,cdviFWElabsstd,cdviFWEl,cdviFWElstd,cdviFWElsuccess
10,12.3667,3.44897,2.76e-02,8.28e-03,0.179171,0.057225,100,12.1333,2.70036,2.62e-02,5.83e-03,0.17225,0.0475616,100,12.1,2.84484,2.60e-02,5.97e-03,0.171385,0.0491632,100,12.9333,3.24763,2.77e-02,7.11e-03,0.181387,0.0530573,100,14.2,3.33632,3.05e-02,7.52e-03,0.198644,0.0519448,100,15.3333,3.52658,3.30e-02,8.13e-03,0.215417,0.0581468,100
5,13.4667,3.29821,2.97e-02,7.81e-03,0.193598,0.0581974,100,13.5667,3.6073,2.90e-02,7.68e-03,0.189402,0.0547344,100,15.0667,4.17656,3.19e-02,8.67e-03,0.208059,0.0634695,100,16.5333,4.86177,3.48e-02,1.03e-02,0.227327,0.0717979,100,17.9667,5.1894,3.81e-02,1.13e-02,0.246477,0.0727657,100,18.5667,5.44365,3.96e-02,1.18e-02,0.258432,0.0814993,100
1,24.2,10.0735,5.06e-02,2.08e-02,0.33228,0.154127,100,31.8,13.6114,6.51e-02,2.80e-02,0.423457,0.189353,100,35,14.8672,7.15e-02,3.10e-02,0.466385,0.212032,100,37.3667,16.3801,7.73e-02,3.49e-02,0.502613,0.233308,100,39.1333,16.3996,8.15e-02,3.52e-02,0.526302,0.226411,100,41.7667,16.8476,8.80e-02,3.67e-02,0.567971,0.234405,100
0.5,35.5667,14.9751,7.25e-02,3.07e-02,0.477935,0.223575,100,46.1,19.1588,9.35e-02,4.02e-02,0.604989,0.267912,100,51.6,23.6696,1.05e-01,4.99e-02,0.684352,0.33846,100,54.7667,28.2644,1.13e-01,6.02e-02,0.739817,0.41307,100,56.1333,27.7436,1.17e-01,5.95e-02,0.760461,0.399083,100,60.0333,29.153,1.27e-01,6.35e-02,0.825877,0.425916,100
0.1,431.2,329.433,8.32e-01,6.04e-01,5.52239,4.16359,100,260.633,250.801,5.09e-01,4.90e-01,3.36637,3.23651,100,247.567,205.842,4.94e-01,4.05e-01,3.28347,2.72955,100,234.5,202.98,4.71e-01,4.00e-01,3.10386,2.62358,100,221.367,186.061,4.50e-01,3.72e-01,2.97436,2.50183,100,205,190.588,4.23e-01,3.87e-01,2.78716,2.574,100
0.05,1518.27,680.344,3.01e+00,1.34e+00,19.8483,9.74918,100,948.867,670.243,1.88e+00,1.35e+00,12.516,9.21401,100,807.067,660.927,1.62e+00,1.33e+00,10.844,9.04124,100,712.267,648.781,1.43e+00,1.31e+00,9.5545,8.68946,100,670.567,655.181,1.36e+00,1.33e+00,9.13608,8.88979,100,635.9,673.287,1.31e+00,1.39e+00,8.72122,9.18272,100
0.01,7802,2132.63,1.46e+01,6.96e+00,84.9784,18.6038,6.66667,7687.4,1693.22,1.40e+01,4.62e+00,81.2891,12.1446,16.6667,6713.33,2353.43,1.31e+01,5.89e+00,72.7548,19.7338,20,6539.89,2443.91,1.33e+01,5.61e+00,79.2891,20.4611,30,5745.4,1801.05,1.17e+01,4.11e+00,68.4727,12.7199,33.3333,5575.09,1810.47,1.15e+01,4.03e+00,67.8523,17.1406,36.6667
0.005,NaN,NaN,NaN,NaN,NaN,NaN,0,NaN,NaN,NaN,NaN,NaN,NaN,0,7535,0,1.25e+01,0.00e+00,92.7827,0,3.33333,5677,0,9.81e+00,0.00e+00,72.8737,0,3.33333,7134,0,1.30e+01,0.00e+00,96.9278,0,3.33333,5769,0,1.09e+01,0.00e+00,80.6394,0,3.33333
0.001,NaN,NaN,NaN,NaN,NaN,NaN,0,NaN,NaN,NaN,NaN,NaN,NaN,0,NaN,NaN,NaN,NaN,NaN,NaN,0,NaN,NaN,NaN,NaN,NaN,NaN,0,NaN,NaN,NaN,NaN,NaN,NaN,0,NaN,NaN,NaN,NaN,NaN,NaN,0

\end{filecontents*}
\csvreader[hyperExpcd]{experiments/proj_hyper_cd_n20_d2_30_tol_high.csv}
		{}{\printExpDatacd}\end{DIFnomarkup}
		\caption{$n=20$, $k=2$. The columns under ``$\hat c_D$'' correspond to the same exact experimental setting as in the ``FW EleSym'' columns of Tables~\ref{table:hyper_hig_deg_20_2} and \hyperref[table:hyper_hig_deg_20_2_det]{\ref*{table:hyper_high_deg_det}\subref*{table:hyper_hig_deg_20_2_det}}, so the relative time  averages are similar but not the same as they correspond to a rerun of the experiments. }\label{table:hyper_high_20_2_cd}
	\end{subtable}
	\begin{subtable}{\textwidth}\begin{DIFnomarkup}
		\csvreader[hyperExpcdAlt]{experiments/proj_hyper_cd_n20_d2_30_tol_high.csv}
		{}{\printExpDatacdAlt}\end{DIFnomarkup}
		\caption{$n=20$, $k=2$.}\label{table:hyper_high_20_2_cdalt}
	\end{subtable}

	\caption{This experiment examines the effect of increasing $c_D$. The meaning of the columns is as in Tables~\ref{table:hyper_high_deg} and \ref{table:hyper_high_deg_det}. }\label{table:hyper_cd}	
\end{table}

\end{document}